\documentclass[reqno,11pt]{amsart}

\usepackage{amsmath,amsfonts,amssymb,graphicx,amsthm,enumerate,url,tikz}
\usepackage[noadjust]{cite}
\usepackage{stmaryrd}

\usepackage{xifthen,xcolor,tikz,setspace}
\usetikzlibrary{decorations.pathmorphing,patterns,shapes}
\usepackage[noadjust]{cite}
\definecolor{refkey}{gray}{.75}
\definecolor{labelkey}{gray}{.5}

\usepackage[colorlinks=true]{hyperref}
\colorlet{DarkGreen}{green!70!black}

\numberwithin{equation}{section}

\renewcommand{\restriction}{\mathord{\upharpoonright}}
\renewcommand{\epsilon}{\varepsilon}
\newcommand{\given}{\;\big|\;}
\newcommand{\one}{\mathbf{1}}

\usepackage{color}
 \definecolor{refkey}{gray}{.5}
 \definecolor{labelkey}{gray}{.5}
\definecolor{light}{gray}{.9}

\newtheorem{maintheorem}{Theorem}

\newtheorem{theorem}{Theorem}[section]
\newtheorem*{theorem*}{Theorem}
\newtheorem{lemma}[theorem]{Lemma}
\newtheorem{claim}[theorem]{Claim}
\newtheorem{proposition}[theorem]{Proposition}

\newtheorem{corollary}[theorem]{Corollary}

\theoremstyle{definition}{

\newtheorem{definition}[theorem]{Definition}
\newtheorem*{definition*}{Definition}

\newtheorem{remark}[theorem]{Remark}
}

\renewcommand{\P}{\mathbb P}

\newcommand{\Z}{\mathbb Z}
\newcommand{\cA}{\ensuremath{\mathcal A}}
\newcommand{\cB}{\ensuremath{\mathcal B}}
\newcommand{\cC}{\ensuremath{\mathcal C}}

\newcommand{\cE}{\ensuremath{\mathcal E}}

\newcommand{\cH}{\ensuremath{\mathcal H}}

\newcommand{\llb }{\llbracket}
\newcommand{\rrb }{\rrbracket}
\newcommand{\anncircuit}{\mathcal{C}_o}

\renewcommand{\epsilon}{\varepsilon}

\newcommand{\tmix}{t_{\textsc{mix}}}

\newcommand{\gap}{\text{\tt{gap}}}
\newcommand{\tv}{{\textsc{tv}}}
\newcommand{\north}{{\textsc{n}}}
\newcommand{\south}{{\textsc{s}}}
\newcommand{\east}{{\textsc{e}}}
\newcommand{\west}{{\textsc{w}}}
\newcommand{\sd}{{\operatorname{sd}}}
\newcommand{\potts}{{\textsc{p}}}
\newcommand{\rc}{{\textsc{rc}}}
\newcommand{\bP}{{\mathbf{P}}}
\newcommand{\bE}{{\mathbf{E}}}

\makeatletter
\newcommand{\superimpose}[2]{%
  {\ooalign{$#1\@firstoftwo#2$\cr\hfil$#1\@secondoftwo#2$\hfil\cr}}}
\makeatother
\newcommand{\nconn}{\mathpalette\superimpose{{\longleftrightarrow}{\!\!\!\!\not}}}

\begin{document}

\title{Mixing times of critical 2D Potts models}

\author{Reza Gheissari}
\address{R.\ Gheissari\hfill\break
Courant Institute\\ New York University\\
251 Mercer Street\\ New York, NY 10012, USA.}
\email{reza@cims.nyu.edu}

\author{Eyal Lubetzky}
\address{E.\ Lubetzky\hfill\break
Courant Institute\\ New York University\\
251 Mercer Street\\ New York, NY 10012, USA.}
\email{eyal@courant.nyu.edu}

\begin{abstract}
We study dynamical aspects of the $q$-state Potts model on an $n\times n$ box at its critical $\beta_c(q)$. Heat-bath Glauber dynamics and cluster dynamics such as Swendsen--Wang (that circumvent low-temperature bottlenecks) are all expected to undergo ``critical slowdowns'' in the presence of periodic boundary conditions: the inverse spectral gap, which in the subcritical regime is $O(1)$, should at criticality be polynomial in $n$ for $1< q \leq 4$, and exponential in $n$ for $q>4$ in accordance with the predicted discontinuous phase transition.
This was confirmed for $q=2$ (the Ising model) by the second author and Sly, and for sufficiently large $q$ by Borgs~\emph{et~al.}

Here we show that the following holds for the critical Potts model on the torus:
 for $q=3$, the inverse gap of Glauber dynamics is $n^{O(1)}$; for $q=4$, it is at most $n^{O(\log n)}$; and for every $q>4$ in the phase-coexistence regime, the inverse gaps of both Glauber dynamics and Swendsen--Wang dynamics are exponential in $n$.

For free or monochromatic boundary conditions and large $q$, we  show that the dynamics at criticality is faster than on the torus
(unlike the Ising model where free/periodic boundary conditions induce similar dynamical behavior at all temperatures): the inverse gap of Swendsen--Wang dynamics is $\exp(n^{o(1)})$.
\end{abstract}
{\mbox{}
\vspace{-1.25cm}
\maketitle
}
\vspace{-0.9cm}

\section{Introduction}\label{sec:intro}

The $q$-state Potts model on a graph $G$ at inverse-temperature $\beta>0$ is the distribution $\mu_{G,\beta,q}$ over colorings of the vertices of $G$ with $q$ colors, in which the probability of a configuration $\sigma$ is proportional to $\exp[\beta \cH(\sigma)]$, with $\cH(\sigma)$ counting the number of pairs of adjacent vertices that have the same color (see~\S\ref{subsec:prelim-Potts}).
 Generalizing the Ising model (the case $q=2$), it is one of the most studied models in Mathematical Physics (cf.~\cite{Wu82}), with particular interest in its phase transition on $\Z^d$ ($d\geq 2$) at the critical $\beta=\beta_c$.

The random cluster (FK) model on a graph $G$ with parameters $0<p<1$ and $q>0$ is the distribution $\pi_{G,p,q}$ over
sets of edges of $G$, where the probability of a configuration~$\omega$ with $m$ edges and $k$ connected components is proportional to $[p/(1-p)]^m q^k$ (see~\S\ref{subsec:prelim-rc}). It generalizes percolation ($q=1$) and electrical networks/uniform-spanning-trees ($q\downarrow 0$), and corresponds at integer $q\geq 2$ to the Potts model via the Edwards--Sokal coupling; e.g., one may produce $\sigma\sim\mu_{G,\beta,q}$ by first sampling $\omega\sim\pi_{G,p,q}$ for $p=1-e^{-\beta}$, then assigning an i.i.d.\ color to the vertices of each connected vertex set of $\omega$. As such, extensively studied in its own right, the random cluster representation has been an important tool in the analysis of Ising and Potts models (see~\cite{Gr04} for further details).

On $\mathbb Z^2$ with $q\geq 1$, significant progress has been made in the study of these models and their rich behavior at the phase transition point $p_c=\frac{\sqrt{q}}{1+\sqrt{q}}$ (and $\beta_c=\log(1+\sqrt{q})$).
It is widely believed (see~\cite[Conj.~6.32 and~(6.33)]{Gr04}) that the phase transition would be continuous (second-order) if $1\leq q \leq 4$ and discontinuous (first-order) for $q>4$: the latter has been proved~\cite{KS82,LMMRS91,LMR86} for $q>24.78$ (see also~\cite[Thm.~6.35]{Gr04}) and supported by exact calculations~\cite{Baxter82} for all $q>4$; the former was very recently proved~\cite{DST15} through an analysis of crossing probabilities in rectangles under various boundary conditions.
Here we build on this recent work to study the \emph{dynamical} behavior of the critical planar Potts and FK models in the three regimes: $1< q<4$, the extremal $q=4$, and $q>4$.

Heat-bath Glauber dynamics is a local Markov chain, introduced in~\cite{Gl63}, that models the evolution of a spin system as well as provides a natural way of sampling from it.
For the Potts model, the dynamics updates each vertex via an i.i.d.\ rate-1 Poisson process, where its new value is sampled according to $\mu_{G,\beta,q}$ conditioned on the values of all other vertices (this dynamics for the FK model is similarly defined via single-bond updates).

Swendsen--Wang dynamics is a Markov chain on Potts configurations, introduced in~\cite{SW87}, aimed at overcoming bottlenecks in the energy landscape (thus providing a potentially faster sampler compared to Glauber dynamics) via global cluster flips:
the dynamics moves from a Potts configuration $\sigma$ to a compatible FK configuration $\omega$ via the Edwards--Sokal coupling, then to a new Potts configuration $\sigma'$ compatible with $\omega$.
Chayes--Machta dynamics~\cite{ChMa97} is a closely related Markov chain on FK configurations, analogous to Swendsen--Wang for integer $q$, yet defined for any real $q\geq1$ (see~\S\ref{subsec-prelim-dyn}).

The spectral gap of a discrete-time Markov chain, denoted $\gap$, is $1-\lambda$ where $\lambda$ is the largest nontrivial eigenvalue of the transition kernel, and for a continuous-time chain it is the gap in the spectrum of its generator.
It serves as an important gauge for the rate of convergence of the chain to equilibrium, as it governs its $L^2$-mixing time.
For the above mentioned dynamics on the Potts/FK models, the inverse spectral gap is expected to feature a well-documented phenomenon known as \emph{critical slowdown}~\cite{HoHa77,LaFo93};
in what follows we restrict our attention to $\Z^2$, though an analogous picture is expected in higher dimensions as well as on other geometries (see, e.g.,~\cite{LScritical} for further details).
Glauber dynamics for the Potts model on an $n\times n$ torus should have $\gap^{-1}$ transition from $O(1)$ at high temperature ($\beta<\beta_c$) to $\exp( c n)$ at low temperatures ($\beta>\beta_c$) through either a critical power-law  when $1 < q \leq 4$ or an order of $\exp(c n)$ when $q>4$ (in accordance with the first-order phase transition believed to occur at $q>4$).
Swendsen--Wang/Chayes--Machta dynamics should, by design, have $\gap^{-1}=O(1)$ both at high and low temperatures, yet should  also exhibit a critical slowdown at $\beta=\beta_c$.

While this picture for the Potts model has been essentially verified for Glauber dynamics for all $\beta<\beta_c$
and Swendsen--Wang for all $\beta\neq\beta_c$ (see~\S\ref{sec:related-work}), the case $\beta=\beta_c$ has largely evaded rigorous analysis, with two exceptions: for $q=2$, a polynomial upper bound on $\gap^{-1}$ of Glauber dynamics for the Ising model was given in~\cite{LScritical}; and for sufficiently large $q$, Borgs \emph{et al.}~\cite{BCFKTVV99} showed in 1999 that the Swendsen--Wang dynamics has $\gap^{-1} = \exp[n^{1-o(1)}]$ (thereafter improved to $\log \gap^{-1}\asymp  n$ in~\cite{BCT12}).

 Crucial to the analysis of the dynamics for $q=2$ were \emph{Russo--Seymour--Welsh} (RSW) estimates for the corresponding FK model---which state that on $n\times m$ rectangles with uniformly bounded aspect ratios and free boundary conditions, crossing probabilities are uniformly bounded away from $0$---obtained by~\cite{DHN11} using the discrete holomorphic observable framework of Smirnov~\cite{Sm10}.
The framework of~\cite{Sm10} is further applicable to the critical Potts model for $q=3$ (where the model is expected to have a conformally invariant scaling limit), and the above RSW-type estimates for the FK-Ising model have been recently extended by Duminil-Copin, Sidoravicius and Tassion~\cite{DST15} to this case; this allows one to similarly extend the dynamical analysis of~\cite{LScritical} to $q=3$. However, at $q=4$, these RSW estimates are no longer expected to hold, and instead crossing probabilities are believed to be highly sensitive to boundary conditions, thus resulting in a quasi-polynomial (rather than a polynomial) upper bound on mixing.

The following theorems demonstrate the change in the critical slowdown of the Potts and random cluster models on $(\Z/n\Z)^2$ between these different regimes of $q$.

\begin{figure}
\vspace{-0.3cm}
\centering
\includegraphics[width=0.3\textwidth]{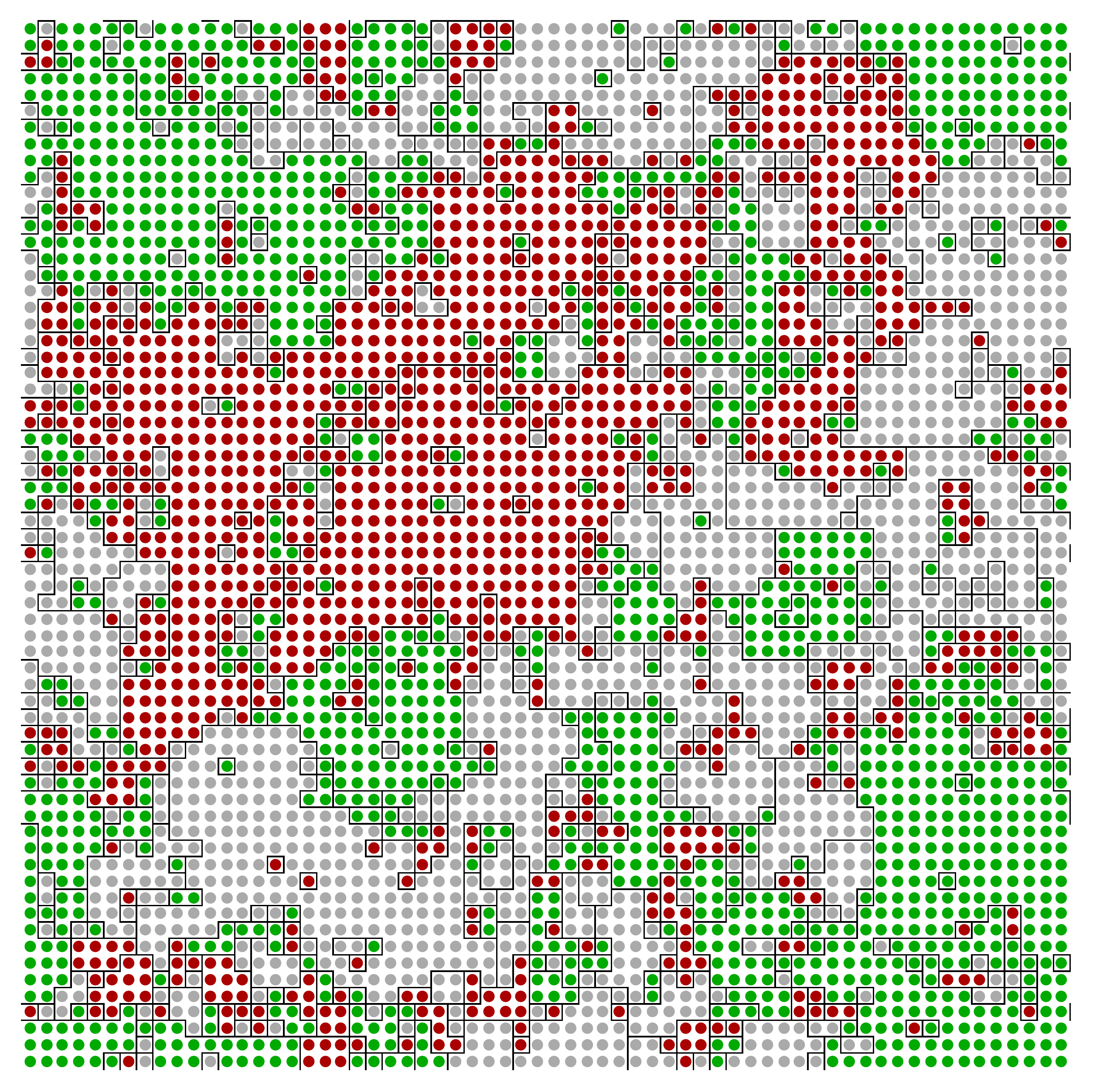}
\hspace{-0.1cm}\includegraphics[width=0.3\textwidth]{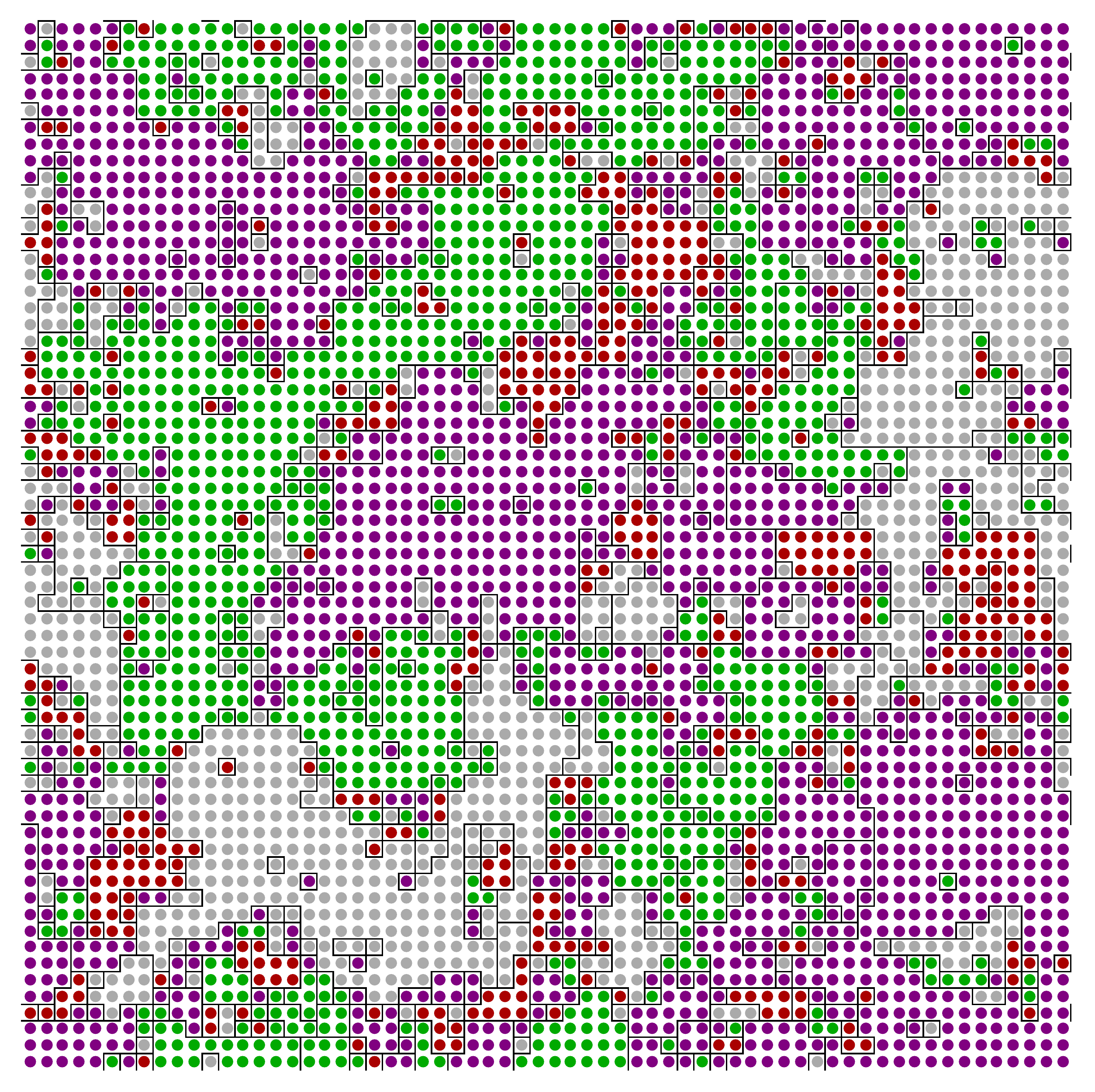}
\hspace{-0.1cm}\includegraphics[width=0.3\textwidth]{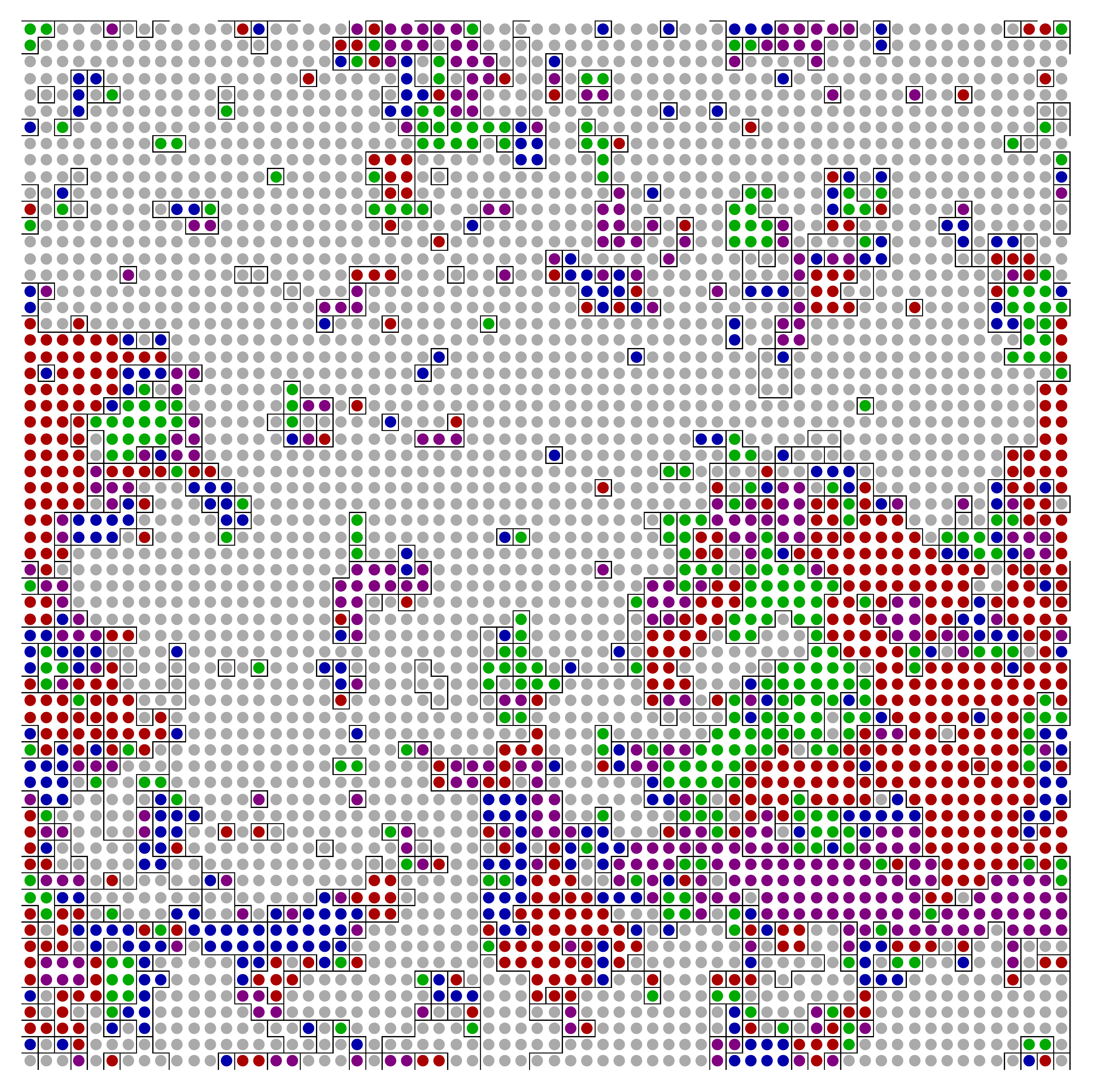}
\vspace{-0.35cm}
 \caption{Critical Potts configurations at $q=3,4,5$ on the torus $(\mathbb Z/n\mathbb Z)^2$ at $n=64$ with FK boundaries between different clusters.}
  \label{fig:potts_fk_q345}
\vspace{-0.2cm}
\end{figure}

\begin{maintheorem}\label{mainthm-q<=4}
There exist absolute $c_1,c_2>0$ so that Swendsen--Wang dynamics and Glauber dynamics on $(\mathbb Z/n\mathbb Z)^2$ satisfy the following:  for the $3$-state critical Potts model,
\begin{align}\label{eq:q=3}
\gap^{-1} \lesssim n^{c_1}\,,
\end{align}
whereas for the $4$-state critical Potts model,
\begin{align} \label{eq:q=4}
\gap^{-1} \lesssim n^{c_2 \log n}\,.
\end{align}
Eqs.~\eqref{eq:q=3} and~\eqref{eq:q=4} hold also for Glauber dynamics and Chayes--Machta dynamics for the critical FK model at $q=3$ and $q=4$, respectively.
\end{maintheorem}

\begin{maintheorem}\label{mainthm-q>4}
Let $q>4$ be such that the critical FK model on $\Z^2$ has two distinct Gibbs measures $ \pi^1_{\mathbb Z^2,q}\neq\pi^0_{\mathbb Z^2,q}$.
There exists $c=c(q)>0$ such that Swendsen--Wang dynamics and Glauber dynamics for the critical $q$-state Potts model on $(\mathbb Z/n\mathbb Z)^2$
satisfy
\begin{align}\label{eq:q>4}\gap^{-1}\gtrsim \exp(c n)\,.
\end{align}
The same holds for Glauber and Chayes--Machta dynamics for the critical FK model.
\end{maintheorem}

\begin{remark}
Since the initial posting of this paper, Duminil-Copin~\emph{et al.}~\cite{DGHMT16} proved the discontinuity of the FK phase transition for all $q>4$ on $\mathbb Z^2$; thus, the bound~\eqref{eq:q>4} from Theorem~\ref{mainthm-q>4} holds for the critical Potts and FK models on $(\mathbb Z/n\mathbb Z)^2$ for all $q>4$. 
\end{remark}

\begin{figure}
\vspace{-0.4cm}
\centering
  \begin{tikzpicture}
    \node (fig1) at (0,0) {
	\includegraphics[width=0.45\textwidth]{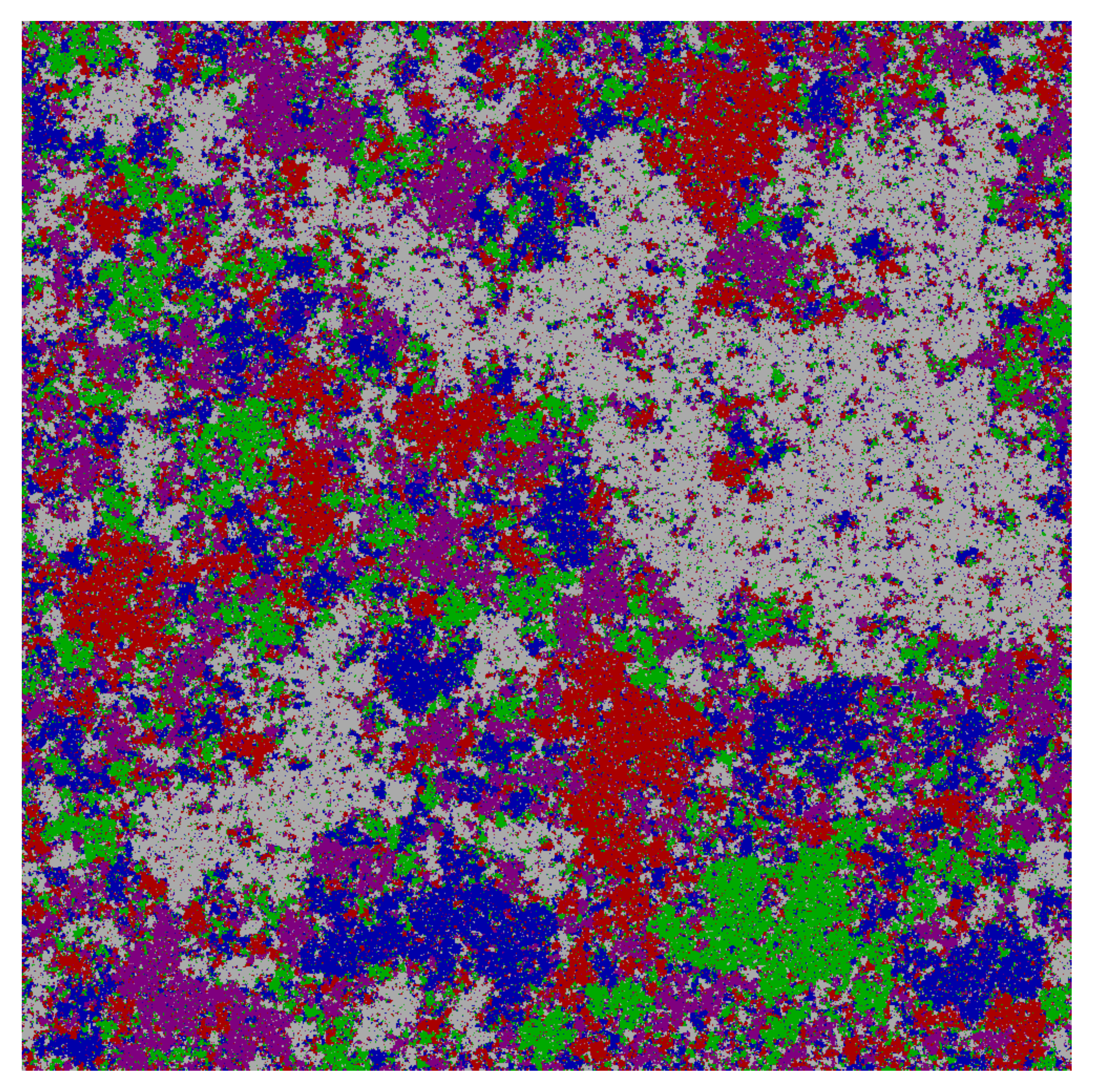}};
    \node (fig2) at (6.8cm,0) {
	\includegraphics[width=0.45\textwidth]{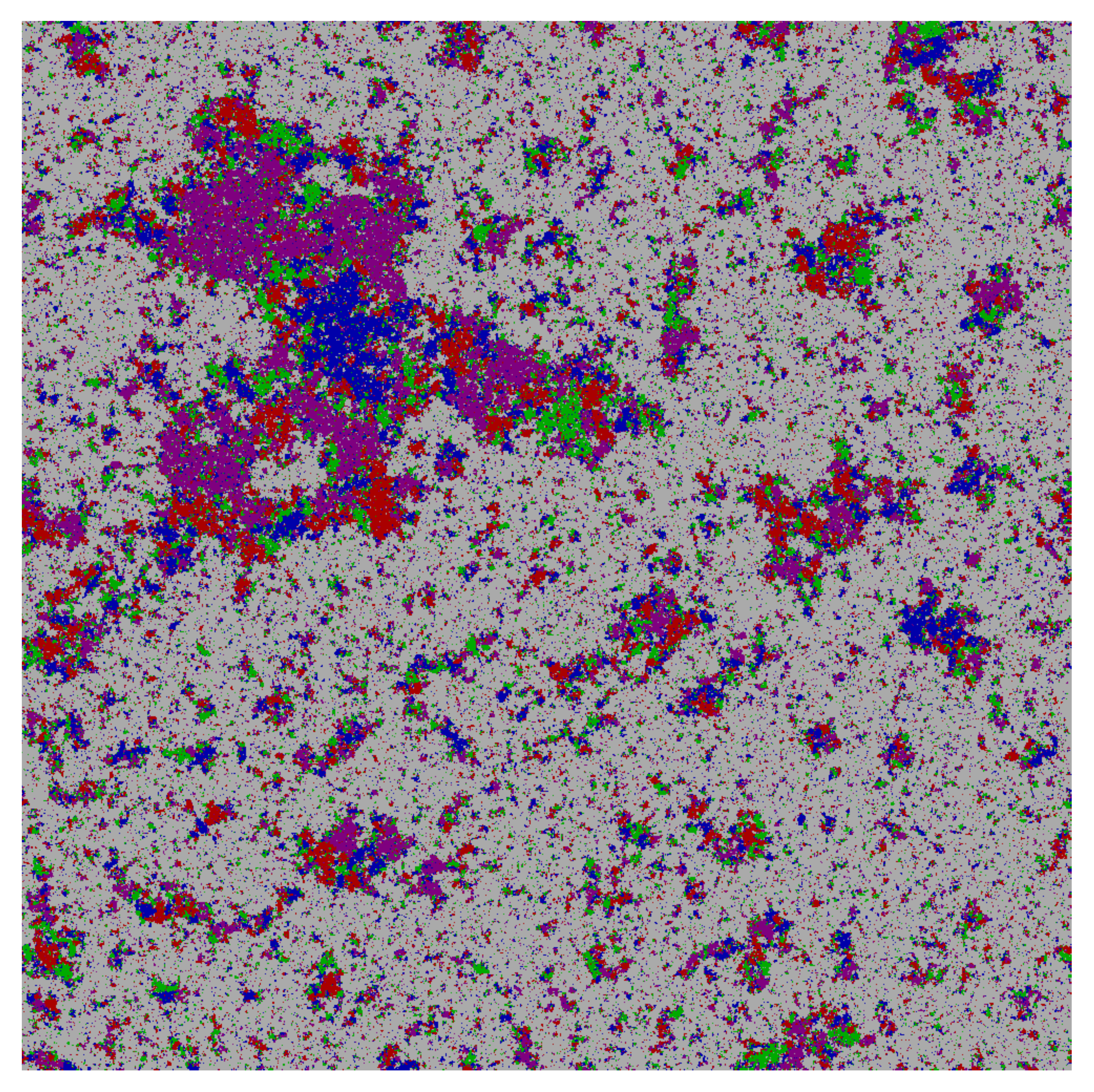}};
    \node (plot1) at (2.8cm,-5.6cm)
    {\includegraphics[width=0.5\textwidth]{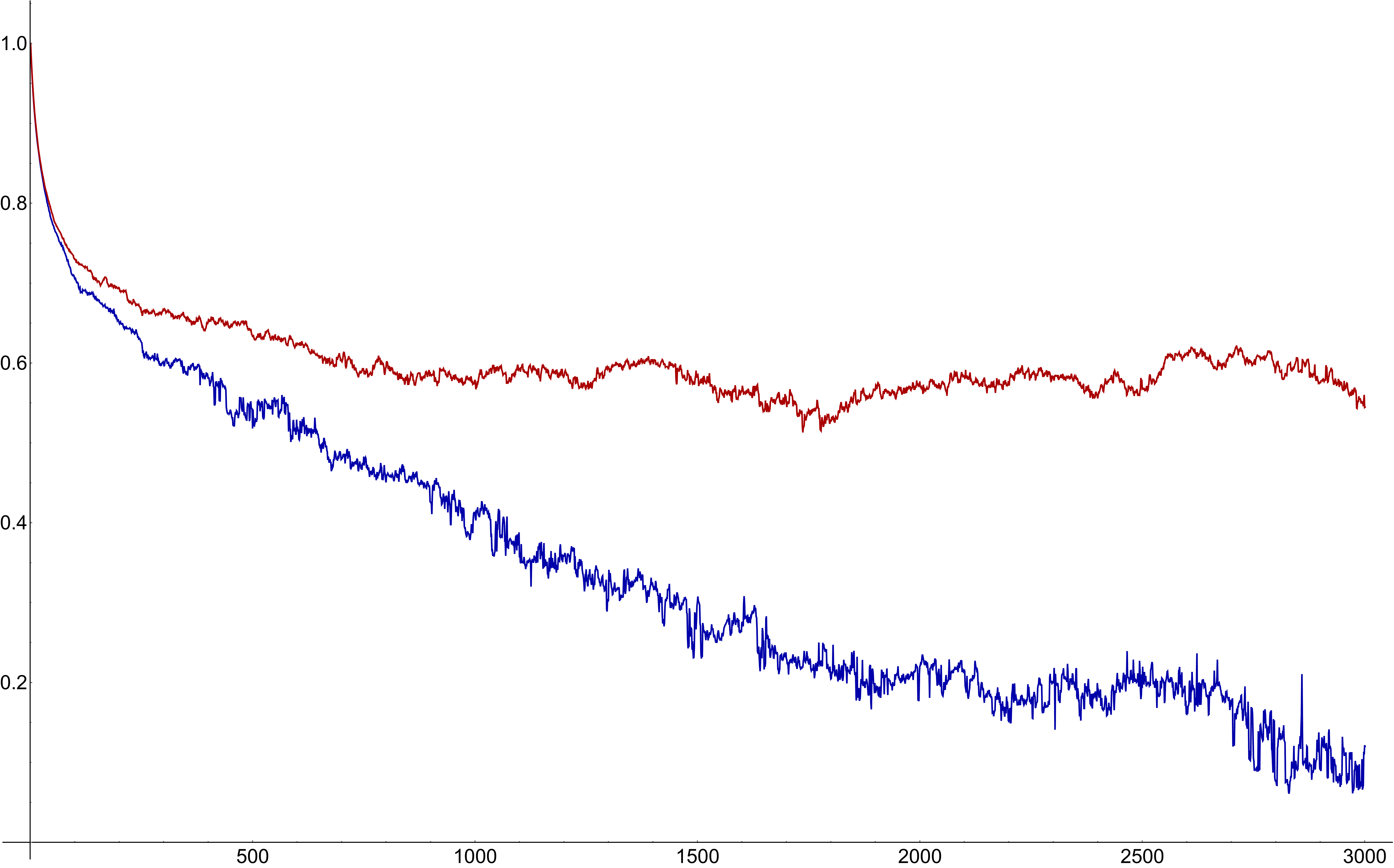}};
    \begin{scope}[shift={(plot1.south west)},x={(plot1.south east)},y={(plot1.north west)}]
      \node[font=\tiny,color=red] at (1.05,0.57) {periodic};
      \node[font=\tiny,color=blue] at (1.03,0.12) {free};
      \end{scope}
  \end{tikzpicture}
\vspace{-0.32cm}
 \caption{Swendsen--Wang for the critical 5-state Potts model on a $1000\times1000$ square (3000 iterations from a monochromatic initial state) under free (left) vs.\ periodic (right) boundary conditions; plot shows largest component (fractional) size in the intermediate FK configuration.}
  \label{fig:potts_sw}
\vspace{-0.3cm}
\end{figure}

Furthermore, the Glauber dynamics upper bounds  in Theorem~\ref{mainthm-q<=4} also hold for boxes with \emph{arbitrary} (as opposed to periodic) Potts boundary conditions (see Corollary~\ref{cor:poly/quasipoly}). In a companion paper \cite{GL16b}, for a wider class of boundary conditions, a matching upper bound to~\eqref{eq:q=4} is established for Glauber dynamics for the FK model at every $q\in (1,4]$. 
On the other hand, for $q>4$, one does not expect Swendsen--Wang and Glauber dynamics to be slow under \emph{every} boundary condition; e.g., monochromatic boundary conditions should destabilize all Gibbs states but one, inducing faster mixing, as in the case of the low temperature Ising model with plus boundary conditions (cf.~\cite[\S6]{Martinelli97}).

If we naively followed the intuition from the low temperature Ising model, free boundary conditions (where plus and minus phases are both metastable so $\gap^{-1}\gtrsim \exp(c n)$), might be expected to induce the same (slow) critical mixing behavior as in the torus. However, this is not case (see Fig.~\ref{fig:potts_sw}), as the following theorem demonstrates.

\begin{maintheorem}\label{mainthm:large-q-upper}
Fix $q$ large enough. The following holds for Swendsen--Wang dynamics for the critical $q$-state Potts model on an $n\times n$ box with free boundary conditions:
\begin{align}
\gap^{-1}\lesssim \exp\big(n^{o(1)}\big)\,. \label{eq:large-q-upper-1}
\end{align}
The same holds for Glauber and Chayes--Machta dynamics for the critical FK model.
\end{maintheorem}

The estimate~\eqref{eq:large-q-upper-1} holds also for monochromatic Potts boundary conditions (which correspond to wired FK boundary conditions), since it is a consequence of the analogous bound for the FK Glauber dynamics, where free boundary conditions at $p_c(q)$ are self-dual to wired boundary conditions. In fact, we establish~\eqref{eq:large-q-upper-1} for all FK boundary conditions sampled from the free or wired Gibbs measures (see Proposition~\ref{prop:all-free}), as well as ones that are free on three sides and wired on the fourth (Corollary~\ref{cor:3-sides}).

\begin{remark}\label{rem-tmix-metropolis}
By the well-known relation between $\gap$ and the total variation mixing time $\tmix$ (see~\S\ref{sec:mcmt}), the bounds in Theorems~\ref{mainthm-q<=4}--\ref{mainthm:large-q-upper} all hold with $\gap^{-1}$ replaced by $\tmix$.  Similarly, standard comparison estimates (see~\cite[Lemma~13.22]{LPW09}) extend our bounds for heat-bath Glauber dynamics to Metropolis (and other flavors of Glauber dynamics).
\end{remark}

\begin{figure}
  \begin{tikzpicture}
    \node (plot1) at (7,0) {};

    \node (plot2) at (0,0) {};

    \begin{scope}[shift={(plot1.south west)},x={(plot1.south east)},y={(plot1.north west)}, font=\small]

     \draw[color=DarkGreen, thick] (0.,5.3) arc (130:230:2) ;
     \draw[color=DarkGreen, thick] (0.,8.3) arc (130:230:2) ;
     \draw[color=DarkGreen, thick] (0.,11.3) arc (130:230:2) ;
     \draw[color=orange, thick] (20.05,1.95) arc (-50:50:2) ;
     \draw[color=orange, thick] (20.05,5.05) arc (-50:50:2) ;
     \draw[color=orange, thick] (20.05,8.15) arc (-50:50:2) ;
     \draw[color=purple, thick] (0,13.04) arc (148:212:11.6) ;

     \draw[draw=black,dashed] (0,7) rectangle (20,14);
     \draw[color=black,thick] (0,0) rectangle (20,14);

     \node (fig1) at (10.02,8.5) {
	\includegraphics[width=159.5pt]{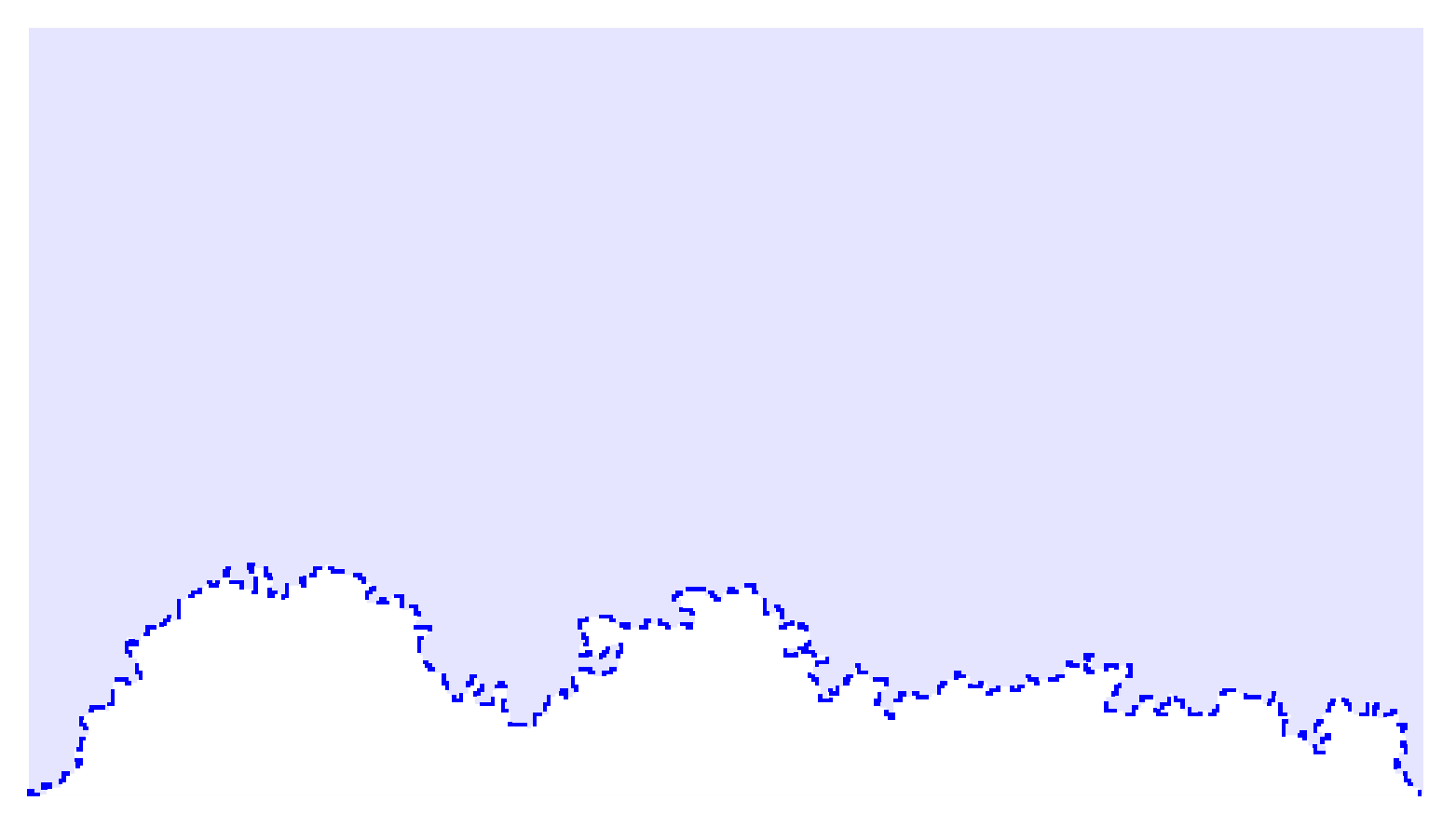}};

     \draw[color=purple,fill=purple] (0,.8) circle (0.05) node[above right] {};
     \draw[color=DarkGreen,fill=DarkGreen]  (0,2.26) circle (0.05) node[above right] {};
     \draw[color=DarkGreen,fill=DarkGreen]  (0,5.263) circle (0.05) node[above right] {};
     \draw[color=DarkGreen,fill=DarkGreen]  (0,8.263) circle (0.05) node[above right] {};
     \draw[color=DarkGreen,fill=DarkGreen]  (0,11.28) circle (0.05) node[above right] {};
     \draw[color=orange,fill=orange]  (20,1.94) circle (0.05) node[above right] {};
     \draw[color=orange,fill=orange]  (20,5.02) circle (0.05) node[above right] {};
     \draw[color=orange,fill=orange]  (20,8.13) circle (0.05) node[above right] {};
     \draw[color=orange,fill=orange]  (20,11.25) circle (0.05) node[above right] {};
    \draw[color=purple,fill=purple]  (0,12.98) circle (0.05) node[above right] {};

    \end{scope}

    \begin{scope}[shift={(plot2.south west)},x={(plot2.south east)},y={(plot2.north west)}, font=\small]

     \draw[color=purple, thick] (0.,5.3) arc (130:230:2) ;
     \draw[color=purple, thick] (0.,8.3) arc (130:230:2) ;
     \draw[color=purple, thick] (0.,11.3) arc (130:230:2) ;
     \draw[color=purple, thick] (20.05,1.95) arc (-50:50:2) ;
     \draw[color=purple, thick] (20.05,5.05) arc (-50:50:2) ;
     \draw[color=purple, thick] (20.05,8.15) arc (-50:50:2) ;
     \draw[color=purple, thick] (0,13.04) arc (148:212:11.6) ;

     \draw[draw=black,dashed] (0,7) rectangle (20,14);
     \draw[color=black,thick] (0,0) rectangle (20,14);

     \node (fig1) at (10.02,8.5) {
	\includegraphics[width=159.5pt]{fig_interface}};

     \node (fig2) at (10.01,1.25) {
	\includegraphics[width=159.5pt]{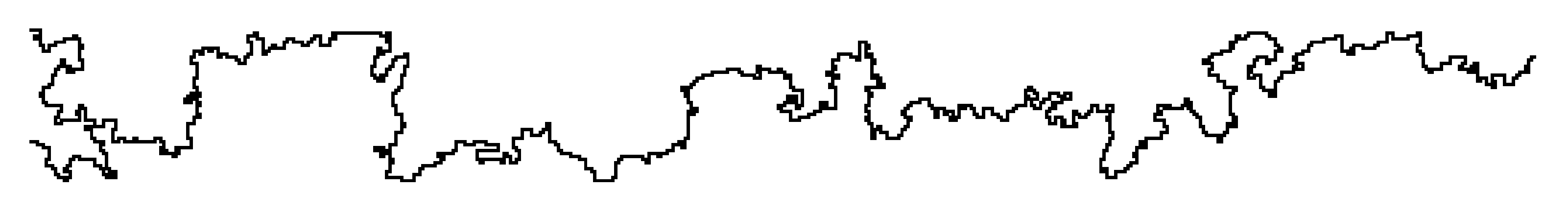}};

     \draw[color=purple,fill=purple] (0,.8) circle (0.05) node[above right] {};
     \draw[color=purple,fill=purple]  (0,2.26) circle (0.05) node[above right] {};
     \draw[color=purple,fill=purple]  (0,5.263) circle (0.05) node[above right] {};
     \draw[color=purple,fill=purple]  (0,8.263) circle (0.05) node[above right] {};
     \draw[color=purple,fill=purple]  (0,11.28) circle (0.05) node[above right] {};
     \draw[color=purple,fill=purple]  (20,1.94) circle (0.05) node[above right] {};
     \draw[color=purple,fill=purple]  (20,5.02) circle (0.05) node[above right] {};
     \draw[color=purple,fill=purple]  (20,8.13) circle (0.05) node[above right] {};
     \draw[color=purple,fill=purple]  (20,11.25) circle (0.05) node[above right] {};
    \draw[color=purple,fill=purple]  (0,12.98) circle (0.05) node[above right] {};
    \draw[color=black!50,decoration={bumps, segment length=0.185in, amplitude=0.02in}, decorate] (20,-0.04) -- (0, -0.04);

    \end{scope}

  \end{tikzpicture}
  \caption{Long-range connections in the boundary conditions stand in the way of coupling FK configurations beyond a boundary interface.}
  \label{fig:long-range-bc}
  \vspace{-0.1cm}
\end{figure}

Theorems~\ref{mainthm-q>4} and~\ref{mainthm:large-q-upper} show the similarities between the dynamical behavior of the Potts model at its critical point $\beta_c$ in the presence of a discontinuous phase transition, and the 2D Ising model in the low temperature regime $\beta>\beta_c$. The proof of Theorem~\ref{mainthm-q>4} is based on identifying a bottleneck involving the geometry of the torus, between the ordered and disordered phases in the critical FK model using only the multiplicity of Gibbs measures; this is akin to the energy barrier between the plus and minus phases in the low temperature Ising model. 

Moreover, through this similarity, our analysis of the Potts model at $\beta_c$ extends to its entire low temperature regime $\beta>\beta_c$,
where the slow mixing behavior of Glauber dynamics was shown for $q$ large enough in~\cite{BCFKTVV99,BCT12}. An adaptation of the proof of Theorem~\ref{mainthm-q>4} establishes this result for all $q>1$.
\begin{maintheorem}\label{mainthm-low-temp}
Consider the Potts model on $(\mathbb Z/n\mathbb Z)^2$ for any $q>1$ and $\beta>\beta_c(q)$. There exists $c=c(\beta,q)>0$ such that the Glauber dynamics has $\gap^{-1}\gtrsim \exp(cn)$.
\end{maintheorem}

The proof of Theorem~\ref{mainthm:large-q-upper} follows the approach used in~\cite{MaTo10} to establish sub-exponential upper bounds on $\tmix$ for Swendsen--Wang in the presence of all-plus boundary conditions, and involves adaptations of cluster-expansion techniques and the Wulff construction framework of~\cite{DKS} to the FK model. 
The absence of monotonicity in the Potts model frequently leads us to work directly with the FK representation. However, unlike the Ising model---where central to the upper bounds on mixing in many related works is the coupling of configurations beyond an interface between clusters (e.g., the interface between the plus and minus phases, used to establish the inductive step in the multi-scale argument of~\cite{Martinelli94})---the boundary conditions of the FK model may feature long-range connections between vertices. Using these as a ``bridge'' over the interface (see Figure~\ref{fig:long-range-bc}), different FK configurations below the interface may induce different distributions above it, thus preventing the coupling. Working around obstacles of this type comprises a significant part of the proof of Theorem~\ref{mainthm:large-q-upper}.

\subsection{Related work}\label{sec:related-work}

The critical slowdown picture of Glauber dynamics for the 2D Ising model
is by now fairly well understood.
For $\beta<\beta_c$, the dynamics on an $n\times n$ torus has $\gap^{-1}=O(1)$ via the work of Martinelli and Olivieri~\cite{MaOl94a,MaOl94b} and Martinelli, Olivieri and Schonmann~\cite{MOS94}, showing that, in this regime, 
there is a uniform bound on the inverse gap (in fact under arbitrary boundary conditions; see~\cite[\S3.7]{Martinelli97}). That this dynamics has $\gap^{-1}\gtrsim \exp(c_\beta n)$ at any $\beta>\beta_c$ for some $c_\beta>0$ was shown by Chayes, Chayes and Schonmann~\cite{CCS87}, and thereafter with the sharp $c_\beta$ 
by Cesi \emph{et al.}~\cite{CGMS96}. Finally, a polynomial upper bound on $\gap^{-1}$ at $\beta=\beta_c$ was given in the aforementioned paper~\cite{LScritical}; establishing the correct \emph{dynamical critical exponent} (believed to be universal and approximately $2.17$; cf.~\cite{LScritical} and its references) remains a challenging open problem.

As for Swendsen--Wang, comparison estimates due to Ullrich~\cite{Ul13,Ul14} imply that its inverse gap is at most that of Glauber dynamics on any graph and at any temperature (see Theorem~\ref{thm:Ullrich-comparison}); thus for $q=2$ on $\mathbb Z^2$ it also has $\gap^{-1}=O(1)$ for all $\beta<\beta_c$ and for all $\beta>\beta_c$ thanks to duality, and similarly at $\beta=\beta_c$ it has $\gap^{-1}=n^{O(1)}$.

For all other $q> 1$, Glauber dynamics for the Potts model on $(\Z/n\Z)^2$ is again known to have $\gap^{-1}=O(1)$ for all $\beta<\beta_c$ by combining the following results: Alexander~\cite{Al98} related exponential decay of connection probabilities in the FK model on $\Z^2$ to an analogous spatial mixing property in the Potts model on a finite box; Beffara and Duminil-Copin~\cite{BeDu12} proved the exponential decay of correlations in the FK model for all $\beta<\beta_c$; and the works of Martinelli \emph{et al.}~\cite{MaOl94a,MaOl94b,MOS94} translate the aforementioned spatial mixing property to 
 an $O(1)$ bound on the inverse gap. In contrast, Potts Glauber dynamics on $(\mathbb Z/n\mathbb Z)^2$ is always expected to be exponentially slow for $\beta>\beta_c$: as mentioned before, this is known for $q=2$, and was proved for  large enough $q$ in~\cite{BCFKTVV99,BCT12}.

Using the above mentioned estimates for high temperatures, comparison estimates, and duality,
Swendsen--Wang dynamics for the Potts model for any $q> 1$ also has $\gap^{-1}=O(1)$ for all $\beta\neq \beta_c$. Blanca and Sinclair~\cite{BlSi15} recently showed that for any $q>1$
both  Chayes--Machta dynamics and (heat-bath) Glauber dynamics for the FK model have $\tmix = O(\log n)$ for all $p\neq p_c$ (enjoying duality, the latter mixes rapidly at $p>p_c$ unlike for the Potts model). That $\tmix$ should at the critical $p=p_c$ be polynomial in $n$ for $1 < q\leq 4$ and exponential in it for every $q>4$ was left in~\cite{BlSi15} as an open question.
(See also Li and Sokal~\cite{LiSo91}; there, a polynomial \emph{lower bound} on the mixing of Swendsen--Wang and Glauber dynamics was given in terms of the \emph{specific heat}---a physical quantity which itself is not rigorously known. In~\S\ref{sec:mixing-q-leq-4} (Theorem~\ref{thm:poly-lower-bound}) we give a rigorous polynomial lower bound for $\gap^{-1}$ of the Potts Glauber dynamics.)

In the only two cases so far where the dynamical critical behavior on $(\Z/n\Z)^2$ has been addressed---the case $q=2$ in~\cite{LScritical} and the case of integer $q$ large enough in~\cite{BCFKTVV99,BCT12}---through the comparison inequalities of Ullrich, the results apply to all Markov chains discussed above (each has $\tmix \lesssim n^c$ at $q=2$ and $\tmix\gtrsim\exp(-cn)$ at $q$ large enough).
Note that the results of~\cite{BCFKTVV99,BCT12} are applicable to every dimension $d\geq 2$, while requiring that $q$ be sufficiently large as a function of $d$.

The dynamics for critical 2D Potts/FK models under \emph{free} boundary conditions takes after Glauber dynamics for the low temperature Ising model under plus boundary conditions. Improving on the original work of Martinelli~\cite{Martinelli94},
a delicate multi-scale analysis due to Martinelli and Toninelli~\cite{MaTo10}, based on censoring inequalities (see~\S\ref{subsec-prelim-dyn}), yielded an 
upper bound of $\exp(n^{o(1)})$ for the Ising model with plus boundary conditions. This was followed by an $n^{O(\log n)}$ bound in~\cite{LMST12} via this approach, extended to all $\beta>\beta_c$. Our proof of~\eqref{eq:large-q-upper-1} is based on this method.

Finally, detailed results are known on the dynamical behavior of Potts/FK models on the \emph{complete graph} (mean-field); see, e.g.,~\cite{BlSi14,GoJe99,CDL12,GSV15} and the references therein.

\section{Preliminaries}\label{sec:preliminaries}

In what follows we review the model definitions and properties, as well as the tools that will be used in our analysis. For a more detailed survey of the random cluster model, see~\cite{Gr04}. For more details on Markov chain mixing times and Glauber dynamics see~\cite{LPW09} and~\cite{Martinelli97}, respectively.
Throughout this paper, we use the notation $ f \lesssim g$ for two sequences $f(n),g(n)$ to denote $f= O(g)$, and let $f \asymp g$ denote $f \lesssim g \lesssim f$.

\subsection{Potts model}\label{subsec:prelim-Potts}

 The (ferromagnetic) $q$-state Potts model on a graph $G=(V,E)$ is the probability distribution over
 configurations $\sigma\in\Omega_{\textsc p}=[q]^{V}$ (viewed as assignments of colors out of $[q]=\{1,...,q\}$ to the vertices of $G$)
in which the probability of $\sigma$ w.r.t.\ the inverse-temperature $\beta>0$ and the boundary conditions $\zeta$ (an assignment of colors in $[q]$ to the vertices of some subgraph $H\subset G$) is given by
\[\mu^\zeta_{G,\beta,q}(\sigma)=\frac 1{\mathcal Z_{\textsc p}} \boldsymbol 1\{\sigma \restriction_H = \zeta\} \exp\Big(\beta \sum_{u\sim v}\one\{ \sigma(u)=\sigma(v)\}\Big)\,,
\]
where the sum is over unordered pairs of adjacent vertices $\{u,v\}$ in $V(G)$, and the normalizing constant $\mathcal{Z}_{\textsc p}$ is the partition function.

Throughout the paper, we consider graphs that are rectangular subsets of $\mathbb Z^2$ with nearest neighbor edges and vertex set
\[\Lambda_{n,n'}:=\llb  0,n\rrb  \times \llb  0,n'\rrb =\{k\in \mathbb Z:0\leq k\leq n\}\times\{k\in \mathbb Z:0\leq  k\le n'\}\,,
\]
where $n'=\lfloor \alpha n \rfloor$ for some fixed \emph{aspect ratio} $0<\alpha \le 1$, and the notation
$\llb a,b\rrb$ stands for $\{k\in\mathbb Z:a\leq k\leq b\}$.
 We use the abbreviated form $\Lambda$ when $n$ and $\alpha$ are made clear from the context.
For general subsets $S\subset \mathbb Z^2$, the boundary $\partial S$ will be the set of vertices in $S$ with a neighbor in $\mathbb Z^2-S$ and its edge set will be all edges in $\mathbb Z^2$ between vertices in $\partial S$; we set the interior $S^0=S-\partial S$. When considering rectangles $\Lambda$, denote the southern (bottom) boundary of $\Lambda$
by $\partial_\south \Lambda:=\llb 0,n\rrb \times \{0\}$, define $\partial_\north$, $\partial_\west$ and $\partial_\east$ analogously, and let multiple subscripts denote their union, i.e., $\partial_{\east,\west} \Lambda=\partial_\east \Lambda\cup \partial_\west \Lambda$.

\subsection{Random cluster (FK) models}\label{subsec:prelim-rc}

For a graph $G=(V,E)$, a random cluster (FK) configuration $\omega \in \Omega_{\rc}=\{0,1\}^{E}$ assigns binary values to the edges of $G$, either \emph{open} ($1$) or \emph{closed} ($0$). (In the context of boundary conditions, these are often referred to instead as wired and free, respectively). A \emph{cluster} is a maximal connected subset of vertices that are connected by open bonds, where singletons count as individual clusters. 

For a subset $H\subset V(G)$, we define FK boundary conditions $\xi$ as follows: first augment the graph to $G'$ by adding edges between any two vertices in $H$ not already connected by an edge; then if the boundary subgraph of $G'$ has vertex set $H$ and edge set $E(H)$ consisting of all edges between vertices in $H$, $\xi$ is an FK configuration in $\{0,1\}^{E(H)}$. A boundary condition $\xi$ can be identified with a partition of $H$ given by the clusters of $\xi$.

The FK model is the probability distribution over FK configurations on the remaining edge set $E(G)-E(H)$, where the probability of $\omega$ under the boundary conditions $\xi$ and parameters $p\in [ 0,1] $, $q>0$ is

\[\pi^\xi_{G,p,q}(\omega)=\frac 1{\mathcal Z_{\rc}} p^{o(\omega)}(1-p)^{c(\omega)}q^{k(\omega)}\,,
\]
where $o(\omega)$, $c(\omega)$, and $k(\omega)$ are the number of open bonds, closed bonds and clusters in $\omega$, respectively, with the number of clusters being computed using connections from $\xi$ as well as $\omega$. The partition function $\mathcal Z_{\rc}$ is again the proper normalizing constant.

Infinite volume Gibbs measures may be found by taking limits of increasing rectangles $\Lambda_n$ under a specified sequence $\xi=\xi(n)$ of boundary conditions on $\partial \Lambda_n$, where the important cases of all-wired and all-free boundary conditions are denoted by $1$ and $0$ respectively; let $\pi^\xi_{\mathbb Z^2}$ denote the weak limit (if it exists) of $\pi^\xi_{\Lambda_n}$ as $n\to\infty$.

\subsubsection*{Edwards--Sokal Coupling}

The Edwards--Sokal coupling~\cite{EdSo88} provides a way to move back and forth between the Potts model and the random cluster model on a given graph $G$ for $q\in \{2,3,\ldots\}$. The joint probability assigned by this coupling to $(\sigma,\omega)$, where $\sigma\in\Omega_\potts$ is a $q$-state Potts configuration at inverse-temperature $\beta>0$ and $\omega\in\Omega_\rc$ is an FK configuration with parameters $(p=1-e^{-\beta},q)$, is proportional to
\[\prod_{xy\in E(G)}\Big[(1-p)\one\{\omega(xy)=0\} + p\one\{\omega(xy)=1,\sigma(x)=\sigma(y)\}\Big]\,.\]
It follows that, starting from a Potts configuration $\sigma\sim \mu_{G,\beta,q}$, one can sample an FK configuration $\omega\sim\pi_{G,p,q}$ by letting $\omega(e)=1$ ($e\in\omega$) with probability $p=1-e^{-\beta}$ if the endpoints $x,y$ of the edge $e$ have $\sigma(x)=\sigma(y)$, and $\omega(e)=0$ ($e\notin \omega$) otherwise. Conversely, from $\omega\sim\pi_{G,p,q}$, one obtains $\sigma\sim \mu_{G,\beta,q}$ by assigning an i.i.d.\ color in $[q]$ to each cluster of $\omega$ (i.e., $\sigma(x)$ assumes that color for every vertex $x$ of that cluster).

In the presence of boundary conditions $\zeta$ for the Potts model, it is possible to sample $\sigma\sim \mu^\zeta_{G,\beta,q}$ using the random cluster model as follows. Associate to $\zeta$ the FK boundary conditions $\xi$ that wire two boundary sites $x,y$ to each other if and only if $\zeta(x)=\zeta(y)$. Further denote by $\mathcal E_\zeta$ the random cluster event that no two boundary sites $x,y$ with $\zeta(x)\neq \zeta(y)$ are connected via $\omega$ in $G$.
Then one can sample a configuration of $\mu^\zeta_{\Lambda,\beta,q}$ by first sampling $\omega\sim\pi^\xi_{\Lambda,p,q}(\cdot \mid \mathcal E_\zeta)$ for $p=1-e^{-\beta}$, then coloring the boundary clusters as specified by $\zeta$, and coloring every other cluster by an i.i.d.\ color uniformly over $[q]$.
For further details, see~\cite{LScritical}, where $\mathcal E_\zeta$ was introduced in the context of the Ising model.

\subsubsection*{Planar duality} On $\mathbb Z^2$, a configuration $\omega$ is uniquely identified with a configuration $\omega^\ast$ on the dual graph $\mathbb Z^2+(\frac12,\frac12)$ as follows: for every primal edge $e$ and its dual edge $e^\ast$ (intersecting at their center points), $\omega^\ast(e^\ast)=1$ if and only if $\omega(e)=0$.

For every $q\geq1$, the involution $p\mapsto p^\ast$ given by $pp^*=q(1-p)(1-p^*)$,
whose fixed point is the self-dual point $p_{\sd}=\frac {\sqrt q}{1+\sqrt q}$,
 satisfies
\[\pi^\xi_{\mathbb Z^2,p,q}\overset {d}=\pi^{\xi^\ast}_{(\mathbb Z^2)^\ast,p^\ast,q}\,,
\]
where the boundary conditions $\xi^\ast$ are determined on a case by case basis, but it is important to note that free and wired boundary conditions are dual to one another.
It is known~\cite{BeDu12} that on $\mathbb Z^2$, for all $q\ge 1$ one has $p_c(q)=p_{\sd}(q)$.
Throughout the paper, unless otherwise specified, let $p=p_{c}(q)$ and $\beta=\beta_c(q)$ (so $p_c=1-e^{-\beta_c}$), omitting these from the notations, as well as $q$ wherever it is clear from the context.

For two vertices $x,y\in V$, denote by $x\longleftrightarrow y$ the event that $x$ and $y$ belong to the same cluster of $\omega$.
In the context of a subgraph $S\subset G$, write $x\stackrel{S}\longleftrightarrow y$ to denote that $x$ and $y$ belong to the same cluster of $\omega\restriction_{E(S-\partial S)}$.
Refer to \[\cC_v(R):=\bigcup\Big\{x\stackrel{R}\longleftrightarrow y : x\in\partial_\south R\,,\,y\in\partial_\north R\Big\}\]
as a \emph{vertical crossing} of a rectangle $R$, and denote the analogously defined \emph{horizontal crossing} of the rectangle $R$ by
$\cC_h(R)$.

Consider a subset of $\mathbb Z^2$ of the form $A=R_2-R_1$ where $R_1,R_2$ are rectangular subsets of $\mathbb Z^2$ with $R_1\subsetneq R_2$. Call such domains annuli, and define open circuits as paths of nontrivial homology in $A-\partial A$ connecting a vertex $x$ to itself. Denote the existence of an open circuit in the annulus $A$ by $\anncircuit(A)$.

Finally, we add the $\ast$-symbol to the above crossing events to refer to the analogous dual-crossings (occurring in the configuration $\omega^\ast$ and the appropriate dual subgraphs).

\subsubsection*{FKG inequality, monotonicity and the Domain Markov property}

An event in the FK model is increasing if it is closed under addition of (open) edges, and decreasing if it is closed under removal of edges. For $q\geq 1$, the model enjoys the FKG inequality~\cite{FKG71}:
\[\pi^\eta_G(A\cap B)\geq \pi^\eta_G(A)\pi^\eta_G(B)\qquad\mbox{for every increasing events $A,B$}\,.
\]
Consequently, the model for $q\ge 1$ is \emph{monotone in boundary conditions}: for every boundary conditions $\eta\geq \xi$ (w.r.t.\ the partial ordering of configurations), $\pi^\eta_G\succeq \pi^\xi_G$, that is, $\pi^\eta_G(A)\ge \pi^\xi_G (A)$ holds for every increasing event $A$.

The \emph{Domain Markov} property of the FK model states that, on any graph $G$ with boundary conditions $\xi$, for every subgraph $G'\subset G$ with boundary conditions $\eta$ that are compatible with $\xi$,
\[\pi^\xi_G\left (\omega\restriction_{G'}\in\cdot \mid \omega\restriction _{G-G'}=\eta\right)=\pi^\eta_{G'}\,.
\]

\subsubsection*{FK phase transition and Russo--Seymour--Welsh (RSW) estimates}

The FK model at fixed $q\geq 1$ undergoes a phase transition at $p_c(q)=\sup\{p: \theta(p,q)=0\}$, where $\theta(p,q)$ is the probability that the origin lies in an infinite cluster under $\pi_{\Z^2,p,q}$.
Our proofs hinge on recent results of~\cite{DST15} on this phase transition, summarized as follows.
\begin{theorem}[{\cite[Theorem~3]{DST15}}]\label{thm:DC-S-T-main}
Let $q\geq 1$; the following statements for the critical FK model on $\Z^2$ are equivalent:
\begin{enumerate}
\item Discontinuous phase transition: $\pi^0_{\mathbb Z^2} \neq\pi^1_{\mathbb Z^2}$.
\item Exponential decay of correlations under $\pi^0$: there exists some $c>0$ such that
\begin{align} \label{eq:exp-decay-1}
\pi^0_{\mathbb Z^2}\left((0,0)\longleftrightarrow \partial \llb -n,n\rrb^2\right)\leq e^{-cn}\,.
\end{align}
\end{enumerate}
\end{theorem}
Discontinuity of the phase transition, conjectured for all $q>4$, was first proved by Koteck{\'y} and Shlosman~\cite{KS82} for sufficiently large $q$; the proof in~\cite{LMR86} applies whenever
$q^{1/4} > (\kappa + \sqrt{\kappa^2 - 4})/2$, where $\kappa$ is the connective constant of $\Z^2$. Plugging in the rigorous bound $\kappa < 2.6792$ due to~\cite{PoAn00} affirms the phase coexistence for all $q > 24.78$.
For $1< q\leq 4$, the continuity of the phase transition was established in~\cite{DST15} via the following RSW estimates\footnote{The proofs in~\cite{DST15} of Theorems~\ref{rmk:crossing-no-bc} and~\ref{rmk:circuit-in-annulus} were for the special case of $\epsilon=\epsilon'$ but readily extend to the more general setting presented here.} (note the difference between $1<q<4$ and the extremal case $q=4$, where full RSW-type bounds are believed to fail).

\begin{theorem}[{\cite[Theorem~7]{DST15}}]
\label{thm:crossing-any-bc-q<4}
Consider the critical FK model for $1\le q < 4$ on $\Lambda=\Lambda_{n,n'}$  with $n'=\lfloor \alpha n\rfloor $ for fixed $0<\alpha\le 1$ and arbitrary boundary conditions $\xi$. Then there exists some $p_0=p_0(q,\alpha)>0$ such that
\[\pi _{\Lambda} ^\xi (\cC_v(\Lambda))>p_0\,.
\]
\end{theorem}

\begin{theorem}[{\cite[Theorem~3]{DST15}}]\label{rmk:crossing-no-bc}
Let $q=4$ and consider the critical FK model on $\Lambda=\Lambda_{n,n'}$ with $n'=\lfloor \alpha n\rfloor$ for fixed $0<\alpha\le1$. Then for every $\epsilon,\epsilon'>0$ there exists some $p_0=p_0(\alpha,\epsilon,\epsilon')>0$ such that, for every boundary condition $\xi$,
\[\pi_{\Lambda}^\xi (\cC_v(\llb \epsilon n, (1-\epsilon)n\rrb \times \llb \epsilon'n',(1-\epsilon' )n'\rrb )>p_0\,.
\]
\end{theorem}

\begin{theorem}[{\cite[Proposition~2]{DST15}}]\label{rmk:circuit-in-annulus}
Fix $\epsilon,\epsilon'>0$ and $0<\alpha\leq 1$, and consider the critical FK model at $1\leq q\leq 4$
on the annulus $A=\Lambda_{n,n'}-\llb \epsilon n,(1-\epsilon)n\rrb \times \llb \epsilon'n',(1-\epsilon')n'\rrb $ for $n'=\lfloor \alpha n\rfloor$. There exists $p_0=p_0(q,\alpha,\epsilon,\epsilon')$ so that, for every boundary condition $\xi$,
\[\pi_{A}^\xi\left(\anncircuit(A)\right)>p_0\,.\]
\end{theorem}

A consequence of the above RSW-type bounds is polynomial decay of correlations for the critical FK model at $1\leq q\leq 4$ (see, e.g., the proof of~\cite[Lemma~1]{DST15}).

\begin{theorem}[decay of correlations]\label{thm:polynomial-decay}
For $1\leq q\leq 4$, there exist $c_1,c_2>0$ such that
\[{n^{-c_1}}\lesssim\pi_{\mathbb Z^2}\left((0,0)\longleftrightarrow\partial \llb -n,n\rrb^2 \right)\lesssim{n^{-c_2}}\,.
\]
\end{theorem}

\subsection{Markov chain mixing times}\label{subsec-prelim-mc-mixing}\label{sec:mcmt}
Consider a Markov chain $X_t$ with finite state space $\Omega$, transition kernel $P$ and stationary distribution $\pi$. In the continuous-time setting, instead of $P^t$ consider the heat kernel given by
\[H_t(x,y)=\P_x(X_t=y)=e^{t\mathcal L}(x,y)\,,
\]
where $\mathcal L =\lim_{t\downarrow 0}\tfrac 1t (H_t-I)$ is the infinitesimal generator.

\subsubsection*{Spectral gap}
The mixing time of the Markov chain is intimately related to the gap in its spectrum: in discrete-time, $\gap:=1-\lambda_2$ where $\lambda_2$ is the second largest eigenvalue of  $P$, and in continuous-time it is the gap in the spectrum of the generator $\mathcal L$. An important variational characterization of the spectral gap is given by the Dirichlet form:
\begin{equation}\label{eq:dirichlet-form}\gap=\inf_{f\in L^2(\pi)} \frac {\mathcal E(f,f)}{\mbox{Var}_\pi f}\,,\quad\mbox{ where }\quad \mathcal E(f,f)=\frac 12\sum _{x,y\in \Omega} \pi(x)P(x,y)(f(y)-f(x))^2\,.
\end{equation}

\subsubsection*{Mixing times}
Denote the (worst-case) total variation distance between $X_t$ and $\pi$ by
\[d_\tv(t)=\max_{x\in\Omega} \|P^t(x,\cdot)-\pi\|_{\tv}\,,\]
where the total variation distance between two probability measures $\nu,\pi$ on $\Omega$ is
\[\|\pi-\nu\|_{\tv}=\sup_{A\subset \Omega} [\pi(A)-\nu(A)] = \tfrac 12 \|\pi-\nu\|_{L^1}\,.
\]
Further define the coupling distance
\[\bar d_\tv(t)=\max_{x,y\in\Omega}\|P^t(x,\cdot)-P^t(y,\cdot)\|_{\tv}\,,
\]
noting that $\bar d_\tv$ is submultiplicative and
$d_\tv(t)\leq \bar d_\tv(t)\leq 2d_\tv(t)
$.
The  total variation \emph{mixing time} of the Markov chain w.r.t.\ the precision parameter $0<\delta<1$ is
\[\tmix(\delta)=\inf_t \{t: \max_{x\in \Omega} \|P^t(x,\cdot)-\pi\|_{\tv}<\delta\}\,.
\]
For any choice of $\delta<\frac12$, the quantity $\tmix(\delta)$ enjoys submultiplicativity thanks to the aforementioned connection with $\bar{d}_\tv$; we write $\tmix$, omitting the precision parameter $\delta$, to refer to the standard choice of $\delta=1/(2e)$.

The total variation mixing time is bounded from below and from above via the gap: one has $\tmix \geq \gap^{-1} - 1$, and if $\gap_\star$ is the \emph{absolute} spectral gap\footnote{For a discrete-time chain, $\gap= \min_i(1-|\lambda_i|)$ where the $\lambda_i$'s are the nontrivial eigenvalues of the transition kernel; in our applications, $\gap_\star=\gap$.} of the chain then $\tmix \leq \log(2e/\pi_{\min}) \gap_\star^{-1}$, where $\pi_{\min} = \min_x \pi(x)$ (see, e.g.,~\cite[\S12.2]{LPW09}).
For the FK and Potts models on a box with $O(n^2)$ and fixed $0<p<1$ and $q\geq 1$, there exists some $c>0$ such that $\pi_{\min}\gtrsim e^{-cn^2}$, thus $\tmix$ are $\gap^{-1}$ are equivalent up to $n^{O(1)}$-factors.

\subsection{Dynamics for spin systems}\label{subsec-prelim-dyn}
\subsubsection*{Heat-bath Glauber dynamics}

Continuous-time heat-bath Glauber dynamics for the Potts model on $\Lambda$ is the following reversible Markov chain w.r.t.\ $\mu_\Lambda$. Assign i.i.d.\ rate-1 Poisson clocks to all interior vertices of $\Lambda$. When the clock at a site $x$ rings, the chain resamples $\sigma(x)$ according to $\mu_\Lambda$ conditioned on the colors of all the sites other than $x$ to agree with their current values in the configuration $\sigma$: the probability that the new color to be assigned to $x$ will be $k\in[q]$ is proportional to $\exp(\beta \sum_{y\sim x}\one\{\sigma(y)=k\})$.

The heat-bath Glauber dynamics for the FK model on $\Lambda$ is the following reversible Markov chain w.r.t.\ $\pi_\Lambda$. Each interior edge of $\Lambda$ is assigned an i.i.d.\ rate-1 Poisson clock; when the clock at an edge $e=xy$ rings, the chain resamples $\omega(e)$ according to $\mathrm{Bernoulli}(p)$ if $x\longleftrightarrow y$ in $\Lambda-\{e\}$ and according to $\mathrm{Bernoulli}(\frac p{p+q(1-p)})$ otherwise.
The random mapping representation of this dynamics views the updates as a sequence $(J_i,U_i,T_i)_{i\geq 1}$, in which $T_1<T_2<\ldots$ are the update times, the $J_i$'s are i.i.d.\ uniform edges (the updated locations), and the $U_i$'s are i.i.d.\ uniform on $[0,1]$: at time $T_i$, writing $J_i=xy$, the dynamics replaces the value of $\omega(J_i)$ by $\one\{U_i\leq p\}$ if $x\longleftrightarrow y$ in $\Lambda-\{J_i\}$ and by $\one\{U_i\leq \frac{p}{p+q(1-p)}\}$ otherwise.

\subsubsection*{Monotonicity and censoring inequalities} \label{sub:grand-coupling}
The heat-bath Glauber dynamics for the FK model at $q\geq 1$ is \emph{monotone}: for every two FK configurations $\omega_1\ge \omega_2$ and every $t\ge 0$,
\[H_t(\omega_1,\cdot)\succeq H_t(\omega_2,\cdot)\,.
\]
The \emph{grand coupling} for Glauber dynamics is a coupling of the chains from all  initial configurations on $\Lambda$: one appeals to the random mapping representation of Glauber dynamics described above, using the same update sequence $(J_i,U_i,T_i)_{i\geq 1}$ for each one of these chains. For $q\geq 1$, the monotonicity of the dynamics guarantees that this coupling preserves the partial ordering of the configurations at all times $t\geq 0$.

In particular, under the grand coupling, the value of an edge $e$ in Glauber dynamics at time $t$ from an arbitrary initial state $\omega_0$, is sandwiched between the corresponding values from the free and wired initial states; thus, by a union bound over all edges, \[d_\tv(t) \leq |E(\Lambda)|\,\|H_t(1,\cdot)-H_t(0,\cdot)\|_{\tv}
\]
(see this well-known inequality, e.g., in~\cite[Eq.~(2.10)]{MaTo10}),
and consequently,
\begin{align}\label{eq:init-config-comparison}
\bar {d}_\tv(t)\leq 2 |E(\Lambda)|\, \|H_t(1,\cdot)-H_t(0,\cdot)\|_{\tv}\,.
\end{align}

The Peres--Winkler censoring inequalities~\cite{PW13} for monotone spin systems allow one to ``guide'' the dynamics to equilibrium by restricting the updates to prescribed parts of the underlying graph, thus supporting an appropriate multi-scale analysis, the key being that censoring all other updates can only slow down mixing
(this next flavor of the inequality follows from the same proof of~\cite[Theorem~1.1]{PW13}; see~\cite[Theorem~2.5]{MaTo10}).
\begin{theorem}[\cite{PW13}]\label{thm:censoring}
Let $\mu_T$ be the law of continuous-time Glauber dynamics at time $T$ of a monotone spin system on $\Lambda$ with stationary distribution $\pi$, whose initial distribution $\mu_0$ is such that $\mu/\pi$ is increasing. Set $0=t_0 < t_1 <\ldots < t_k = T$ for some $k$, let $(\Lambda_i)_{i=1}^k$ be subsets of the sites $\Lambda$, and let $\tilde\mu_T$ be the law at time $T$ of the censored dynamics, started at $\mu_0$, where only updates within $\Lambda_i$ are kept in the time interval $[t_{i-1},t_i)$. Then $\|\mu_T-\pi\|_\tv \leq \|\tilde\mu_T-\pi\|_\tv$ and $\mu_T \preceq \tilde\mu_T$; moreover, $\mu_T/\pi$ and $\tilde\mu_T/\pi$ are both increasing.
\end{theorem}

\subsubsection*{Cluster dynamics}

 Swendsen--Wang dynamics for the $q$-state Potts model on $G=(V,E)$ at inverse-temperature $\beta$ is the following discrete-time reversible Markov chain. From a spin configuration $\sigma\in\Omega_{\potts}$ on $G$, generate a new state $\sigma'\in\Omega_{\potts}$ as follows.
\begin{enumerate}
\item Introduce auxiliary random cluster edge variables and set $e=xy\in E$ to be open with probability $0$ if $\sigma_x\neq \sigma_y$ and probability $1-e^{-\beta}$ if $\sigma_x = \sigma_y$.
\item For every connected vertex set of the resulting edge configuration,
reassign the cluster, collectively, an i.i.d.\ color in $[q]$, to obtain the new configuration $\sigma'$.
\end{enumerate}
 Chayes--Machta dynamics for the FK model on $G=(V,E)$ with parameters $(p,q)$, for $q \geq 1$ and $0<p<1$,
is the following analogous discrete-time reversible Markov chain:
 From an FK configuration $\omega\in\Omega_{\rc}$ on $G$,  generate a new state $\omega'\in\Omega_{\rc}$ as follows.
\begin{enumerate}
\item Assign each cluster $C$ of $\omega$ an auxiliary i.i.d.\ variable $X_c\sim\mathrm{Bernoulli}(1/q)$.
\item Resample every $e=xy$ such that $x$ and $y$ belong to clusters with $X_c=1$
via i.i.d.\ random variables $X_e\sim\mathrm{Bernoulli}(p)$, to obtain the new configuration $\omega'$.
 \end{enumerate}

In the presence of boundary conditions, Step (2) of the Swendsen--Wang dynamics does not reassign the color of any cluster that is incident to a vertex whose color is dictated by the boundary conditions, and analogously, Step (2) of the Chayes--Machta dynamics does not resample an edge whose value is dictated by the boundary conditions.

Variants of Chayes--Machta dynamics with $1\le k\le \lfloor q\rfloor$ ``active colors" have also been studied, with numerical evidence for  $k=\lfloor q \rfloor$ being the most efficient choice; see~\cite{GOPS11}.

\subsubsection*{Spectral gap comparisons}

The following comparison inequalities between the above Markov chains are due to Ullrich (see~\cite[Thm.~1]{Ul13},~\cite[Thm~4.8 and Lem.~2.7]{Ul14}).

\begin{theorem}[\cite{Ul13,Ul14}]\label{thm:Ullrich-comparison}
Let $q \geq 2$ be integer.
Let $\gap_{\textsc p}$ and $\gap_{\rc}$ be the spectral gaps of Glauber dynamics for the Potts and FK models, respectively,  on a graph $G=(V,E)$ with maximum degree $\Delta$ and no boundary conditions, and let $\gap_{\textsc {sw}}$
be the
spectral gap of Swendsen--Wang.
Then we have 
\begin{align}
\gap_{\textsc p} &\leq 2 q^2(qe^{2\beta})^{4\Delta} \gap_{\textsc {sw}}\,,\label{eq-ullrich1}\\
(1-p+p/q)\gap_{\rc} &\leq\gap_{\textsc {sw}}\leq 8 \gap_{\rc}\, |E| \log |E| \,.\label{eq-ullrich2}
\end{align}
\end{theorem}
The proof of~\eqref{eq-ullrich2} further extends to all real $q> 1$, whence
\begin{align}\gap_{\rc} &\lesssim\gap_{\textsc {cm}}\lesssim \gap_{\rc}\, |E|\log |E| \label{eq-ullrich2-realq}\,,
\end{align}
as was observed (and further generalized) by  Blanca and Sinclair~\cite[\S5]{BlSi14}, where $\gap_{\textsc{cm}}$ is the spectral gap of Chayes--Machta dynamics.
In particular, because the estimates of Theorem~\ref{thm:Ullrich-comparison} hold for general graphs $G$ and are formulated in the absence of  boundary conditions, they hold immediately for $\Lambda$ with periodic or free boundary conditions.

\begin{remark} \label{rem:Ullrich-boundary-conditions}
In the presence of Potts (FK) boundary conditions, one can define a new graph $G'$ where all boundary vertices of the same color (in the same FK boundary cluster) are identified. However, the constant in~\eqref{eq-ullrich1} is exponential in the maximum degree; this can be improved to exponential in the maximum degree of all but one vertex (see~\cite[Theorem~1']{Ul13}). Thus, Eq.~\eqref{eq-ullrich1} holds also in the presence of boundary conditions through which \emph{at most one} given subset of vertices are wired, or, by spin flip symmetry, assigned a fixed color in $[q]$.  The estimates of~\eqref{eq-ullrich2} and~\eqref{eq-ullrich2-realq} are uniform in $\Delta$ and therefore hold in the presence of general FK boundary conditions.
\end{remark}

\subsubsection*{Block dynamics}\label{subsub:block-dynamics}

A key ingredient in the proof of~\cite{LScritical}, as well as our proof of Theorem~\ref{mainthm-q<=4}, is the block dynamics technique due to Martinelli (see~\cite[\S3]{Martinelli97}) for bounding the spectral gap of the Glauber dynamics. Suppose $B_1,...,B_k$ are such that $B_1 -\partial B_1,...,B_k-\partial B_k$ covers $\Lambda$. Then the block dynamics is the corresponding Glauber dynamics that updates one block (instead of one site) at a time: each block is assigned a rate-1 Poisson clock; when the clock at $B_i$ rings, resample the configuration on $B_i-\partial B_i$ according to $\mu_{B_i}^\sigma$ where the boundary conditions $\sigma$ are given by the chain restricted to $\Lambda-(B_i-\partial B_i)$.

\begin{theorem}[{\cite[Proposition~3.4]{Martinelli97}}]\label{thm:block-dynamics}
Consider a continuous-time single-site Markov chain for the Potts model on $\Lambda$ with boundary condition $\zeta$, which is reversible w.r.t.\ the Gibbs distribution $\mu_\Lambda^\zeta$.
Let $\gap_\Lambda^\zeta$ and $\gap_{\cB}^{\zeta}$ respectively be the spectral gaps of the single-site dynamics on $\Lambda$ and block dynamics corresponding to $B_1,\ldots,B_k$ such that $B_1^o,...,B_k^o$ cover $\Lambda$. Then letting $\chi = \sup_{x\in\Lambda} \#\{i : B_i \ni x\}$, we obtain
\begin{equation*}
  \gap_\Lambda^\zeta \geq \chi  ^{-1} \gap_\mathcal{B}^\zeta \inf_{i,\varphi} \gap_{B_i}^\varphi \,.
\end{equation*}
\end{theorem}

\subsubsection*{Canonical paths} \label{subsub:canonical-paths}
The following well-known geometric approach (see~\cite{DiSa93,DiSt91,JeSi89,Sinclair92} as well as~\cite[Corollary 13.24]{LPW09}) serves as an effective method for obtaining an upper bound on the inverse gap of a Markov chain, and will be used in our proof of Theorem~\ref{mainthm:large-q-upper}.

\begin{theorem}\label{thm:canonical-paths}
Let $P$ be the transition kernel of a discrete-time Markov chain with stationary distribution $\pi$, and  write $Q(x,y)=\pi(x)P(x,y)$ for every $x,y\in\Omega$. For
each $(a,b)\in\Omega^2$, assign a path
$\gamma(a,b)=(x_0=a,\dots,x_n=b)$ such that
$P(x_i,x_{i+1})> 0$ for all $i$, and write $|\gamma(a,b)|=n$, identifying $\gamma(a,b)$ with $\{(x_{i-1},x_{i}):i=1,\ldots,n\}$. Then,
\begin{align} \label{eq:canonical-paths}
  \gap^{-1}\le \max_{\substack{x,y\in\Omega\\ Q(x,y)>0}}
\frac1{Q(x,y)}\sum_{\substack{a,b\in\Omega \\ (x,y)\in\gamma(a,b)}}
|\gamma(a,b)|\pi(a)\pi(b).
\end{align}
\end{theorem}

A very standard application of Theorem~\ref{thm:canonical-paths} (see e.g.,~\cite{Martinelli94} in the setting of the Ising model) proves upper bounds on mixing times of spin systems in terms of the cut-width of the underlying graph. We omit the proof and note that it follows for the Potts and FK models by making the natural modifications and observing that in the FK setting, the probability of any single edge-flip is at least some $c(p,q)>0$. 

\begin{lemma}\label{lem:gap-shorter-side} Consider the Glauber dynamics for the $q$-state Potts model at inverse temperature $\beta$ on a rectangle $Q=\llb 0,n\rrb \times \llb 0,\ell\rrb$ for $0\leq \ell \leq n$, with arbitrary boundary conditions. There exists a constant $c(\beta,q)>0$ such that
\[\gap _{Q}^{-1}\lesssim e^{c\ell}\,,
\]
and an analogous bound holds for the heat-bath dynamics on the FK model.
\end{lemma}

\section{Mixing at a continuous phase transition}\label{sec:mixing-q-leq-4}

This section contains the proof of Theorem~\ref{mainthm-q<=4} (as well as its analogs for boxes with non-periodic boundary conditions); recall from~Theorem~\ref{thm:Ullrich-comparison} that it suffices to prove the desired bounds for Glauber dynamics for the Potts model in order to obtain them for FK Glauber as well as Swendsen--Wang and Chayes--Machta dynamics. Consider $\Lambda=\Lambda_{n,n'}=\llb 0,n\rrb\times \llb0,n'\rrb$ for $n'=\lfloor \alpha n\rfloor$, where  $\alpha\in[\bar \alpha,1]$ for some fixed $0<\bar \alpha \leq \frac 12$.

\subsection{Mixing under arbitrary boundary conditions}\label{sub:Potts}

We first establish analogues of Eqs.~\eqref{eq:q=3}--\eqref{eq:q=4} for Glauber dynamics for the Potts model with arbitrary boundary conditions, modulo an equilibrium estimate on crossing probabilities at $q=4$ which we establish in~\S\ref{sub:q=4}. Whenever we refer to arbitrary or fixed boundary conditions we mean ones that include an assignment of a color, or \emph{free} to each of the vertices of $\partial \Lambda$ (in contrast to periodic). The following is a general form of the approach of~\cite{LScritical} to proving upper bounds on mixing times in the presence of RSW bounds; we stress that, while this proof does extend from the Ising model to the Potts model, in fact it fails to produce a polynomial upper bound for the critical FK model at noninteger $1<q<4$, despite the availability of the necessary (uniform) RSW estimates (cf.~\cite{GL16b}).

\begin{theorem}\label{thm:general-RSW-mixing} Suppose $q\geq 1$ and there exists a nonincreasing sequence $(a_n)$ such that
\begin{align}\label{eq:a-n}
\inf_\xi \pi^\xi_{\Lambda_{n/3,n}}\left (C_v(\Lambda_{n/3,n})\right )\geq a_n\,.
\end{align}
Then there exists some absolute constant $c>0$ such that
Glauber dynamics for the Potts model on $\Lambda=\Lambda_{n,n'}$ with arbitrary boundary conditions, $\zeta$, satisfies
\[\gap^{-1}\leq (c\, a_n)^{-2\log_{3/2} n}\,.
\]
\end{theorem}

Combining the RSW bound of Theorem~\ref{thm:crossing-any-bc-q<4} with Theorem~\ref {thm:general-RSW-mixing} establishes the  analog of Eq.~\eqref{eq:q=3} for a rectangle with arbitrary (non-periodic) boundary conditions. At $q=4$, we will later prove a polynomially decaying bound on crossing probabilities uniform in boundary conditions (see Theorem~\ref{thm:crossing-any-bc}), through which Theorem~\ref{thm:general-RSW-mixing} will yield the matching quasi-polynomial upper bound on mixing.

\begin{corollary}\label{cor:poly/quasipoly}
There exist absolute constants $c_1,c_2>0$ such that Glauber dynamics for the critical $3$-color Potts model on $\Lambda$ with arbitrary fixed boundary satisfies
\[
\gap^{-1} \lesssim n^{c_1}\,,
\]
whereas for the $4$-state critical Potts model on $\Lambda$ with arbitrary boundary conditions,
\[
\gap^{-1} \lesssim n^{c_2 \log n}\,.
\]
\end{corollary}

From this corollary, Eqs.~\eqref{eq:q=3}--\eqref{eq:q=4} of Theorem~\ref{mainthm-q<=4} follow by moving from the box to the torus exactly as done in~\cite[Theorem~4.4]{LScritical} (see also~\S\ref{sub:torus}). For $q=\{2,3,4\}$, by Theorem~\ref{thm:Ullrich-comparison}, these imply the analogous upper bounds on the inverse gap of the Swendsen--Wang dynamics, as well as Glauber dynamics for the FK model.

\begin{proof}[\textbf{\emph{Proof of Theorem~\ref{thm:general-RSW-mixing}}}]
We use the block dynamics technique of Theorem~\ref{thm:block-dynamics} used in~\cite{LScritical}. Define two sub-blocks of $\Lambda$, as follows:
\[B_\west:=\llb 0,\tfrac {2n}3\rrb \times\llb 0,n'\rrb\,, \qquad B_\east:=\llb \tfrac n3,n\rrb \times \llb 0,n'\rrb\,.
\]
Then let  $\mathcal B$ denote the block dynamics on $\Lambda$ with sub-blocks $B_\west,B_\east$ as defined in~\S\ref{subsub:block-dynamics}.
We bound $\gap_\mathcal B ^\zeta$ and $\gap_{B_i} ^\varphi$ of Theorem~\ref{thm:block-dynamics} uniformly in $\zeta, \varphi$.

\begin{lemma}
For any two initial configurations $\sigma,\sigma'$ on $\Lambda$ with corresponding block dynamics chains $X_t$ and $Y_t$, there exists an absolute constant $c>0$ such that, if $(a_n)$ is a sequence satisfying~\eqref{eq:a-n}, there is a grand coupling, such that
$\mathbb P(X_1\neq Y_1)\leq 1-c\, a_n
$.
Moreover, there exists some $c'>0$ such that
$\gap_{\mathcal B}^\zeta\geq c'\, a_n
$ uniformly in $\zeta$.
\end{lemma}

\begin{proof}
We construct explicitly a grand coupling that allows us to couple the two configurations with the above probability. First recall that the Potts boundary condition $\zeta$ on $\Lambda$ corresponds to an FK boundary condition $\xi$ where two boundary vertices are in the same cluster if and only if they have the same color, along with the decreasing event $\mathcal E_\zeta=\{\omega:\forall x,y\in \partial \Lambda, \zeta(x)\neq \zeta(y) \implies x\stackrel{\Lambda}\nconn y\}$.
Via the Edwards--Sokal coupling, we move from the Potts model with boundary $\zeta$ to the corresponding FK model with boundary $\xi$ conditional on the event $\mathcal E_\zeta$.

Suppose the clock at block $B_\west$ rings first. The two initial configurations $\sigma,\sigma'$ induce two Potts boundaries $\eta,\eta'$ corresponding to FK boundaries $\psi,\psi'$ on $\partial_\east B_\west$
along with the events $\mathcal E_{\eta,\zeta}$ and $\mathcal E_{\eta',\zeta}$; $(\eta,\zeta)$ is the boundary condition on $B_1$ with $\eta$ on $\partial_\east B_1$ and $\zeta\restriction_{\partial B_\west}$ on the rest of $\partial{B_\west}$. Here and throughout the rest of the paper, when discussing boundary conditions, we use the restriction to a line to denote the boundary condition induced on that line by the configuration we have revealed.

We seek to couple the two initial configurations on all of $\Lambda$ by first coupling them on $\Lambda-B_\east^o$. For each initial configuration, the block dynamics samples a Potts configuration on $B_\west$ by sampling an FK configuration from $\pi^{\psi,\xi}_{B_\west}(\cdot \mid \mathcal E_{\eta,\zeta})$, $\pi^{\psi',\xi}_{B_\west}(\cdot \mid \mathcal E_{\eta',\zeta})$.

Via the grand coupling defined in~\S\ref{sub:grand-coupling} of all boundary conditions on $\partial_\east B_\west$, we reveal the open component of $\partial_\east B_\west$ in order to condition on the right-most dual vertical crossing of $\Lambda-B_\east^o$. Note that all FK measures we consider are stochastically dominated by $\pi^{1,\xi}_{B_\west}$ (by monotonicity in boundary conditions and since $\mathcal E_{\eta,\zeta}$ is a decreasing event). Then if a sample from $\pi^{1,\xi}_{B_\west}$ has a dual vertical crossing in $B_\west\cap B_\east$, under the grand coupling, so will all the samples of $\pi^{\psi,\xi}_{B_\west}(\cdot \mid \mathcal E_{\eta,\zeta})$.

Under the event $\mathcal E_{\eta,\zeta}$, by construction it is impossible to add boundary connections by modifying the interior of $B_\west$ (either such connections would be between monochromatic sites in which case they are already in the same cluster, or otherwise such connections are impossible under $\mathcal E_{\eta,\zeta}$). Thus, if there is such a dual vertical crossing under $\pi^{1,\xi}_{B_\west}$, the event $\mathcal E_{\eta,\zeta}$ ensures that to the west of that crossing, all realizations of $\pi_{B_\west}^{\psi,\xi}$ see the same boundary conditions. By the domain Markov property, the grand coupling then couples all such realizations west of the right-most dual-crossing of $\pi^{1,\xi}_{B_\west}$ and therefore on all of $\Lambda-B_\east^o$ (for the explicit revealing procedure, see~\cite[\S3.2]{LScritical}).
We then use the same randomness to color coupled clusters the same way, and couple all corresponding Potts configurations on $\Lambda-B_\east^o$. The colorings of the boundary clusters are predetermined, but because the $\partial_\east B_\west$ boundary clusters cannot extend past the dual vertical crossing, the two Potts configurations can be coupled west of the dual vertical crossing.

Suppose the clock at block $B_\east$ rings next. If we have successfully coupled $X_t$ and $Y_t$ in $\Lambda-B_\east^o$, then the identity coupling couples the configurations on all of $\Lambda$.
But note that by the assumption of Theorem~\ref{thm:general-RSW-mixing}, and the fact that $n'\le n$,
\[\pi^1_{B_\west} \left(\cC_v^\ast (\llb \tfrac n3,\tfrac {2n}3\rrb \times \llb 0,n'\rrb )\right)\geq a_n\,.
\]

Moreover, by time $t=1$, there is a probability $c>0$ that the dynamics rang the clock of $B_\west$ and then the clock of $B_\east$ in which case we have coupled the two configurations with probability $a_n$ at time $t=1$.
By the submultiplicativity of $\bar d(t)$, for all $t>0$,
\[
\P(X_t\neq Y_t)\leq (1-c\, a_n)^{t}\,,
\]
which implies that there exists a new constant $c>0$ such that $\tmix\leq c / a_n$ and
\[(\log {\tfrac 1 {2\epsilon}}){(\gap ^\zeta_\mathcal B)}^{-1}\leq \tmix(\epsilon)\leq c /a_n\,.
\]
In particular there exists $c'>0$ such that
${(\gap ^\zeta_\mathcal B)}^{-1}\leq c' /a_n$. \qedhere
\end{proof}

By Theorem~\ref{thm:block-dynamics} there exists $c>0$ such that we get the following relation between the gap of Glauber dynamics on $\Lambda$ and Glauber dynamics on the blocks $B_i$, $i\in\{\east,\west\}$:
\begin{align*}
(\gap_{{\Lambda}}^{\zeta})^{-1}\leq &( c\, a_n)^{-1}\max_{i}\max_{\sigma}(\gap_{B_{i}}^{\sigma})^{-1}.
\end{align*}
However, each $B_{i}$ is a rectangle $\Lambda_{2n/3,n'}$ with arbitrary boundary conditions and one can check by hand that for $\alpha\in[\bar \alpha,1]$, it also has, up to rotation, aspect ratio $\alpha_i\in[ \bar \alpha,1 ] $. It follows that~$\max_{i}\max_{\sigma}(\gap_{B_{i}}^{\sigma})^{-1}$
satisfies the same relation as $(\gap_{\Lambda}^{\zeta})^{-1}$.
Recursing $2\log_{3/2} n$ times yields the desired bound on $(\gap_{\Lambda}^\zeta)^{-1}$ for any $\zeta$.
\end{proof}

\subsection{Crossing probabilities at \texorpdfstring{$q=4$}{q=4}}\label{sub:q=4}

Recall that, for $1\leq q<4$, the probability of a horizontal crossing of a rectangle with arbitrary boundary conditions is uniformly bounded away from 0 (Theorem~\ref{thm:crossing-any-bc-q<4}), whereas at $q=4$, under \emph{free} boundary conditions, it is expected that the probability of such a crossing of $\Lambda$ in fact decays to $0$ as $n\to\infty$. We lower bound this crossing probability under general boundary conditions.

\begin{theorem}\label{thm:crossing-any-bc}
Let $q=4$ and consider the critical FK model on $\Lambda=\Lambda_{n,n'}$, where $n'=\lfloor \alpha n \rfloor $ for a fixed aspect ratio $0<\alpha \le 1$, with arbitrary boundary conditions $\xi$.
Then there exist some $c(\alpha),\gamma(\alpha)>0$ (independent of $\xi$) such that
\[\pi_{\Lambda}^\xi\left(\cC_h(\Lambda)\right)>cn^{-\gamma}\,.
\]
\end{theorem}

\begin{proof}
By monotonicity in boundary conditions, it suffices to prove the above for free boundary conditions (the case $\xi=0$).
Fix $\delta>0$ and let $R=\llb 0,n\rrb \times\llb (\tfrac 12 -\delta) n' ,(\tfrac 12+\delta)  n'\rrb$. We will show the stronger result that there exist some $\gamma(\alpha),\gamma'(\alpha)>0$ such that,
\begin{equation}\label{eq:crossing-any-bc}
n^{-\gamma} \lesssim \pi_{\Lambda}^0 \left((0, n'/2)\stackrel{R}\longleftrightarrow(n,n'/2 )\right) \lesssim n^{-\gamma'}\,.
\end{equation}

The upper bound in~\eqref{eq:crossing-any-bc} is a consequence of the polynomial decay of correlations in Theorem~\ref{thm:polynomial-decay}, and it remains to establish the lower bound. 
Observe that for every $e$,
\[\inf_{\omega'\in \{0,1\}^{E(\Lambda)-\{e\}}}\pi^{\omega'}_{\Lambda}(e\in\omega)= \frac{p_c}{p_c+q(1-p_c)} = \frac{1}{1+\sqrt{q}} \,;
\]
thus, we can force all the edges of $R_0=\llb 0, 2\log n\rrb \times \{\tfrac {n'}2\}$ to be open with probability at least
$(1+\sqrt{q})^{-2\log n} = n^{-2\beta_c}$,
where we recall that $\beta_c = \log(1+\sqrt{q})$.

We boost this to a horizontal crossing of length $\delta  n/2$ from the boundary by stitching together horizontal and vertical crossings and applying the FKG inequality (see Fig.~\ref{fig:stitching-anybc}). 
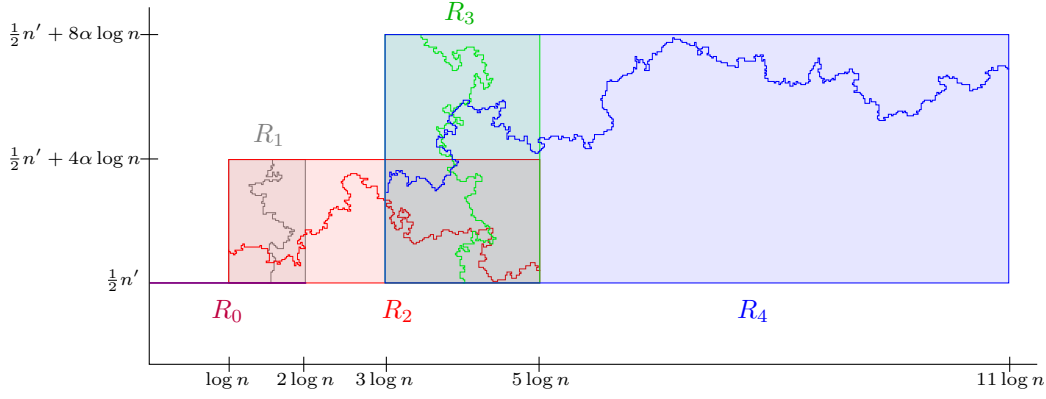
\begin{figure}
  \hspace{-0.15in}
  \begin{tikzpicture}
    \node (plot) at (0,0)
    {\hspace{-0.14in}\includegraphics[width=0.82\textwidth]{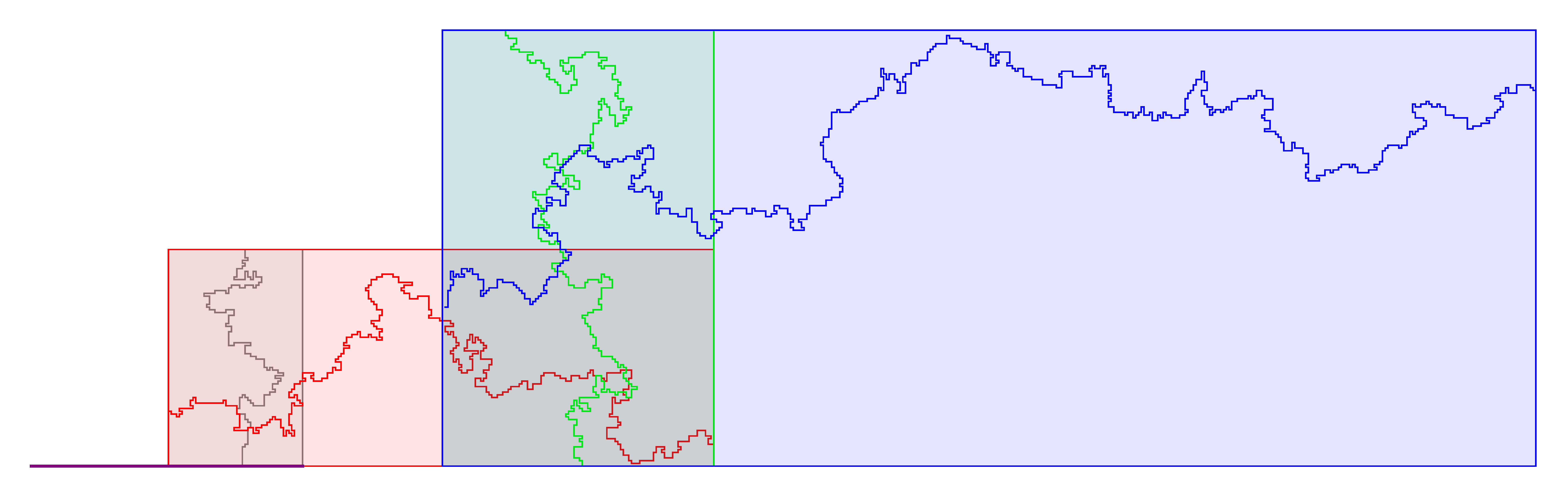}};

    \begin{scope}[shift={(plot.south west)},x={(plot.south
        east)},y={(plot.north west)}, font=\small]

      \draw[color=black] (1,-0.177) -- (0,-0.177) -- (0,1);
      \draw[color=black] (0.01,0.5) -- (-0.01,0.5);
      \draw[color=black] (0.01,0.91) -- (-0.01,0.91);

      \draw[color=black] (0.09,-0.202) -- (0.09,-0.152);
      \draw[color=black] (0.174,-0.202) -- (0.174,-0.152);
      \draw[color=black] (0.267,-0.202) -- (0.267,-0.152);
      \draw[color=black] (0.439,-0.202) -- (0.439,-0.152);
      \draw[color=black] (0.969,-0.202) -- (0.969,-0.152);

      \node[font=\tiny] at (-0.03,0.1) {$\frac12n'$};
      \node[font=\tiny] at (-0.085,0.5) {$\frac12n' + 4\alpha \log n$};
      \node[font=\tiny] at (-0.085,0.91) {$\frac12n' + 8\alpha \log n$};

      \node[font=\tiny] at (0.088,-0.23) {$\log n$};
      \node[font=\tiny] at (0.176,-0.23) {$2\log n$};
      \node[font=\tiny] at (0.265,-0.23) {$3\log n$};
      \node[font=\tiny] at (0.441,-0.23) {$5\log n$};
      \node[font=\tiny] at (0.97,-0.23) {$11\log n$};

      \node[color=purple] at (0.088,0) {$R_0$};
      \node[color=gray] at (0.135,0.58) {$R_1$};
      \node[color=red] at (0.28,0) {$R_2$};
      \node[color=DarkGreen] at (0.35,0.98) {$R_3$};
      \node[color=blue] at (0.68,0) {$R_4$};

    \end{scope}
  \end{tikzpicture}
  \caption{The first four steps of the stitching technique of Theorem~\ref{thm:crossing-any-bc}. Repeating $\log \epsilon n$ times creates a macroscopic horizontal open crossing.}
  \label{fig:stitching-anybc}
\end{figure}
Fix $\epsilon>0$ sufficiently small (e.g., a choice of $\epsilon =\delta/10$ would suffice),
and consider
\begin{align*}
R_{2k-1}=\llb (2^k-1)\log n,(3 \cdot 2^{k-1}-1)\log n\rrb \times \llb \tfrac {n'}2,\tfrac {n'}2+2^{k+1} \alpha \log n\rrb \,,\\
R_{2k}=\llb (2^k-1)\log n,(3\cdot 2^{k}-1)\log n\rrb \times\llb \tfrac {n'}2,\tfrac {n'}2+2^{k+1} \alpha \log n\rrb
\end{align*}
for $k=1,\ldots, K$, where $K = \lfloor  \log_2 \big(\tfrac{\epsilon n}{\log n}\big) \rfloor$.

Moreover, take $\tilde R_{2k-1}$ and $\tilde R_{2k}$ to be the concentric $\frac32$-dilations of $R_{2k-1}$ and $R_{2k}$, respectively. By construction, each $R_{2K-1}$ and $ R_{2K} $ has width at most $2\epsilon n$ and height at most $2\epsilon n'$, hence their respective dilations $\tilde R_{2K-1}$ and $\tilde R_{2K}$ are both contained in $R$.

As a consequence of $\tilde R_i \subset \Lambda$ the free boundary conditions on $\tilde R_i$ are dominated by the measure over boundary conditions induced by $\pi^0_\Lambda$. Thus, there exists some $p_1(\alpha),p_2(\alpha)>0$ given by Theorem~\ref{rmk:crossing-no-bc} such that
\[\pi ^0_{\Lambda}(\cC_v(R_{2k-1}))\geq \pi^0_{\tilde R_1}(\cC_v(R_{2k-1}))\ge p_1\,,\qquad \mbox{and likewise} \qquad \pi_{\Lambda}^0(\cC_h(R_{2k})) \geq p_2\,.
\]
(Notice the aspect ratios of $\tilde R_{2k-1}$ are the same for all $k$, and similarly for $\tilde R_{2k}$.) Further, for every $k$, these events are increasing; thus by the FKG inequality,
\[\pi^0_{\Lambda}\left(\cC_v(R_{2k-1})\cap \cC_h(R_{2k})\right)\geq p_1 p_2\,.
\]
At the final scale $K$, the width of $R_{2K}$ is $(2-o(1))\epsilon n$ and its height is $(2-o(1))\epsilon n'$, so
\[\Bigg(\{\omega\restriction_{R_0}=1\}\cap \bigcap _{k=1}^{K} (\cC_v(R_{2k-1})\cap \cC_h(R_{2k})) \Bigg)\subset(0,\tfrac {n'}2)\stackrel{R}\longleftrightarrow \{2\epsilon n\}\times\llb \tfrac12 {n'},(\tfrac12 + 2\epsilon) n'\rrb
\]
for any sufficiently large $n$.
By repeated application of the FKG inequality,
\begin{equation}
\label{eq:left-epsilon-crossing}\pi_{\Lambda} ^0 \left((0,\tfrac {n'}2)\stackrel{R}\longleftrightarrow \{2\epsilon n\}\times\llb \tfrac12 {n'},(\tfrac12 + 2\epsilon) n'\rrb \right)\geq n^{-2\beta_c}(p_1p_2)^{K} = n^{-\gamma}
\end{equation}
for some $\gamma>0$.
By symmetry, the exact same argument shows that
\begin{equation}\label{eq:right-epsilon-crossing} \pi_{\Lambda} ^0 \left((n,\tfrac {n'}2)\stackrel{R}\longleftrightarrow \{(1-2\epsilon) n\}\times\llb \tfrac12 {n'},(\tfrac12 + 2\epsilon) n'\rrb \right)
\geq n^{-\gamma}\,.
\end{equation}
In order to complete the desired horizontal crossing, we require an open path connecting the left and right crossings, via an open circuit in the annulus $A_1$ given by
 \[ A_1=\llb 0, n\rrb \times \llb (\tfrac 12-\delta )n',(\tfrac12+\delta )n'\rrb -\llb\epsilon n,(1-\epsilon)n\rrb \times \llb (\tfrac12 - 3\epsilon) n',(\tfrac12 +3\epsilon)n'\rrb\,.
 \]
By Theorem~\ref{rmk:circuit-in-annulus}, there is an absolute constant $p_3(\alpha)>0$ such that $\pi_{A_1}^0(\anncircuit(A_1))>p_3$. Since  the induced boundary conditions on $\partial A_1$ by $\pi_\Lambda^0$ stochastically dominate free boundary conditions on $\partial A_1$, it follows that $\pi_{\Lambda}^0\left(\anncircuit(A_1)\right)>p_3$. 
Finally, the event $\anncircuit(A_1)$ is increasing, and its intersection with the two horizontal crossing events from~\eqref{eq:left-epsilon-crossing} and~\eqref{eq:right-epsilon-crossing}  is a subset of the event $\{(0,\tfrac {n'}2)\stackrel{R}\longleftrightarrow(n,\tfrac {n'}2)\}$. Thus, by FKG, the latter has probability at least $p_3 n^{-2\gamma}$,
establishing~\eqref{eq:crossing-any-bc}, as desired.
\end{proof}

\subsection{Periodic boundary conditions}\label{sub:torus} We now complete the proof of Theorem~\ref{mainthm-q<=4}.

\begin{proof} [\textbf{\emph{Proof of Theorem~\ref{mainthm-q<=4}}}]
The proof to go from arbitrary boundary conditions to the torus is the same as the proof of Theorem~4.4 of~\cite{LScritical} which used block dynamics twice to first reduce mixing on the torus $(\mathbb Z/n\mathbb Z)^2$ to a cylinder $(\mathbb Z/n\mathbb Z)\times \llb 0,n'\rrb$, and then that cylinder to a rectangle with fixed boundary conditions, on which Corollary~\ref{cor:poly/quasipoly} gives the desired polynomial (quasi-polynomial) mixing time bound. We only observe that the proof goes through after replacing the RSW bounds there by the estimate in Theorem~\ref{thm:crossing-any-bc}, and conditioning on the event $\mathcal E_{\zeta}$ as before.
\end{proof}

\subsection{Polynomial lower bounds}
In order to provide as complete a picture as possible, we also extend the polynomial lower bound of~\cite{LScritical} to the Glauber dynamics for the $q=3,4$ Potts models, showing that indeed they undergo a critical slowdown. We do not have access to precise arm exponents as exist for $q=2$, but we adapt a standard argument for obtaining the Bernoulli percolation two-arm exponent, to lower bound the Potts one-arm exponent and prove a polynomial lower bound on $\gap^{-1}$.

\begin{lemma}\label{lem:two-point-corr}
Fix an $\epsilon>0$ and consider the critical FK model for $q\in (1,4]$ on $R=\llb -n,n \rrb^2$; there exists $c(q)>0$ such that 
\begin{align}\label{eq:one-arm}
\pi_{R}^0 \big(0 \longleftrightarrow \partial (\llb- \tfrac n2,\tfrac n2\rrb^2)\big) \geq cn^{-\frac 12}\,,
\end{align}
and thus, there exists $c'(\epsilon,q)>0$ such that for every $x,y\in \llb -(1-\epsilon)n,(1-\epsilon)n\rrb^2$, 
\begin{align}\label{eq:two-point-corr}
\pi_R^0 \big(x\longleftrightarrow y\big) \geq c'\|x-y\|^{-1}\,.
\end{align}
\end{lemma}

\begin{proof}
Let $L= \llb -\frac n8,\frac n8 \rrb\times \{0\}$, for every $x\in L$, set $R_x = x+\llb -\frac n2, \frac n2\rrb^2$, and let $R_x^+$ be the top half of $R_x$. Also let $B=\llb -\frac{3n}4,\frac{3n}4\rrb^2$ and define the event 
\begin{align*}
\Gamma_x= \{x\longleftrightarrow \partial_\north R_x\mbox{ in $R_x^+$}\}\cap \{x+(\tfrac 12,0) \stackrel{*}\longleftrightarrow \partial_\north R_x \mbox{ in $R_x^+$}\}\,.
\end{align*}
Consider the event $\Gamma$ that there exists a site $x\in L$ such that $\Gamma_x$ holds. We begin by proving that $\pi_{B}^1(\Gamma)\geq c$ for some $c>0$ independent of $n$. By using Theorem~\ref{thm:crossing-any-bc-q<4}--\ref{rmk:crossing-no-bc} and stochastic domination twice, we see that there exists $c(q)>0$ such that
\begin{align*}
\pi_B^1(\cC_v(\llb -\tfrac n8,-\tfrac n{10}\rrb\times \llb 0,\tfrac n2\rrb) \cap \cC_v^\ast (\llb \tfrac n{10},\tfrac n8\rrb \times \llb 0,\tfrac n2\rrb ))\geq c\,.
\end{align*}
But one can observe that the above event implies that the right-most point on $L$ that is part of the cluster of the vertical open crossing in $R_0^+$ satisfies $\Gamma_x$, so $\pi_B^1(\Gamma) \geq c$. At the same time, we have by a union bound that 
\begin{align*}
\pi_B^1(\Gamma) \leq \sum_{x\in L} \pi_B^1 (\Gamma_x) & \leq (n/4) \max_{x\in L} \pi_B^1(\Gamma_x) \qquad \mbox{so that} \\
\max_{x\in L} \pi_B^1 (\Gamma_x) & \geq 4c  n^{-1}\,.
\end{align*}
The maximum on the left-hand side is attained by some deterministic $x\in L$ which we set to be $j$, for which we have, by the FKG inequality and self-duality, that 
\begin{align}\label{eq:two-arm-fkg}
\pi_B^1(\Gamma_j) &  \leq \pi_B^1\Big(j\longleftrightarrow \partial_\north R_j\mbox{ in $R_j^+$}\Big) \pi_B^1\Big(j+(\tfrac 12,0)\stackrel{\ast}\longleftrightarrow \partial_\north R_j \mbox{ in $R_j^+$}\Big) \nonumber \\ 
& = \pi_B^1\Big(j\stackrel{R_j^+}\longleftrightarrow \partial_\north R_j\Big)\pi_B^0(j\stackrel{R_j^+}\longleftrightarrow \partial_\north R_j)\,.
\end{align}
By the RSW estimate, Theorem~\ref{rmk:circuit-in-annulus}, we see that $\pi_{B}^0(\anncircuit(B-\bigcup_{j\in L} R_j))\geq \epsilon$ for some $\epsilon(q)>0$, and therefore by monotonicity in boundary conditions,
\begin{align*}
\pi_B^0 \Big(j\stackrel{R_j^+}\longleftrightarrow\partial_\north R_j\Big) \geq \epsilon \pi_B^1\Big(j\stackrel{R_j^+}\longleftrightarrow \partial_\north R_j\Big)\,.
\end{align*}
Plugging this in to~\eqref{eq:two-arm-fkg} implies that $\pi_B^0(j\longleftrightarrow \partial R_j) \geq 2\sqrt{\frac{c\epsilon}n}$. In order to complete the proof of~\eqref{eq:one-arm}, we translate by $-j$ to see that $\pi_{B-j}^0(0\longleftrightarrow \partial \llb -\frac n2,\frac n2\rrb^2) \geq c'n^{-\frac 12}$ for some $c'(q)>0$. Since $j\in L$, $B-j\subset R$ and by monotonicity, we deduce~\eqref{eq:one-arm}.

Going from~\eqref{eq:one-arm} to~\eqref{eq:two-point-corr} is a standard exercise in using RSW estimates (Theorems~\ref{thm:crossing-any-bc-q<4}--\ref{rmk:crossing-no-bc}) and the stitching arguments used in the proof of Theorem~\ref{thm:crossing-any-bc}; since both $x,y$ are macroscopically far from $\partial R$, we can use~\eqref{eq:one-arm} to connect each of them to some distance $O(\|x-y\|)$ away, and stitch open crossings to connect these two together via the FKG inequality and Theorems~\ref{thm:crossing-any-bc-q<4}--\ref{rmk:crossing-no-bc}, yielding the desired. 
\end{proof}

\begin{theorem}\label{thm:poly-lower-bound}
Let $q=3,4$ and consider the critical Potts model on $\Lambda=\llb 0,n\rrb ^2$ with boundary conditions $\eta$. The continuous-time Glauber dynamics has $\gap^{-1} \gtrsim cn$
for some $c(q)>0$, and the same holds on $(\mathbb Z/n\mathbb Z)^2$.
\end{theorem}

\begin{proof}
Now that we have a bound on connection probabilities macroscopically away from boundaries, we modify the lower bound of~\cite{LScritical} to our setting.
Fix any boundary condition $\eta$ (if we are considering the torus, reveal $\sigma \restriction_{\partial \Lambda}$ and fix that to be your boundary condition $\eta$). Let $\Lambda_1=\llb \frac n4,\frac{3n}4\rrb^2$ and $\Lambda_2 =\llb \frac n3,\frac {2n}3\rrb^2$. Recall the variational form of the spectral gap, Eq.~\eqref{eq:dirichlet-form}, and consider the test function $f(\sigma) = \sum_{x\in  \Lambda_2} \boldsymbol 1\{\sigma(x)=q\}$. 
Since $f$ is $1$-Lipschitz, $\cE(f,f)$ is easily seen to be $O(n^2)$. We now lower bound $\mbox{Var}_\mu(f)$. To do so, we move to the FK representation of the Potts model on $\Lambda$ via the Edwards--Sokal coupling using the event $\mathcal E_\eta$. For the remainder of this proof only, let $\mathbb E$ and $\mbox{Var}$ be with respect to the joint distribution over FK and Potts configurations given by the Edwards--Sokal coupling. By Theorem~\ref{rmk:circuit-in-annulus}, we see by the FKG inequality, 
\begin{align*}
\pi_{\Lambda}^\xi (\anncircuit^\ast(\Lambda-\Lambda_1) \mid \mathcal E_\eta) \geq \pi_\Lambda^1(\anncircuit^\ast(\Lambda-\Lambda_1))\geq c_1
\end{align*}
for some $c_1(q)>0$. 
By the law of total variance and the above, we see that 
\begin{align*}
\mbox{Var}_\mu(f) \geq \mathbb E \big[ \mbox{Var}(f\mid \omega \restriction_{\Lambda-\Lambda_1}) \boldsymbol1\{ \anncircuit^\ast(\Lambda-\Lambda_1\}\big]\geq
c_1 \mbox{Var}(f\mid \omega \restriction_{\Lambda-\Lambda_1}, \partial \Lambda_1 \nconn \partial \Lambda)\,.
\end{align*}
But given that $\partial \Lambda_1 \nconn \partial \Lambda$, the probability of $\sigma(x)=q$ for $x\in \Lambda_2$ is $1/q$; in particular, by the Edwards--Sokal coupling and FKG, we can expand the above as 
\begin{align*}
\mbox{Var}(f\mid \omega\restriction_{\Lambda-\Lambda_1}, \partial \Lambda_1 \nconn \partial \Lambda) & \geq q^{-1}\sum_{x,y\in \Lambda_2} \pi_\Lambda^\xi(x\longleftrightarrow y\mid \omega\restriction_{\Lambda-\Lambda_1}, \partial \Lambda_1 \nconn \partial \Lambda) \\
& \geq q^{-1} \sum_{x,y\in \Lambda_2} \pi_{\Lambda_1}^0 (x\longleftrightarrow y) \geq c_2 n^{3}
\end{align*}
for some $c_2(q)>0$, where the last inequality follows from Eq.~\eqref{eq:two-point-corr} of Lemma~\ref{lem:two-point-corr}.
\end{proof}

\section{Slow mixing at a discontinuous phase transition}\label{sec:mixing-q>4}

At a discontinuous phase transition, the dynamical behavior of the Potts model is expected to exhibit an exponential critical slowdown on the torus but otherwise depend on the choice of boundary conditions. We demonstrate this in the following sections.

\subsection{Exponential lower bound on the torus}\label{sub:q>4}
We will establish Theorem~\ref{mainthm-q>4}, showing that $\gap^{-1}\gtrsim \exp(c n)$ in the presence of multiple Gibbs measures at $p_c$ if the boundary conditions are periodic.
Consider $\Lambda=\Lambda_{n,n'}$ with $n'=\lfloor \alpha n\rfloor$ for a fixed aspect ratio $0< \alpha \le1$ and periodic boundary conditions. It suffices to prove the result in the case of Glauber dynamics for the FK model, which, thanks to Theorem~\ref{thm:Ullrich-comparison} (in particular, Eqs.~\eqref{eq-ullrich1}--\eqref{eq-ullrich2} and~\eqref{eq-ullrich2-realq}), implies all other cases of the theorem. Moreover, it suffices to prove the result in the discrete time setting, since any $O(n^2)$ cost of moving to continuous time can be absorbed in the exponential lower bound.

\subsection*{Proof idea}
For $q>4$, to obtain an exponential lower bound when there is a discontinuous phase transition, we establish a bottleneck in the state space consisting of vertical and horizontal crossings, each forming a loop around the torus. The basic idea is that a combination of a horizontal loop and a vertical loop in the torus can be translated to form a macroscopic wired circuit. Escaping such configurations via a pivotal edge would require a macroscopic dual-crossing inside the circuit, an event with an exponentially small probability (see Theorem~\ref{thm:DC-S-T-main}). Unfortunately, conditioning on the locations of these two loops includes negative information about the interior of the circuit and prevents us from appealing to the decay of correlations estimates. If we instead considered \emph{two} pairs of horizontal and vertical loops: one could expose the required wired circuit with no information on its interior. However, a subtler problem then arises, where after exposing one pair of loops (say the vertical ones), revealing the second (horizontal) pair might leave the potentially pivotal edge outside of the formed wired circuit, preventing us from estimating the probability of it being pivotal.
It turns out that using \emph{three} pairs of horizontal and vertical loops supports a suitable way of exposing a wired circuit
such that the potential edge is pivotal only if it supports a macroscopic dual-crossing within that circuit, thus leading to the desired lower bound.

\begin{proof}[\textbf{\emph{Proof of Theorem~\ref{mainthm-q>4}}}]
A standard technique for proving lower bounds on mixing times is constructing a set $S\subset \Omega$ that is a bottleneck for the Markov chain dynamics.
For a chain with transition kernel $P(x,y)$ and stationary distribution $\pi$, let the edge measure between $A,B\subset\Omega$ be
\[Q(A,B)=\sum _{\omega\in A} \pi(\omega)\sum_{\omega'\in B}P(\omega,\omega')\,,
\]
(see, e.g.,~\cite[Chapter 7]{LPW09}). Then the conductance/Cheeger constant of $\Omega$ is
\begin{align}\label{eq:Cheeger-constant}
\Phi=\max_{S\subset \Omega}\frac {Q(S,S^c)}{\pi(S)\pi(S^c)}\,.
\end{align}
and the following relation between $\Phi$ and the spectral gap of the chain~\cite{LPW09} holds:
\begin{align}\label{eq:Cheeger-inequality}
2\Phi\geq \gap \geq  {\Phi^2}/{2}\,.
\end{align}

By the dual version of  Theorem~\ref{thm:DC-S-T-main}, there exists some $c(q)>0$ such that
\begin{align}\label{eq:exp-decay}
\pi^1_{\mathbb Z^2}\left ((0,0)\overset{\ast}\longleftrightarrow\partial\llb -n,n\rrb^2\right )\leq e^{-cn}\,.\end{align}
Using~\eqref{eq:exp-decay}, we will establish a bottleneck set $S$ for the random cluster model on $\Lambda$ with periodic boundary conditions. Define the bottleneck event
\[S=\bigcap _{i=1,2,3} S^i_v\cap S^i_h
\] where the constituent events are defined as follows for $i=1,2,3$:
\begin{align*}
S_v^i&:=\left\{\omega:\exists x\mbox{ such that }(x,0)\longleftrightarrow (x,n')\mbox{ in $\llb \frac {(i-1)n}3,\frac {i n}3\rrb \times \llb 0,n'\rrb $}\right\}\,, \\
S_h^i&:=\left\{\omega:\exists y\mbox{ such that }(0,y)\longleftrightarrow (n,y)\mbox{ in $\llb 0,n\rrb \times\llb \frac {(i-1)n'}3,\frac {in'}3\rrb $}\right\}\,.
\end{align*}

The crossings in the above events are all loops on the torus of homology class $(1,0)$ and $(0,1)$. We aim to get an exponentially decaying upper bound on  $\pi^p_{\Lambda}(\partial S\mid S)$ (the superscript $p$ denotes periodic boundary conditions on $\Lambda$), where the boundary subset $\partial S=\{\omega\in S:P(\omega,S^c)>0\}$ is the event that there exists an edge $e$ that is pivotal to $S$. Specifically,
\[\partial S=\{\omega\in S\;:\;\omega-\{e\} \notin S\mbox{ for some }e\in\omega\}\,,
\]
in which configurations $\omega$ are identified with their edge-sets.
The bound $P(S,S^c)\leq 1$ implies $Q(S,S^c)/\pi^p_{\Lambda}(S)\leq \pi_{\Lambda}^p(\partial S\mid S)$. We now control $\pi_{\Lambda}^p(S)$ to express the conductance in terms of $\pi_{\Lambda}^p(\partial S\mid S)$.

Via the symmetry of periodic FK boundary conditions, RSW estimates on the torus $(\mathbb Z/n\mathbb Z)\times (\mathbb Z/{\alpha n \mathbb Z})$ were proved in~\cite{BeDu12} 
for all $q\geq 1$ at $p_c(q)$. By \cite[Theorem 5]{BeDu12}, therefore, there exists some $\rho(\alpha,q)>0$ such that $\pi_{\Lambda}^p(\cC_v^\ast(\llb \tfrac n3,\tfrac{2n}3\rrb \times \llb \tfrac {n'}3,\tfrac{2n'}3\rrb )\geq \rho$, and 
\[\frac {Q(S,S^c)}{\pi_\Lambda^p(S)\pi_\Lambda^p(S^c)}\leq \frac{Q(S,S^c)}{\pi_\Lambda^p(S)\rho} \leq \rho^{-1}\pi_{\Lambda}^p{(\partial S\mid S)}\,.
\]
Combining this with Eqs.~\eqref{eq:Cheeger-inequality} and~\eqref{eq:lower-bound}, it suffices to prove that there exist constants $c_1=c_1(\alpha,q)>0$ and $c_2=c_2(\alpha,q)>0$ such that,
\begin{align}\label{eq:lower-bound}
\pi_{\Lambda}^p(\partial S \mid S)\leq c_1 e^{-c_2 n}\,.
\end{align}
to obtain the desired bound on the inverse spectral gap. A union bound implies
\[\pi_{\Lambda}^p(\partial S\mid S)\leq \sum _e \pi^p_{\Lambda}(\omega - \{e\} \notin S \mid \omega \in S)\,,
\] where the sum is over all edges. We also union bound over whether $e$ is pivotal to $S_v^i$ or $S_h^i$ for $i=1,2,3$ (see Fig.~\ref{fig:lower-bound} for an illustration of a configuration in $S$).

\begin{figure}

  \centering
  \includegraphics[width=0.75\textwidth]{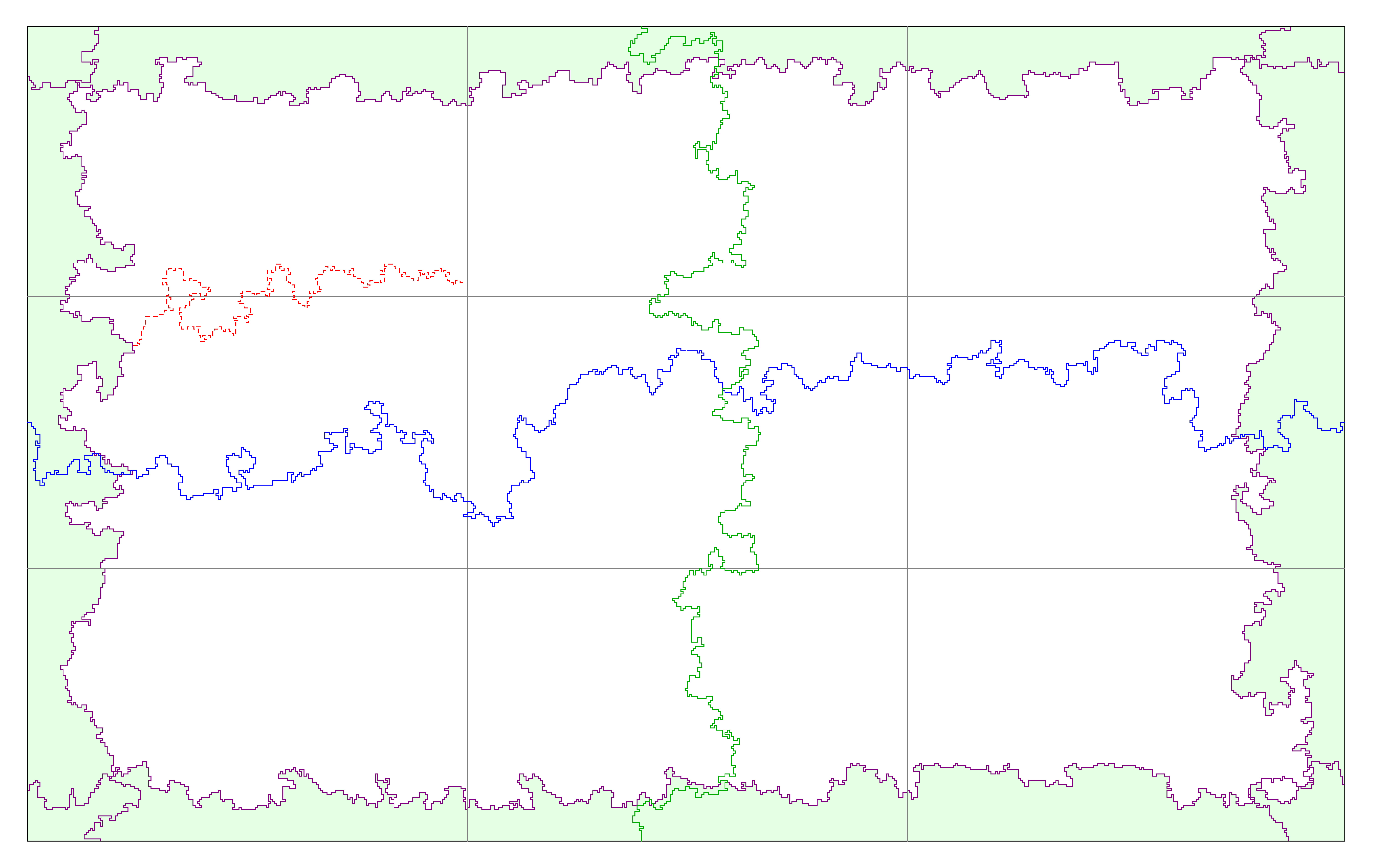}
  \vspace{-0.1cm}
  \caption{The lower bound construction: the box $\Lambda'$ is equipartitioned into three vertical and three horizontal strips. The purple crossings form an outermost open circuit $C$. A boundary configuration $\omega\in\partial S$ contains a macroscopic dual-crossing like the red path extending from the pivotal edge $e$ to the boundary of the vertical strip.}
  \label{fig:lower-bound}
\end{figure}

Without loss of generality, examine the probability that $e$ is pivotal to $S_v^1$. The other cases can be treated analogously. Fix an edge $e$ in $\llb 0,\frac {n}3\rrb \times \llb 0,n'\rrb$
and consider its horizontal coordinate
 (in Fig.~\ref{fig:lower-bound}, $e$ is the edge of intersection of the purple and red paths): the edge $e$ is either closer to the left side, or closer to the right side of the vertical strip and is either in the top, middle, or bottom third of the vertical strip (if there is ambiguity in these choices, choose arbitrarily).

Now move to a translate of $\Lambda$ on the universal cover of the torus, which we call $\Lambda'=\Lambda'_{n,n'}$. Choose a translate such that:
\begin{enumerate}
\item The horizontal third of $\Lambda$ that contained $e$, is now the middle horizontal third of $\Lambda'$.
\item If $e$ was closer to the left of its vertical strip than the right, that strip is the left vertical third of $\Lambda'$. Otherwise, that strip is the right vertical third of $\Lambda'$.
\end {enumerate}

Because we work with periodic boundary conditions, $\pi^p_{\Lambda'}=\pi^p_{\Lambda}$.
Begin by exposing all the edges of $\partial \Lambda'$ so as to fix a boundary condition and move from a torus to a rectangle with some fixed boundary condition.

Then expose the outermost circuit $C$ in $\Lambda'$ as follows. First expose the left-most vertical crossing of $\Lambda'$ by revealing the dual component of $\partial_\west \Lambda'$: by construction, its adjacent primal edges will form the left-most vertical crossing of $\Lambda'$ and we will not have revealed any edges to its right. Repeating this procedure on the $\north,\east,\south$ sides reveals the outermost circuit $C$ without exposing any of its interior edges (see Fig.~\ref{fig:lower-bound} where the shaded region consists of the edges we reveal).

If $e$ is not an open edge in one of the vertical crossings we have exposed, then certainly $e$ is not pivotal to the event $S$. So suppose $e$ is an open edge in a vertical loop and---by construction---also an open edge in $C$. Denote by $C^o$ the set of edges that have not yet been exposed, i.e., all edges interior to $C$.

It is a necessary condition for $e$ to be pivotal to $S$ that there exists a dual path from an interior dual-neighbor of $e$ to the inner (to $C$) vertical boundary of the vertical third containing $e$ in $\Lambda'$. If there is no such dual-path, then there is a different primal vertical crossing of that strip which does not contain $e$, and that crossing---because of the exposed horizontal loops---is itself a loop of homology class $(0,1)$.

But observe that the inner boundary of the vertical strip is a graph distance at least $n/6$ from $e$ because we chose $\Lambda'$ such that $e$ is farther from the inner boundary of the strip than the outer one. Also note that such a dual-crossing event is a decreasing event. By the Domain Markov property and monotonicity,
\[\pi^p_{\Lambda'}(\omega \restriction_{C^o} \in \cdot \mid S)=\pi^1_{C\cup C^o}(\cdot\mid S)\succeq \pi^1_{C\cup C^o}\succeq \pi^1_{\mathbb Z^2}(\omega\restriction_{C^o} \in \cdot)\,.
\]

By~\eqref{eq:exp-decay}, the probability of $e$ being dual-connected to the inner vertical boundary of the vertical third it is in, under $\pi^1_{\mathbb Z^2}$, is less than $e^{-cn/6}$ where $c(q)>0$ is from~\eqref{eq:exp-decay}. Thus, there exists an absolute $c(q)>0$ such that, for any fixed $e$,
\[\pi^p_{\Lambda}(\omega-\{e\}\notin S_v^1\mid \omega \in S)\leq e^{-cn}\,.
\]
If $e$ were pivotal to $S_h^i$ the analogous claim would hold with probability less than $e^{-cn'/6}$. Summing over all six crossing events and summing over all edges $e$ we conclude that
$\pi^p_{\Lambda}(\partial S\mid S)\leq 12\alpha n^2e^{-c\alpha n}$, implying Eq.~\eqref{eq:lower-bound} as desired.
\end{proof}

\subsection{Exponential lower bound at \texorpdfstring{$\beta>\beta_c(q)$}{beta>beta\_c}} Modifying the proof of Theorem~\ref{mainthm-q>4} to the setting of Potts Glauber dynamics at $\beta>\beta_c(q)$ allows us to prove slow mixing of the Potts Glauber dynamics at all low temperatures and all choices of $q>1$.  The key difference is we can no longer work directly with the FK dynamics, because it is---by duality to high temperature---fast for all $p>p_c(q)$ (see, e.g.,~\cite{Ul14,BlSi15}).

\begin{proof} [\textbf{\emph{Proof of Theorem~\ref{mainthm-low-temp}}}]
In the context of this proof, denote by $x\leftrightsquigarrow y$ the existence of a sequence of sites $\{x_i\}_{i=1}^k$ such that $x_i$ is adjacent to $x_{i+1}$ and with $x_1=x$ and $x_k = y$, with $\sigma(x_i)=\sigma(x_{i+1})$ for all $i$.
Define the bottleneck set $S=\bigcap _{i=1,2,3} S_{v,q}^i\cap S_{h,q}^i$,
\begin{align*}
S_{v,q}^i=\{\sigma:\exists x \mbox{ such that } (x,0)\leftrightsquigarrow (x,n) \mbox{ in } \llb \tfrac {(i-1)n}3,\tfrac {in}3\rrb \times \llb 0,n\rrb,\sigma(x)=q\}\,, \\
S_{h,q}^i=\{\sigma:\exists y \mbox{ such that }  (0,y)\leftrightsquigarrow (n,y) \mbox{ in } \llb 0,n\rrb \times \llb \tfrac{(i-1)n}3,\tfrac {in}3\rrb,\sigma(y)=q\}\,.
\end{align*}

We call each of these Potts connections of nontrivial homology on the torus Potts loops. We obtain an exponentially decaying upper bound on $\mu^p_\Lambda (\partial S\mid S)$ (where $p$ denotes periodic boundary) by examining the pivotality of vertices to $S$. As before, we union bound over all vertices in $\Lambda$ and the six different crossing events: fix a vertex $v$ whose pivotality to---without loss of generality---$S_{v,q}^1$ we examine. We choose a translate of $\Lambda$ on the universal cover, $\Lambda'$, according to the same rules as for the FK model, with $e$ replaced by $v$, where by periodicity of the boundary conditions, $\mu_\Lambda^p\stackrel{d}=\mu_{\Lambda'}^p$.

In order to bound the probability of $\sigma(v)$ being pivotal to $S_{v,q}^1$, we reveal the outermost Potts loops of color $q$ in $\Lambda'$ as follows. First reveal the spin values on $\partial \Lambda'$ to reduce the torus to a rectangle with fixed boundary conditions. Then reveal, starting from $\partial \Lambda'$, all spins $\star$-adjacent (either adjacent or diagonal to) to vertices whose spin value is not $q$. By construction, we will have revealed the outermost $q$-colored paths, and therefore a Potts circuit, $C^q$, of spin value $q$, and nothing interior to it.

If $v$ is pivotal to $S_{v,q}^1$ it must be the case that $v\in C^q$, so we suppose that $v\in C^q$. In order to obtain bounds on its pivotality, we now move to the FK representation of the Potts model inside $C^q$. By our definition of $C^q$ and the Edwards--Sokal coupling on bounded domains, the FK representation of the Potts region we have not yet revealed has fully wired boundary conditions, and therefore the event $\mathcal E_{\sigma\restriction_{C^q}}$ is trivially satisfied.

We claim that in order for $v\in C^q$ to be pivotal to $S_{v,q}^1$, there must be a dual-crossing from one of the interior dual-neighbors of $v$ (an edge whose center is distance $\frac {\sqrt 5}2$ from $v$) to one of $\{\tfrac n3\} \times \llb 0,n\rrb$ or $\{\tfrac {2n}3\}\times \llb 0,n\rrb$ in $\Lambda'$ (the choice depends on which translate $\Lambda'$ is chosen), and it must be contained completely within $C^q$. Suppose there were no such dual-crossing. Then there would have to be a primal FK connection in the same vertical third connecting the two neighbors of $v$ in $C^q$. By the Edwards--Sokal coupling and the definition of Potts connections, such an FK connection would translate to a new Potts connection of color $q$ that does not use the vertex $v$. Like for the FK model, the new Potts crossing is still a vertical loop because of the exposed horizontal crossings.

Since the event $\mathcal E_{\sigma\restriction_{C^q}}$ is trivially satisfied, monotonicity in boundary conditions, combined with the exponential decay of dual-connections under $\pi_{\mathbb Z^2}$ whenever $p>p_c$, implies that the probability of such a macroscopic dual-crossing is bounded above by $e^{-cn}$ for some $c=c(\beta,q)>0$. By spin flip symmetry of the torus, and our definition of Potts loops, $\mu_{\Lambda'}^p (S)\leq \frac 12$. Using $S$ as the bottleneck in~\eqref{eq:Cheeger-inequality}, we conclude that for every $\beta>\beta_c(q)$, there exists $c=c(\beta,q)>0$ such that $\gap^{-1}\gtrsim \exp(cn)$.
\end{proof}

\section{Upper bounds under free boundary conditions} \label{sec:upper-swendsen-wang}
We now prove Theorem~\ref{mainthm:large-q-upper}, showing that the dynamics for the critical Potts model in the phase coexistence regime should be sensitive to the boundary conditions. 
We prove the desired upper bound for Swendsen--Wang dynamics on $\Lambda$ with free or monochromatic boundary conditions using censoring inequalities (Theorem~\ref{thm:censoring}). The monotonicity requirement of the censoring prevents us from carrying this out in the setting of the Potts Glauber dynamics.  
We thus restrict our analysis now to the FK Glauber dynamics, a bound on which would imply the analogous bounds for Chayes--Machta and Swendsen--Wang via Theorem~\ref{thm:Ullrich-comparison}. As in \cite{MaTo10}, we will work with distributions over boundary conditions induced by infinite-volume measures; this is more delicate in the setting of the FK model where boundary interactions are no longer nearest-neighbor. We now formally define these distributions.

\begin{definition}[``free"/``wired" boundary conditions] \label{def:boundary-condition} In order to sample a boundary condition on $\Delta\subset \partial \Lambda$ from $\pi^{0}_{\mathbb Z^{2}}$ or $\pi^{1}_{\mathbb Z^2}$ we sample the infinite-volume configuration on $\mathbb Z^2 -\Lambda^0$ and then identify the induced boundary condition with the partition of $\Delta$ that induced by that configuration. A measure $\bP$ over boundary conditions is called ``wired" if $\bP \succeq \pi_{\mathbb Z^2}^1$ (resp., ``free" if $\bP \preceq \pi_{\mathbb Z^2}^0$), i.e., it dominates the measure over boundary conditions induced by $\pi_{\mathbb Z^2}^1$.
When sampling boundary conditions on $A,B\subset \partial \Lambda$ according to different measures, we do so sequentially clockwise from the origin\footnote{For our purposes, $A,B$ will be edge disjoint (e.g., different sides of $\Lambda$) so that the order irrelevant.}.
\end{definition}

We prove the following proposition, from which Eq.~\eqref{eq:large-q-upper-1} of Theorem~\ref{mainthm:large-q-upper} follows easily.

\begin{proposition}\label{prop:all-free} Let $q$ be sufficiently large and consider Glauber dynamics for the critical FK model on $\Lambda=\Lambda_{n,n}$. Let $\bP$ be a distribution over boundary conditions on $\partial \Lambda$ such that $\bP\preceq \pi_{\mathbb Z^2}^0$ or $\bP\succeq \pi_{\mathbb Z^2}^1$ and let $\bE$ be its corresponding expectation.  Then for every $\epsilon>0$, there exists $c(\epsilon,q)>0$ such that for $t_\star=\exp(cn^{3\epsilon})$,
\[\max_{\omega_0\in\{0,1\}}\bE\left[\|P^{t_\star} (\omega_0,\cdot)-\pi_{\Lambda}^\xi\|_{\tv}\right]\lesssim e^{-cn^{2\epsilon}}\,,
\]
where $\xi\sim \bP$. In particular,
$\bP(\tmix\gtrsim t_{\star})\lesssim \exp(-cn^{2\epsilon})$.
\end{proposition}

\subsection*{Proof ideas}
In order to prove Proposition~\ref{prop:all-free}, we appeal to the Peres--Winkler censoring inequalities~\cite{PW13} for monotone spin systems, a crucial part of the analysis in~\cite{MaTo10} (then later in~\cite{LMST12}) of the low-temperature Ising model under ``plus" boundary, a class of boundary conditions that dominate the plus phase. A major issue when attempting to adapt this approach to the critical FK model at sufficiently large $q$ with ``free" boundary conditions is that the typical ``free" boundary conditions still have many boundary connections, inducing problematic long-range interactions along the boundary (see Fig.~\ref{fig:long-range-bc}), which prevent coupling beyond interfaces. 

To remedy this, at every step of the analysis we modify the boundary conditions to all-free on appropriate segments of length $n^{o(1)}$ (this modification can only affect the mixing time by an affordable factor of $\exp(n^{o(1)})$), and with high probability no boundary connections circumvent the interface past the modified boundary by the exponential decay of correlations in the ``free" phase. Refined large deviation estimates on fluctuations of FK interfaces then allow us to control the influence of other long-range boundary interactions (see Figure~\ref{fig:recursion-1nd-step}) and to couple different configurations beyond distance $n^{1/2+o(1)}$ (the length scale that captures the normal fluctuations of the interface),
 yielding mixing estimates on $n\times n^{1/2+o(1)}$ boxes, the basic building block of the proof.

\subsection{FK boundary modifications}
Let 
\[d_{\omega_0}^{\xi}(t) = \|P^t(\omega_0,\cdot)-\pi^\xi\|_\tv\,,
\]
and for a pair of FK boundary conditions $\xi,\xi'$, with mixing times $\tmix,\tmix'$, define \[M_{\xi,\xi'}=\| \pi^\xi / \pi^{\xi'}\|_\infty \;\vee\; \|\pi^{\xi'}/\pi^{\xi}\|_\infty\,.\]
It is easy to verify (see~\cite[Lemma~2.8]{MaTo10}) that for some $c$ independent of $n,\xi,\xi'$,
\begin{align}\label{eq:tmix-M}
\tmix\leq cM_{\xi,\xi'}^3 |E|\tmix'
\end{align}
(indeed, in the variational characterization of the spectral gap, the Dirichlet form, expressed in terms of local variances, produces a factor of $M_{\xi,\xi'}^2$, and the variance produces another factor of $M_{\xi,\xi'}$).

\begin{definition}[boundary modification] \label{def:modification}
If $\bP$ is a distribution over boundary conditions on $\partial \Lambda$, $\Delta\subset \partial \Lambda$, we let $\bP^{\Delta}$ be the distribution which samples boundary conditions $\xi\sim\bP$ and modifies them as follows: if $\xi$ corresponds to the partition $\mathcal P_1,...,\mathcal P_k$ of $\partial \Lambda$, then $\xi'=\xi^{\Delta}$ is given by the partition $\mathcal P_1-V(\Delta),...,\mathcal P_k-V(\Delta),V(\Delta)$; this induces a coupling of $(\xi,\xi')\sim (\bP,\bP^{\Delta})$. E.g., if $\Delta=\partial \Lambda$ then $\xi'=0$, and if $\Delta$ consists of a single vertex $v$ and $\xi$ is induced by a configuration where every boundary vertex is connected to $v$ and to no other boundary vertex, then  $\xi'$ would be wired on $\partial \Lambda -\{v\}$.
\end{definition}

\begin{lemma}  \label{lem:bc-perturbation}
Denote a pair of boundary conditions obtained from the coupling $(\bP,\bP^{\Delta})$ given by Definition~\ref{def:modification} by $(\xi,\xi')$. Then,
\begin{align}\label{eq:bc-pert-1}
M_{\xi,\xi'}\leq q^{|V(\Delta)|}=: M_{\Delta}\,,
\end{align}
and consequently, for some universal $c>0$, if $t'=ct|E|^{-2}M_{\Delta}^{-4}$,
\begin{align}\label{eq:bc-pert-2}
\bE[d_1^{\xi}(t)\vee d_0^{\xi} (t)]\leq e^{-M_{\Delta}}+8\bE[d^{\xi'}_1(t')\vee d^{\xi'}_0(t')]\,.
\end{align}
\end{lemma}

\begin{proof} By inspection, one sees that the addition of at most $|V(\Delta)|$ boundary clusters can increase the total number of clusters by at most $|V(\Delta)|$. By definition of the FK model, every additional cluster receives a weight of $q$.

In order to prove~\eqref{eq:bc-pert-2}, begin with the observation that by~\eqref{eq:tmix-M},
\[\bE[d_1^{\xi}(t)\vee d_0^{\xi}(t)]\leq e^{-M_{\Delta}}+\bP(\tmix \geq t/M_{\Delta}) \leq e^{-M_{\Delta}}+\bP(\tmix'\geq |E|t')\,.
\]
For $\tmix'\geq s$, there must exist an $\omega_0$ such that $d^{\xi'}_{\omega_0}\geq 1/(2e)$. But by Eq.~\eqref{eq:init-config-comparison}, $d^{\xi'}_{\omega_0}(s)\leq 2|E|(d^{\xi'}_1(s)\vee d^{\xi'}_0(s))$,
which implies that
\[\bP(\tmix'\geq |E| t')\leq \bP\bigg(d^{\xi'}_1(|E|t')\vee d^{\xi'}_0(|E|t')\geq (4e|E|)^{-1}\bigg)\,.
\]
As a consequence of Theorem~\ref{thm:censoring}, it is well-known (see~\cite[Corollary~2.7]{MaTo10}) that
\[
d_{1}^{\xi'}(t)\vee d_0^{\xi'}(t)\leq \left( 4(d_{1}^{\xi'}(t_0)\vee d_0^{\xi'}(t_0))\right)^{\lfloor t/t_0 \rfloor}\,,
\] and thus
\[\bP\left(d_{1}^{\xi'}(|E|t')\vee d_0^{\xi'}(|E|t')\geq (4e|E|)^{-1}\right)\leq 8\bE[d_{1}^{\xi'}(t')\vee d_0^{\xi'}(t')]\,. \qedhere
\]
\end{proof}

\subsection{Necessary equilibrium estimates}
We now include some equilibrium estimates from cluster expansions at sufficiently large $q$, which are adaptations of the necessary low-temperature Ising equilibrium estimates of~\cite{DKS} to the setting of the critical FK model (cf.\ Proposition~\ref{prop:surface-tension}).

For any $\phi\in (-\pi/2,\pi/2)$, the strip $\mathcal S_n=\llb 0,n\rrb \times \llb -\infty,\infty \rrb $ has $(1,0,\phi)$ boundary conditions denoting wired on $\partial \mathcal S_n  \cap\llb 0,n\rrb\times \{y\leq x\tan \phi\}$ and free elsewhere. Then $\Gamma$ is the set of all order-disorder interfaces (bottom-most dual crossings from $(0,0)\longleftrightarrow (0,n\tan \phi)$). We define the \emph{cigar-shaped region} for $d,\kappa>0$ by
\[U_{\kappa,d,\phi}=\mathcal S_n\cap \left\{ (x,y):|y-x\tan \phi| \leq d\left|\tfrac {x(n-x)}n\right|^{\frac 12+\kappa}\right\}\,, 
\]
and call $\Gamma_{\kappa,d,\phi}^r\subset \Gamma$ the set of interfaces that are contained in the cigar-shaped region.

\begin{proposition} \label{prop:fk-415} Consider the critical FK model on $\mathcal S_n$ and fix a $\delta>0$. There exists some $q_0$ such that for all $\phi\in [-\frac \pi 2+\delta ,\frac \pi 2-\delta ]$, there exists  $c(d,\phi)>0$ such that for every $q\geq q_0$ and every $\kappa>0$,
\[\pi_{\mathcal S_n}^{1,0,\phi}(\Gamma^r_{\kappa,d,\phi})\lesssim n^2 \exp(-cn^{2\kappa})\,.
\]
\end{proposition}
\begin{proof}
We use an extension of~\cite[\S4]{DKS} to the framework of the FK/Potts models in the phase coexistence regime by~\cite{MMRS91}. The specific case of~\cite[\S5]{MMRS91} states the following: consider the critical FK model on the strip $\mathcal S_n=\llb 0,n\rrb \times \llb -\infty, \infty \rrb$ with $(1,0,\phi)$ boundary conditions. Then for every $q\geq q_0$ and every $d,\kappa>0$, and $\phi \in (-\frac \pi2, \frac \pi 2)$,
\[\lim_{n\to\infty} \frac 1n \log \frac {\sum_{\mathcal I\in\Gamma^r_{d,\kappa,\phi}}\pi_{\mathcal S_n}^{1,0,\phi} (\mathcal I)}{\sum_{\mathcal I\in\Gamma}\pi_{\mathcal S_n}^{1,0,\phi} (\mathcal I)}=0\,.
\]
A straightforward adaptation of the proof of~\cite[Proposition~5]{MMRS91} to the form of~\cite[Proposition~4.15]{DKS} in fact yields the very large deviation bound we desire. A specific case is when $\phi=0$: there exist some $q_0,\tilde c>0$ such that for all $q\ge q_0$ and every $a\ge 0$,
\begin{align}\label{eq:strip-ldp}
\pi_{\mathcal S_n}^{1,0,\phi=0}(|\bar H| \ge a) \lesssim n^2\exp(-\tilde c a^2/n)\,,
\end{align}
where $|\bar H|$ denotes the maximum vertical distance of an edge $e$ in the interface to the $x$-axis. Though the result of~\cite{MMRS91} is written with the Potts model in mind and thus with integer $q$, the cluster expansion and all the results hold with noninteger $q$ as well.
\end{proof}

We now prove the following estimate on FK interfaces near a repulsive boundary.

\begin{proposition} \label{prop:surface-tension}
Fix $c>0$ and consider the critical FK model on $R=\llb 0,n\rrb \times \llb 0, \ell \rrb $ for $c\sqrt {n\log n}\le \ell \le n$ with boundary conditions $(1,0)$ denoting $1$ on $\partial_{\north,\east,\west} R$  and $0$ on $\partial_\south R$. Let $H$ be the maximum vertical height of the bottom-most horizontal open crossing. There exist constants $q_0, c'>0$ such that, for all $q\ge q_0$ and every $0\leq a\leq \ell$,
\[\pi^{1,0}_R (H\ge a)\lesssim n^2\exp\left(-c'\, a^2 / n\right)\,.
\]
\end{proposition}

\begin{proof}
Fix the $a\leq \ell$ from the statement of Proposition~\ref{prop:surface-tension}. Denote by $(1,0,\ast)$ boundary conditions that are still wired (resp.\ free) on the intersection of $\partial \mathcal S_n$ and the upper (resp.\ lower) half plane, but now also free on all of $\mathcal S_n\cap \{(x,y): y\leq -a/2\}$ (this induces a free bottom boundary on the semi-infinite strip starting from $y=-a/2$).

Clearly  $\pi^{1,0}_{\mathcal S_n}\succeq\pi^{1,0,\ast}_{\mathcal S_n}$. Then by Eq.~\eqref{eq:strip-ldp}, there exists $A(q)>0$ such that with probability bigger than $1-An^2\exp(-\tilde c a^2/4n)$, the interface---bottom-most open crossing connecting $(0,0)$ to $(0,n)$---under $\pi_{\mathcal S_n}^{1,0}$ does not touch the line $y=-a/2$, and therefore, with that probability, there is a horizontal dual-crossing of $\mathcal S_n$ contained entirely above $y=-a/2$. Via the grand coupling, since this is a decreasing event, the same horizontal dual-crossing would be present under $\pi^{1,0,\ast}_{\mathcal S_n}$ and we expose the bottom most  horizontal dual-crossing above the line $y=-a/2$ and couple the configurations above it.

At the same time, using~\eqref{eq:strip-ldp}, we have that under $\pi^{1,0}_{\mathcal S_n}$, with the same probability, the maximum  $y$-coordinate of the interface does not exceed $a/2$. If we have coupled the two configurations above a bottommost  dual-crossing above the line $y=-a/2$, the same would be true of the interface under $\pi^{1,0,\ast}_{\mathcal S_n}$. Thus, taking a union bound,
\[\pi^{1,0,\ast}_{\mathcal S_n}(H\ge a)\leq 2An^2\exp(-\tilde c a^2/4n)\,.
\]

Using the monotonicity of the FK model, and denoting by $R'$ the vertical translate of $R$ by $-a/2$, we obtain $\pi^{1,0}_{R'}\succeq \pi^{1,0,\ast}_{\mathcal S_n}(\omega \restriction_{R'})$; together with the fact that $\{H\ge a\}$ is a decreasing event, for $c'=\tilde c/4$,
\[\pi^{1,0}_R(H\ge a)\leq 2An^2\exp(-c' a^2/n)\,,
\]
where now, the interface is again between $(0,0)$ and $(0,n)$.
\end{proof}

Now for any fixed $\epsilon>0$, consider the rectangle $V_\ell=\llb 0,n \rrb \times \llb 0, \ell \rrb$ with $\ell\geq 4n^{\frac 12 +\epsilon}$, and $(1,0,\Delta)$ denoting free boundary conditions on $\partial V\cap \{(x,y):y\geq 2n^{\frac 12+\epsilon}\}$ and $\Delta=\{0\}\times \llb \frac n2 -n^{3\epsilon},\frac n2+n^{3\epsilon}\rrb$ and wired elsewhere. Denote the four points at which the boundary conditions change by $(w_1,w_2)\in \partial_{\west} V\times \partial_{\east} V$, $z_1,z_2\in \partial_\south V$ with $z_1$ to the left of $z_2$. Let $C_1$ and $C_2$ be the blocks $\llb 0,\frac n2\rrb\times \llb 0,\ell\rrb $ and $\llb \frac n2\rrb \times \llb 0,\ell \rrb$ respectively. The main equilibrium estimate we use in the sequel reads as follows.

\begin{proposition} \label{prop:D-estimate} Let $\Gamma=\{\omega:w_i\stackrel{C^\ast_i}\longleftrightarrow z_i,i=1,2\}$.
 There exists $q_0>0$ so that the following holds. For every $q\geq q_0$ there exists $c=c(q)>0$ such that the corresponding critical FK model on $V_\ell$ satisfies
\[\pi^{1,0,\Delta}_{V_\ell}(\Gamma^c)\lesssim e^{-cn^{3\epsilon}}\,.
\]
\end{proposition}

\begin{proof}This corresponds to Claim~3.10 proven in the appendix of~\cite{MaTo10} for $\ell=n^{\frac 12+\epsilon}$ in the setting of the low-temperature Ising model using the cluster expansion of~\cite{DKS} and the analogues of Propositions~\ref{prop:fk-415} and~\ref{prop:fk-416} (Propositions 4.15 and Theorem~4.16 of~\cite{DKS} respectively), but the extension to larger $\ell$ is immediate. We sketch the proof of~\cite{MaTo10} before justifying its extension to the current setting.

\begin{enumerate}[(a)]
\item Via a surface tension estimate analogous to~\eqref{eq:surface-tension}, with probability $1-\exp(-cn^{3\epsilon})$, the interfaces connect $w_i\longleftrightarrow z_i$ instead of $w_1\longleftrightarrow w_2$. In order to then claim that the two interfaces are confined to the right and left halves of $V$, it is shown in the appendix of~\cite{MaTo10} that with high probability the two interfaces do not interact, via the exponential decay of the Ising cluster weights in cluster lengths.

\item The next step in~\cite{MaTo10} was to show that the interface does not deviate farther than $n^{\epsilon}$ from the east of $z_1$. The complication in the Ising setup was that the plus boundary on $\partial_\south V$ produced a repulsive force on the interface so neither Proposition~4.15 nor Theorem~4.16 of~\cite{DKS} were directly applicable.

\item  To circumvent the problem that $\partial_\south V$ is not sufficiently far from $z_1$ to contain the cigar-shaped region, the region $V$ is extended in~\cite{MaTo10} to $V\cup \llb 0,n/2-n^{3\epsilon}\rrb\times \llb -n,0\rrb$ with appropriate boundary conditions. Thereafter, the proof concludes by lower bounding the weight of all interfaces between $w_1$ and $z_1$ that do not interact with the extension of $V$, repeatedly using Theorem~4.16 of~\cite{DKS}: the key to this lower bound consists of stitching cigar shaped regions of increasing length, all sufficiently far from the extension of $V$ and lying above the straight line connecting $w_1$ to $z_1$.

\item If the extension is accounted for, an estimate of the form of Theorem~4.16 of~\cite{DKS} with the appropriate angle $\phi$ implies that with high probability, the interface does not deviate far to the east of $z_1$, thus stays bounded away from $\partial_\east C_1$.
\end{enumerate}
For more details on these arguments, see~\cite[Appendix~A]{MaTo10}. In the setting of the FK model, the central cluster expansion estimate we require is the following (Proposition~\ref{prop:fk-416}), the FK analogue of Theorem~4.16 whose proof is a direct adaptation of the proof in \cite{DKS} of Theorem~4.16 using the FK cluster expansion techniques of ~\cite{MMRS91}.

\begin{definition}For an angle $\phi\in [-\frac \pi 2+\delta, \frac \pi 2 -\delta]$, define an \emph{edge cluster weight function} $\Phi(\mathcal C,\mathcal I)$ as a function with first argument that is a connected set of bonds in $\mathcal S_n$ and second argument that is an FK interface connecting $(0,0)$ to $(n,n\tan \phi)$, satisfying
\begin{enumerate}
\item $\Phi(\mathcal C,\mathcal I)=0$ when $\mathcal C\cap \mathcal I=\emptyset$\,,
\item $\Phi(\mathcal C,\mathcal I_1)=\Phi(\mathcal C,\mathcal I_2)$ when $\Pi_{\mathcal C} \cap \mathcal I_1=\Pi_{\mathcal C}\cap \mathcal I_2$\,,
\item $\Phi(\mathcal C,\mathcal I)=\Phi(\cC+(0,s),\mathcal I_1)$ when $\mathcal I_1\cap \Pi_{\mathcal C}=(0,s)+\mathcal I\cap \Pi_{\mathcal C}$\,,
\item $\Phi(\mathcal C,\mathcal I)\leq \exp(-\lambda d(\mathcal C))$\,.
\end{enumerate}
where $\lambda>0$, $d(\mathcal C)$ is the length of the shortest connected subgraph of $\mathbb Z^2$ containing all boundary edges of $\mathcal C$, and $\Pi_{\mathcal C}=\{(x,y)\in\mathbb Z^2:\exists y' \mbox{ s.t.\ }(x,y')\in \mathcal C\}$.

The \emph{FK order-disorder weight function} is a specific choice of $\Phi$ that gives rise to the FK distribution on wired-free interfaces, i.e.,
\[\pi_{\mathcal S_n}^{1,0,\phi} (\mathcal I)=\lambda^{|\mathcal I| + \sum_{\mathcal C\cap\mathcal I \neq \emptyset} \Phi(\mathcal C,\mathcal I)}\,,
\] and is given explicitly by Proposition~5 of~\cite{MMRS91}.
\end{definition}

\begin{proposition} \label{prop:fk-416} Consider the critical FK model on a domain $\bar V_n\supset U_{\kappa,d,\phi}$ and let $\Phi$ be the FK order-disorder weight function. Let $\tilde \Phi$ be any function satisfying
\[\tilde \Phi(\mathcal C,\mathcal I)=\Phi(\mathcal C,\mathcal I) \mbox{ when } \mathcal C\subset U_{\kappa,d,\phi}\,, \qquad \tilde \Phi(\mathcal C,\mathcal I)\leq \exp(-\lambda d(\mathcal C))
\] (e.g., $\tilde \Phi=\Phi \one\{\cC\subset U_{\kappa,d,\phi}$). Let $\tilde {\mathcal Z}(\bar V_n,\phi)$ be the partition function with weights $\tilde \Phi$ on $\bar V_n$ (see~\cite{DKS,MMRS91}). Then there exist  $q_0>0$ and $f(\kappa)\lesssim \kappa^{-1}$ such that for all $q\geq q_0$,
\begin{align}\label{eq:domain-comparison}
|\log \tilde{\mathcal Z}(\bar V_n,\phi)-\log \mathcal Z(\mathcal S_n,\phi)|\lesssim (\log n)^{f(\kappa)}\,.
\end{align}
Moreover, for $\tau_{f,w}(\phi)$ the order-disorder surface tension in the direction of $\phi$ (see~\cite{MMRS91}),
\begin{align}\label{eq:surface-tension}
|\log {\tilde {\mathcal Z}} (\bar V_n,\phi)-n\log (1+\sqrt q) (\cos\phi)^{-1}\tau_{f,w}(\phi)|\lesssim (\log n)^{f(\kappa)}\,,
\end{align}
and the large deviation estimate of Proposition~\ref{prop:fk-415} holds for $\bar V_n$ and $d/2$.\end{proposition}

With Propositions \ref{prop:fk-415} and \ref{prop:fk-416}, steps (a)--(e) carry through in the setting of the critical FK model with large $q$, proving Proposition \ref{prop:D-estimate}. Notice that the long-range interactions of the FK model are irrelevant to this situation where all boundary conditions used are completely free or completely wired and cannot be perturbed. \end{proof}

\subsection{The recursive scheme}
Throughout this subsection, let
$\bP$ be a distribution over FK boundary conditions on $\Lambda_{n,m}=\llb 0,n\rrb\times \llb 0,m\rrb$ and $\bE$ the corresponding expectation. For $\xi\sim \bP$, we say that
\[\cA^{\bP}_{n,m}(t,\delta) \qquad \mbox{holds if} \qquad \max_{\omega_0\in \{0,1\}} \bE \left[\left\| P^t(\omega_0,\cdot)-\pi_{\Lambda_{n,m}}^\xi \right\|_\tv\right] \leq \delta\,.
\]

Using this notation, the following corollary is a consequence of Lemma~\ref{lem:bc-perturbation}.

\begin{corollary} \label{cor:bc-perturbation}
Consider $\Lambda_{n,m}$ with boundary conditions $\xi\sim \bP$ and $\Delta\subset \partial \Lambda_{n,m}$ such that $|V(\Delta)| \asymp n^{3\epsilon}$, with $\bP^{\Delta}$ defined as in Definition~\ref{def:modification}. If for some $t,\delta$,
\[\max_{\omega_0\in\{0,1\}}\bE^{\Delta}\left[\|P^{t}(\omega_0,\cdot)-\pi_{\Lambda_{n,m}}^{\xi}\|_{\tv}\right]\leq \delta\,,
\]
then $\mathcal A^{\bP}_{{n,m}}(t',\delta')$ holds with $\delta'=8\delta +\exp({-e^{cn^{3\epsilon}}})$ and $t'=\exp({cn^{3\epsilon}})t$ for some $c(q)>0$ independent of $\Delta$ and $\xi$. Similarly, $\mathcal A^{\bP}_{{n,m}}(t,\delta)$ implies, \[\max_{\omega_0\in\{0,1\}}\bE^{\Delta}\left[\|P^{t'}(\omega_0,\cdot)-\pi_{\Lambda_{n,m}}^{\xi}\|_{\tv}\right]\leq \delta'\,.\]
\end{corollary}

Before proving the main theorem, we fix an $\epsilon>0$ and prove a recursive scheme that yields a mixing time bound on rectangles with side lengths $n\times n^{\frac 12 +\epsilon}$ for $n$ of the form $n\in\{2^k\}_{k\in \mathbb N}$. We remark that as in~\cite{MaTo10}, this is a technical assumption that is not requisite to the upper bound, (see Remark 3.12 of~\cite{MaTo10}).

For the base scale of the recursion, we use a consequence of the canonical paths estimate of Theorem~\ref{thm:canonical-paths}, specifically Lemma~\ref{lem:gap-shorter-side}, and the submultiplicativity of $\bar d_{\tv}$.

\begin{proposition} \label{prop:base-scale} There exists $c=c(q)>0$ such that for every $n$, for every $q$, for the FK Glauber dynamics, $\mathcal A^{\bP}_{{n,m}}(t,\exp[-te^{-c(n\wedge m)}])$ holds independent of $\bP$.
\end{proposition}

An intermediate step to proving Proposition \ref{prop:all-free} is proving analogous bounds for rectangles with ``free" boundary conditions on three sides and ``wired" on the fourth.

\begin{definition} A distribution $\bP$ over boundary conditions on $\Lambda_{n,m}$ is in $\mathcal D(\Lambda_{n,m})$ if it is dominated by $\pi_{\mathbb Z^2}^0$ on $\partial_{\north,\east,\west} \Lambda_{n,m}$ and dominates $\pi_{\mathbb Z^2}^1$ on $\partial_\south \Lambda_{n,m}$. 

We say that $\cA_{n,m}(t,\delta)$ holds if $\cA^{\bP}_{n,m}(t,\delta)$ holds for every $\bP\in \mathcal D(\Lambda_{n,m})$.
\end{definition}

The main estimate for our recursion on increasing rectangles is the following.

\begin{proposition} \label{prop:recursion} For the critical FK Glauber dynamics with $q$ large enough on $\Lambda_{n,m}$, the following holds: for any $m\in \llb n^{\frac 12+\epsilon},n\rrb$ and $\alpha\in (1,2)$, there exist $c_1,c_2>0$ such that for every $t,\delta$,
\begin{align}
\mathcal A_{n,m}(t,\delta)\implies & \mathcal A_{n,\lfloor \alpha m\rfloor}\left(2e^{c_2n^{3\epsilon}}t~,~c_1(\delta+e^{-c_2n^{2\epsilon}}+n^2 t^{-c_2})\right) \label{eq:recursion-1}\,,
\end{align}
and  for every $m\asymp n^{\frac 12+\epsilon}$ there exist $c_1,c_2>0$ such that for every $t,\delta$,
\begin{align}
\cA_{n,m}(t,\delta) \implies & \mathcal A_{2n,m}\left(3e^{c_2 n^{3\epsilon}}t~,~c_1(\delta+e^{-c_2n^{3\epsilon}})\right)\,. \label{eq:recursion-2}
\end{align}
\end{proposition}
Before proving the implications in Proposition~\ref{prop:recursion}, we first prove two easy but important consequences.
\begin{corollary}\label{cor:sqrt-mixing}There exists $q_0$ such that, for every $q\geq q_0$, there exist $c,c'>0$ such the following holds. If $n\in\{2^k\}_{k\in\mathbb N}$ is sufficiently large, then
the statement $\mathcal A_{n,n^{1/2+\epsilon}}(\exp[cn^{3\epsilon}], \exp[-c'n^{2\epsilon}])$ holds for critical FK Glauber dynamics on $\Lambda_{n,n^{1/2+\epsilon}}$.
\end{corollary}
\begin{proof} Choose $n_0\asymp n^{\epsilon}$ and let $t_{0}=\exp(c'n^{\epsilon})$ for some constant $c'>0$ large enough that $\mathcal A_{n_0,n_0^{1/2+\epsilon}}(t_{0}, \delta_{0})$ holds with $\delta_{0}=\exp(-c'n^{\epsilon})$, noting that such a choice of $c'$ exists by Proposition~\ref{prop:base-scale}. Applying~\eqref{eq:recursion-1} followed by~\eqref{eq:recursion-2} with $m=n^{\frac 12+\epsilon}$ and $\alpha=2^{\frac 12+\epsilon}$ allows one to, for any $n$, express the mixing time of a $2n \times (2n)^{\frac 12+\epsilon}$ box in terms of that of a $n\times n^{\frac 12+\epsilon}$ box. Starting with the scale $n_0$ and repeating this step $\log_2 n$ times implies that there exists $c>0$ fixed such that $\mathcal A_{n,n^{1/2+\epsilon}}( \exp(cn^{3\epsilon}),c \exp[-n^{-2\epsilon}/c])$ holds.
\end{proof}

We now use the bound on $n\times n^{\frac 12+\epsilon}$ rectangles to obtain mixing time bounds on the $n\times m$ rectangle with boundary conditions that are disordered on three sides and ordered on the fourth.

\begin{corollary} \label{cor:3-sides}
Consider the critical FK Glauber dynamics on $\Lambda_{n,m}$ for $m\in \llb n^{\frac 12+\epsilon},n\rrb$ and boundary conditions $\xi\sim \bP$. Then there exists $q_0>0$ such that for all $q\geq q_0$ there exists a constant $c=c(m,q)>0$ such that for large enough $n$, for every $\bP\in \mathcal D(\Lambda_{n,m})$,
\[\max_{\omega_0\in\{0,1\}}\bE\left[\|P^{t_\star}(\omega_0,\cdot)-\pi_{\Lambda_{n,m}}^{\xi}\|_\tv\right]\lesssim e^{-cn^{2\epsilon}}\,,
\]
for $t_\star=\exp(cn^{3\epsilon})$, and in particular if $\tmix^{m}$ is the corresponding mixing time,
\[\bP(\tmix^{m}\geq t_\star)\lesssim e^{-cn^{2\epsilon}}.
\]
\end{corollary}
\begin{proof}
Choose an $\alpha\in (1,2)$ such that $\alpha^k n^{\frac 12+\epsilon}=m$ for some integer $k\asymp \log m$. Then let $h_j=\lfloor \alpha^j n^{\frac 12+\epsilon}\rfloor $ and let $\Lambda^j=\Lambda_n^j=\llb 0,n\rrb \times \llb 0,h_j\rrb$ so that $\Lambda^k=\Lambda_{n,m}$.

We prove the above by induction on $j\in \llb 0,k\rrb$ for $n\times h_j$ rectangles, showing that
\begin{equation}
  \label{eq-cor-3sides-induction}
  \max_{\omega_0\in\{0,1\}}\bE\left[\|P^{t_{h_j}}(\omega_0,\cdot)-\pi_{\Lambda_{n,h_j}}^{\xi}\|_\tv\right]\leq  (c_1^j+2jc_1)e^{-c_2n^{2\epsilon}}\,,
\end{equation}
where $c_1,c_2$ are the constants of \eqref{eq:recursion-1} for $m=h_j$, and
\[t_{h_j}=2^j{h_j}^{c_2(1+n^{3\epsilon})}\,.\]
The base case $j=0$ is given by Corollary~\ref{cor:sqrt-mixing}, and if~\eqref{eq-cor-3sides-induction} holds for some fixed $j\in \llb 0,k-1\rrb$, then an application of \eqref{eq:recursion-1} immediately implies
it for $j+1$. The observations that $j\leq \log n$ and $h_j\leq n$ allow us to choose slightly different constants to obtain the first inequality of Corollary \ref{cor:3-sides}. The triangle inequality and Eq.~\eqref{eq:init-config-comparison} can then be used to boost the bound on $d_1(t)\vee d_0(t)$ to a bound on $\bar d(t)$, so that Markov's inequality implies the second inequality.
\end{proof}

\begin{figure}
  \begin{tikzpicture}
    \node (plot1) at (0,0) {};
    \node (plot2) at (7.5,0.2) {};
    \begin{scope}[shift={(plot1.south west)},x={(plot1.south east)},y={(plot1.north west)}, font=\small]

     \filldraw[draw=blue,thick,opacity=0.755,fill=blue, fill opacity=0.1] (0,0) rectangle (20,10);
     \filldraw[draw=DarkGreen,thick,opacity=0.75,fill=DarkGreen,fill opacity=0.1] (0,4) rectangle (20,14);
     \draw[color=black,thick,opacity=0.5] (-0.1,-0.1) rectangle (20.1,14.1);

      \node[color=blue,font=\Large] at (10,4.25) {$B$};
      \node[color=green!50!black,font=\Large] at (10,9.75) {$A$};
      \node[font=\Large] at (10,-2) {$Q$};

     \draw[color=black,style=dotted,<->] (0,14.5) -- (20,14.5);
     \draw[color=black,style=dotted,<->] (20.5,0) -- (20.5,14);
     \draw[color=black,style=dotted,<->] (-0.5,0) -- (-0.5,10);

      \node[color=black] at (10,15.25) {$n$};
      \node[color=black] at (22.25,7) {$\lfloor\alpha m\rfloor$};
      \node[color=black] at (-1.4,5) {$m$};
    \end{scope}

    \begin{scope}[shift={(plot2.south west)},x={(plot2.south east)},y={(plot2.north west)}, font=\small]

    \newcommand{\hsep}{10}
    \newcommand{\vsep}{8.}

    \filldraw[draw=blue,thick,opacity=0.755,fill=blue, fill opacity=0.2] (0,0+\vsep) rectangle (7,3.5+\vsep);
     \filldraw[draw=DarkGreen,thick,opacity=0.75,fill=DarkGreen,fill opacity=0.2] (0,1.4+\vsep) rectangle (7,4.85+\vsep);
     \node[color=black,font=\tiny] at (3.5,0.7+\vsep) {$1$};
     \node[color=black,font=\tiny] at (3.5,2.4+\vsep) {$1$};
     \node[color=black,font=\tiny] at (3.5,4.2+\vsep) {$1$};

     \draw[color=black,->](3.5,\vsep-0.75) -- (3.5,4.85+0.75);
     \node at (4.25,6.5) {$t$};

      \filldraw[draw=blue,thick,opacity=0.755,fill=blue, fill opacity=0.2] (0,0) rectangle (7,3.5);
     \filldraw[draw=DarkGreen,thick,opacity=0.75,fill=DarkGreen,fill opacity=0.2] (0,1.4) rectangle (7,4.85);
     \node[color=black,font=\tiny] at (3.5,0.7) {$1$};
     \draw[pattern=north east lines, pattern color=DarkGreen, opacity=0.5] (0,1.4) rectangle (7,4.85);
     \node[color=black] at (3.5,3) {$\nu_1$};

     \draw[color=black,->] (7+0.75,2.4) -- (\hsep-0.75,2.4);

      \filldraw[draw=blue,thick,opacity=0.755,fill=blue, fill opacity=0.2] (0+\hsep,0) rectangle (7+\hsep,3.5);
     \filldraw[draw=DarkGreen,thick,opacity=0.75,fill=DarkGreen,fill opacity=0.2] (0+\hsep,1.4) rectangle (7+\hsep,4.85);
     \node[color=black,font=\tiny] at (3.5+\hsep,0.7) {$1$};
     \draw[pattern=north east lines, pattern color=DarkGreen, opacity=0.5] (0+\hsep,3.5) rectangle (7+\hsep,4.85);
     \node[color=black,font=\tiny] at (3.5+\hsep,2.4) {$1$};
     \node[color=black] at (3.5+\hsep,4.15) {$\eta$};

     \draw[color=black,->] (3.5+\hsep,4.85+0.75) -- (3.5+\hsep,\vsep-0.75);
     \node at (4.25+\hsep,6.5) {$t$};

      \filldraw[draw=blue,thick,opacity=0.755,fill=blue, fill opacity=0.2] (0+\hsep,0+\vsep) rectangle (7+\hsep,3.5+\vsep);
     \filldraw[draw=DarkGreen,thick,opacity=0.75,fill=DarkGreen,fill opacity=0.2] (0+\hsep,1.4+\vsep) rectangle (7+\hsep,4.85+\vsep);
      \draw[pattern=north east lines, pattern color=DarkGreen, opacity=0.5] (0+\hsep,3.5+\vsep) rectangle (7+\hsep,4.85+\vsep);
      \draw[pattern=north west lines, pattern color=blue, opacity=0.3] (0+\hsep,0+\vsep) rectangle (7+\hsep,3.5+\vsep);
     \node[color=black] at (3.5+\hsep,1.95+\vsep) {$\nu_2^\eta$};
     \node[color=black] at (3.5+\hsep,4.15+\vsep) {$\eta$};

    \end{scope}

  \end{tikzpicture}
  \caption{Setup for the proof of Eq.~(\ref{eq:recursion-1}) starting from wired initial conditions.}
  \label{fig:setup-recursion-1}
\end{figure}
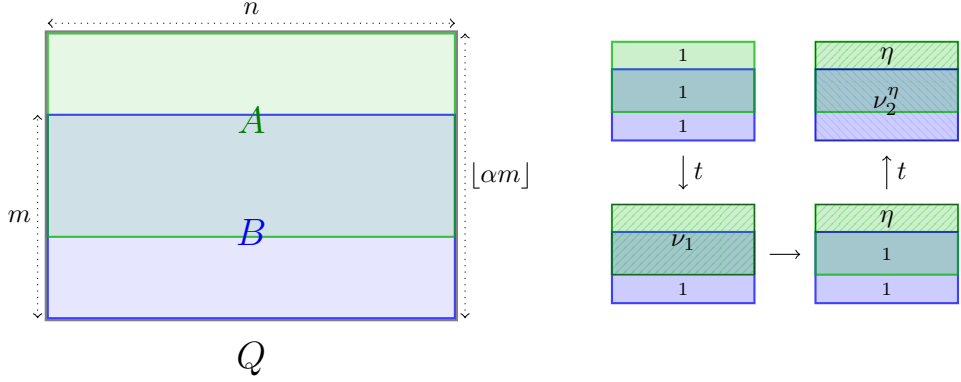

We now prove Proposition~\ref{prop:recursion} from which the above corollaries follow. The proof of Proposition~\ref{prop:all-free} then follows from Corollary~\ref{cor:3-sides} using similar techniques (see~\S\ref{sub:proof-large-q-upper-1})

\begin{proof} [\textbf{\emph{Proof of Eq.~\eqref{eq:recursion-1}}}] Fix any $\bP\in \mathcal D(\Lambda_{n,\lfloor \alpha m \rfloor})$ and observe that the proof is independent of this choice of $\bP$. Consider the quantity, $\bE [\| P^t(\omega_0,\cdot)-\pi_{\Lambda_{n,\lfloor \alpha m\rfloor}}^\xi \|_\tv]$ for $\omega_0=0,1$.

\emph{(i) Wired initial conditions.} Begin with the case when $\omega_0=1$. Let $A,B$ be two copies of $\Lambda_{n,m}$ with $A$ translated upwards by $\lfloor (\alpha-1)m\rfloor$ such that $Q:=A\cup B= \Lambda_{n,\lfloor \alpha m\rfloor}$ and $A\cap B$ is the middle rectangle in $Q_n$ of thickness $\asymp m$.

In order to compensate for the long-range interactions of the FK model, that are not present in the setting of~\cite{MaTo10}, we force a set of boundary edges to be free (in a manner similar to part 2 of the proof of Theorem~3.2 of~\cite{MaTo10}) to ``disconnect" $B$ from $A$. Consider the boundary condition $\xi'$, a modification of $\xi\sim \bP$ on $\Delta=\Delta_\south \cup \Delta_\north$ for,
\begin{align*}\Delta_\south= & \{(x,y)\in \partial_{\east,\west} A:y\leq \lfloor (\alpha-1)m \rfloor+ n^{3\epsilon}\}\,, \\
\Delta_\north= & \{(x,y)\in \partial_{\east,\west} B:y\geq \lfloor m -n^{3\epsilon} \rfloor\}\,,
\end{align*} according to Definition~\ref{def:modification}. By Corollary~\ref{cor:bc-perturbation} it suffices, up to new choice of constants $c_1,c_2$ to show that the FK Glauber dynamics under $\bP^{\Delta}$ on $A\cup B$ satisfies~\eqref{eq:recursion-1}.

Denote by $\tilde P$ the transition kernel of the censored dynamics $(\tilde X_s)_{s\geq 0}$ started from the all wired configuration, only accepting updates in block $A$ up to time $t$, resetting all edge values in $B$ to $1$ at time $t$ then only accepting updates in block $B$ from time $t$ to time $2t$ (observe that as in Lemma~3.4 of~\cite{MaTo10}, by Theorem~\ref{thm:censoring}, resetting all edge values to $1$ only slows mixing). Let $\nu_1$ denote the distribution after time $t$ on $A$ and let $\nu_2^{\eta}$ denote the distribution after time $2t$ of configurations on $B$ given that at time $t$ the configuration on $B$ was set to $1$ and the boundary condition on $B^c$ was $\eta$ (see Fig.~\ref{fig:setup-recursion-1}).

The monotonicity of the FK model along with Theorem~\ref{thm:censoring} yields,
\begin{align} \label{eq:first-wired-censor}
d_1^{\xi'}(2t)\leq \|\tilde P^{2t}(1,\cdot)- \pi_{Q}^{\xi'}\|_{\tv}\,.
\end{align}
Now we aim to show that $\bE^{\Delta}\big[\|\tilde P^{2t}(1,\cdot)- \pi_{Q}^{\xi'}\|_{\tv}\big]\leq \delta'$ where we let $\delta'$ be the second argument in the right hand side of~\eqref{eq:recursion-1}. For $R=A,B$ we denote by $\pi^{\xi',\eta}_R$ the modified stationary distribution with $\eta$ boundary conditions on $Q-R$.

To simplify the notation, throughout the rest of this section, we let $\|\mu-\nu\|_R$ denote $\|\mu\restriction_R-\nu\restriction_R\|_\tv$. Also, for any $R,\xi$ and any random variable $X$, let $\pi_R^\xi(X)$ denote the expectation of $X$ under $\pi_R^\xi$. By the Markov property and the triangle inequality,
\begin{align}
\bE^{\Delta}\left[\|\tilde P^{2t}(1,\cdot)-\pi^{\xi'}_Q\|_\tv\right]\leq ~ & \bE^{\Delta}\left[\| \nu_1-\pi^{\xi',1}_A\|_{B^c}\right]+ \bE^{\Delta}\left[\|\pi_A^{\xi',1}-\pi_Q^{\xi'}\|_{B^c}\right] \nonumber\\
 +&\bE^{\Delta}\left[\|\pi_Q^{\xi'}-\pi_Q^{\xi',0}\|_{B^c}\right]+\bE^{\Delta}\left[\pi_Q^{\xi',0}(\|\nu_2^\eta-\pi_B^{\xi',\eta}\|_\tv)\right]\,.
\label{eq:first-wired}
\end{align}

We begin by bounding the first and fourth terms which are easier, then use the equilibrium estimates of the cluster expansion to bound the third term in Lemma~\ref{lem:new-technique}, analogous to which the second term can be bounded. First observe that with probability $1-\exp(-cn^{3\epsilon})$ for some $c>0$, the boundary conditions on $A$ are sampled from a distribution in $\mathcal D(A)$. The concern is that the wired initial configuration may add connections to $\partial_{\north,\east,\west}A$ via the long-range FK interactions. Such an effect on the boundary conditions on $A$ is impossible if there are no boundary connections from $\partial_{\east,\west} A^c$ to $\partial_{\north,\east,\west} A$. Because of the modification on $\Delta_\south$ such a connection would require a connection of length at least $n^{3\epsilon}$ under $\bP$ and therefore also in the free phase, which has probability less than $\exp(-cn^{3\epsilon})$ (see Eq.~\eqref{eq:exp-decay-1}). If no such connection exists along the boundary, the boundary conditions on $\partial_{\north,\east,\west} A$ are sampled from a measure dominated by $\pi_{\mathbb Z^2}^0$ because the modification of Definition~\ref{def:modification} only removes connections. From now on, paying a cost of $\exp(-cn^{3\epsilon})$, we assume the decreasing event that this is the case.

Then by the assumption that $\mathcal A_{n,m}(t,\delta)$ holds, the first term in~\eqref{eq:first-wired} is smaller than $\delta$. The observation that $\bP^{\Delta}\preceq \bP$ on $\partial_{\east,\west} Q$, implies that for any decreasing $f$ only depending on $\partial_{\north,\east,\west}B$,
\begin{align} \label{eq:avg-avg}
\bE^{\Delta}\left[\pi_{Q}^{\xi',0}(f)\right]\geq \pi_{\mathbb Z^2}^0(\pi_{Q}^{\xi}(f))= \pi_{\mathbb Z^2}^0(f)\,,
\end{align}
so that the $\pi_Q^{\xi',0}$-averaged distribution on boundary conditions on $B$ is in $\mathcal D(B)$, the statement $\mathcal A_{n,m}(t,\delta)$ applies, and the fourth term in~\eqref{eq:first-wired} is also bounded above by $\delta$.

We now turn to the second and third terms of~\eqref{eq:first-wired}, which can be bounded similarly, and thus we only go through the details of the third term:
\begin{lemma} \label{lem:new-technique} There exists $c(q)>0$ such that
\[\bE^{\Delta}\left[\|\pi_Q^{\xi'}-\pi_Q^{\xi',0}\|_{B^c}\right]\lesssim e^{-cn^{2\epsilon}}\,.
\]
\end{lemma}
\begin{proof}
We bound the total variation distance by proving that under the grand coupling of the two distributions on $B^c$, they agree with probability $1=e^{-cn^{2\epsilon}}$. Let $\partial^{\pm} (Q)$ denote the two connected components of $\partial Q -\Delta_\north$ above and below $\Delta_\north$ respectively. We break up the expectation into an average over $\Gamma_1$, the set of $\xi'$ in which there does not exist a pair $(x,y)\in\partial^+Q\times \partial^- Q$ such that $x\stackrel{\xi'}\longleftrightarrow y$ (i.e.\ they are in the same boundary component), and $\Gamma_1^c$. By Eq.~\eqref{eq:exp-decay-1} of Theorem~\ref{thm:DC-S-T-main} and a union bound over pairs of boundary vertices, there exists a constant $c'(q)>0$ such that
\begin{align}
\bP^{\Delta}(\Gamma_1^c)\leq 16n^2e^{-cn^{3\epsilon}}\lesssim \exp({-c'n^{3\epsilon}})\,,
\end{align}

\begin{figure}
  \begin{tikzpicture}
    \node (plot1) at (0,0) {};

    \begin{scope}[shift={(plot1.south west)},x={(plot1.south east)},y={(plot1.north west)}, font=\small]

      \draw[color=DarkGreen] (0,0) rectangle (30,15);
      \draw[style=thick] (7.5,0.075) rectangle (22.5, 7.575);
      \draw[color=gray!70!black,style=dashed] (0,0) -- (0,-7.5) -- (30, -7.5) -- (30,0);

      \draw[color=DarkGreen,style=dotted,<->] (15,7.52) -- (15,14.85) ;
      \draw[color=DarkGreen,style=dotted,<->] (22.65,3.75) -- (29.85,3.75) ;
      \draw[color=DarkGreen,style=dotted,<->] (0.15,3.75) -- (7.35,3.75) ;

      \draw[color=DarkGreen,decoration={bumps, segment length=0.085in, amplitude=0.02in}, decorate, very thin] (29.96,-0.03) -- (0, -0.03);


      \node[color=DarkGreen] at (5.25,4.5) {$n$};
      \node[color=DarkGreen] at (24.75,4.5) {$n$};
      \node[color=DarkGreen] at (15.85,12.55) {$n$};
      \node[font=\Large] at (15,3.75) {$Q$};
      \node[font=\Large,color=DarkGreen] at (4,13.5) {$E_n(Q)$};
      \node[font=\Large,color=gray!70!black] at (27.2,-6.15) {$E'_n(Q)$};

       \node at (15.02,1.1) {
	\includegraphics[width=120pt]{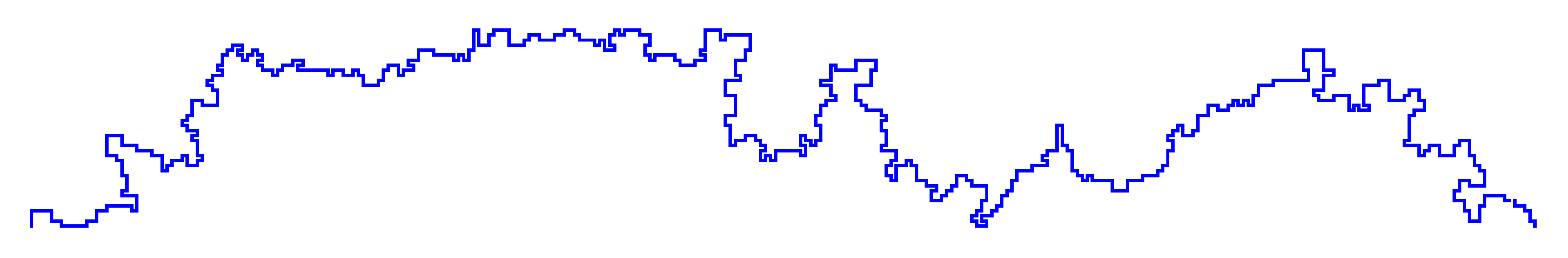}};

       \node at (15,3.5) {
	\includegraphics[width=210pt]{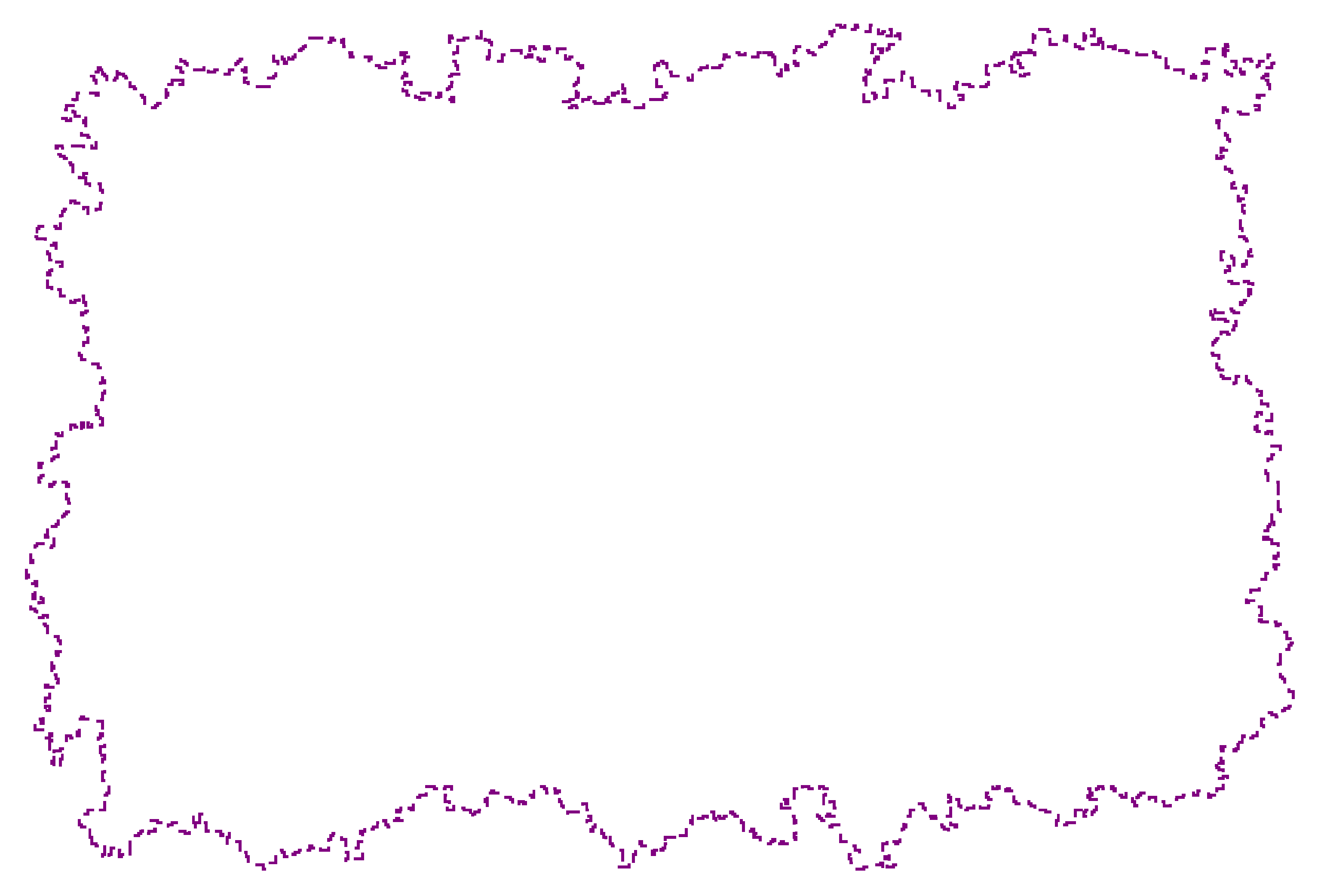}};

    \end{scope}
  \end{tikzpicture}
  \caption{Bounding $\Gamma_2$: if a dual-open circuit (purple) exists under $\pi_{\mathbb Z^2}^0$ in $E'_n(Q)-Q$ then the boundary conditions on $\partial_{\north,\east,\west} Q$ are dominated by those under free on $E'_n(Q)$. The wired boundary conditions on $\partial_\south (E_n(Q))$ then also dominate the ``wired" boundary conditions on $\partial_\south Q$ allowing us to dominate the interface (blue) in $Q$ by that in $E_n(Q)$.}
  \label{fig:gamma-2}
\end{figure}
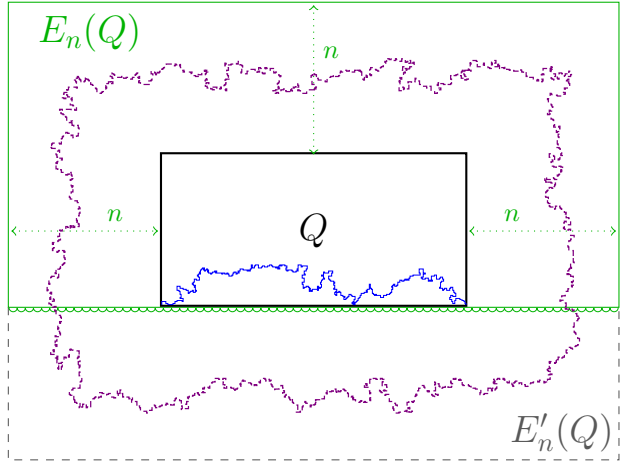

For all such $\xi'$, we use the worst bound of $1$ on the total variation distance. Suppose now that $\Gamma_1$ holds and observe that this is a decreasing event so $\bP^{\Delta}(\cdot \mid \Gamma_1)\preceq \bP^{\Delta}$. Let $\Gamma_2$ denote the decreasing event that the interface (bottom-most horizontal dual crossing) of $Q$ is contained entirely below $\partial_\south (B^c)$. Let $\Gamma_3$ be the decreasing event that there does not exist any vertex $x\in \partial_\east^-Q$ such that $x\longleftrightarrow \partial_\north B$, and there does not exist any $y\in \partial_\west^- Q$ such that $y\longleftrightarrow \partial_\north B$.

By monotonicity and the domain Markov property, for $\xi'\in\Gamma_1$, if $\Gamma_2\cap \Gamma_3$ holds, it is the case that the boundary conditions on $\partial_{\north,\east,\west} B^c$ will not have been affected by the updates on $A\cap B$ as the interface and all its long-range interactions with $\partial Q$ would be confined to $A\cap B$. Then one could reveal all boundary components of $\partial_{\east,\south,\west}B$ so that they are all confined to $B$ and by monotonicity under the grand coupling the two distributions would be coupled on $B^c$. As a result,
\begin{align}
\bE^{\Delta}\left[\|\pi_Q^{\xi'}-\pi_Q^{\xi',0}\|_{B^c}\given \Gamma_1\right]\leq \bE^{\Delta} \left[\pi_Q^{\xi',1}(\Gamma_2^c\cup\Gamma_3^c)\given \Gamma_1\right]\,.
\end{align}
We bound the two probabilities separately and take a union bound. To bound the probability of $\Gamma_2^c$, consider the enlarged rectangle,
\begin{align}\label{eq:enlargement}
E_n(Q)=\llb -n,2n\rrb \times \llb 0,n+\alpha m\rrb \supset Q\,,
\end{align} with $(0,1)$ boundary conditions denoting wired on $\partial _{\south}E_n(Q)$ and free elsewhere. By Definition~\ref{def:boundary-condition} we sample $\partial_{\north,\east,\west} Q$ separately and then $\partial_\south Q$. First observe that by Eq.~\eqref{eq:exp-decay-1}, with $\pi_{\mathbb Z^2}^0$-probability $1-e^{-cn}$, there is a dual circuit between $Q$ and its enlargement by $n$ in all four sides $E_n '(Q)$. In that case, the boundary conditions on $Q$ are dominated by those with free on $\partial E'_n(Q)$ (see Figure~\ref{fig:gamma-2}). We can subsequently dominate the boundary conditions on $\partial _\south Q$ by making them all wired and extending them all the way across $E_n(Q)$ to obtain that there exists $c>0$ such that
\[\bE^{\Delta}\left[\pi_Q^{\xi',1}(\Gamma_2^c)\given\Gamma_1\right]\leq \pi^{0,1}_{E_n(Q)}(\Gamma_2^c)+e^{-cn}\,.
\]

\begin{figure}
  \hspace{-0.15in}

  \begin{tikzpicture}
    \node (plot) at (0,0)
    {\hspace{-0.14in}\vspace{-0.3in}\includegraphics[width=0.82\textwidth]{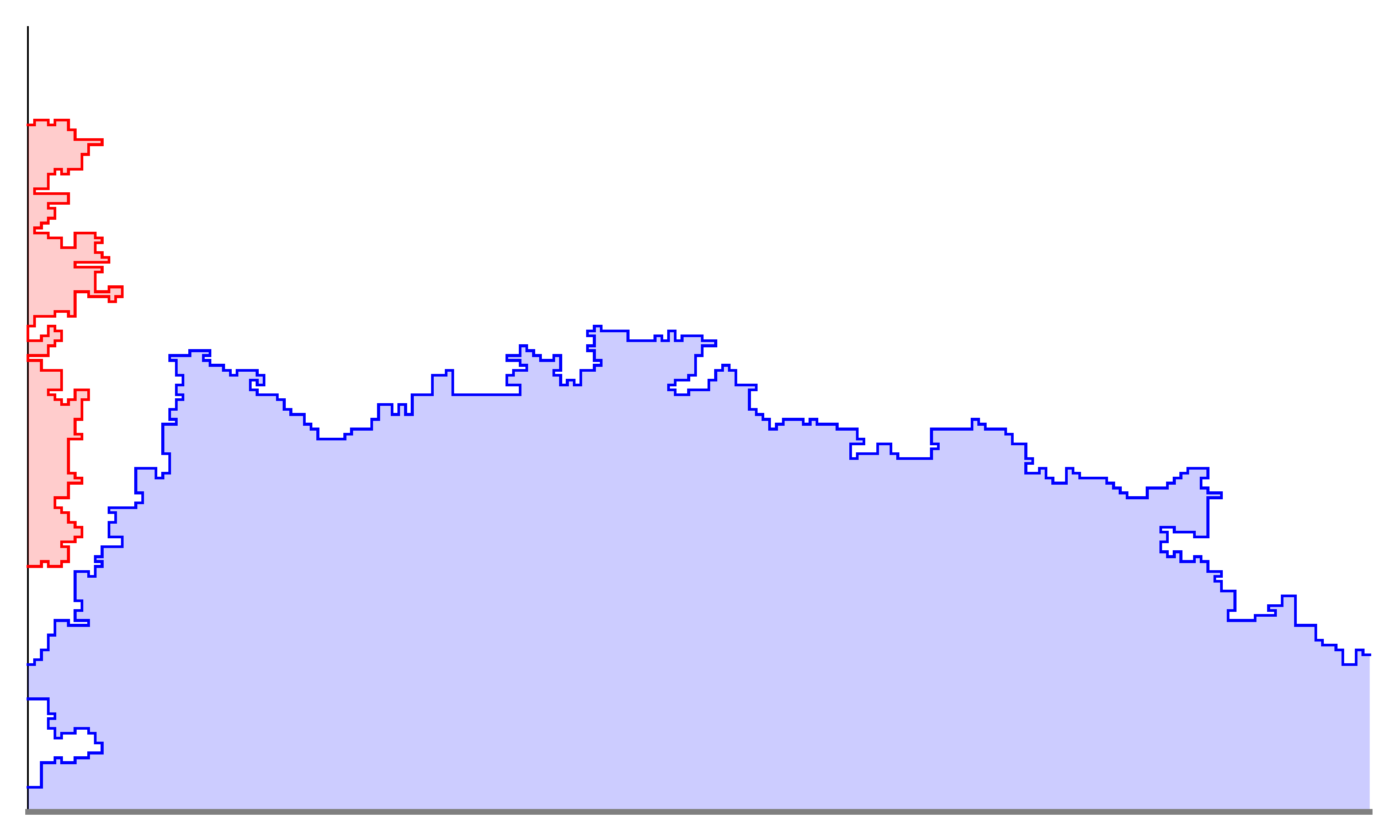}};

    \begin{scope}[shift={(plot.south west)},x={(plot.south
        east)},y={(plot.north west)}, font=\small]

      \draw[color=gray] (0.00,0.93) -- (.97,0.93);
      \draw[color=purple, thick] (0.0012,0.185) arc (160:200:.19) ;
      \draw[color=purple, thick] (0.0012,0.35) arc (160:200:.25) ;
      \draw[color=purple, thick] (0.0012,0.33) arc (160:200:0.12) ;
      \draw[color=purple, thick] (0.0012,0.25) arc (160:200:0.06) ;
      \draw[color=purple, thick] (0.0012,0.182) arc (-40:40:0.02) ;
      \draw[color=DarkGreen, thick] (0.001,0.62) arc (160:200:0.07) ;
      \draw[color=orange, thick] (0.001,0.64) arc (160:200:0.18) ;

      \draw[style=dashed] (-0.03,0.87) -- (0.00,0.87) ;
      \draw[style=dashed] (-0.03,0.93) -- (0.00,0.93) ;

     \node[color=black] at (0.5,0) {``$\boldsymbol 1$"};
      \node[color=gray] at (-.02,0.4) {``$0$"};
      \node[color=gray] at (-0.02,0.90) {$0$};
      \node[color=gray] at (-0.02,0.97) {``$0$"};
      \node[color=black] at (0.03,0.90) {$\Delta_\north$};
      \node[font=\Large] at (0.5,0.8) {$B$};

    \end{scope}
  \end{tikzpicture}
  \caption{Modification to disconnect long-range FK interactions. The event $\Gamma_2$ is the event that the blue cluster does not climb too high. Under $\Gamma_2$, there is a horizontal dual-crossing of $B$ immediately adjacent the blue interface, but the boundary of the red component may have been perturbed by connections in the blue shaded region. The event $\Gamma_3$ is the event that the red component then does not then climb above $\Delta$.}
  \label{fig:recursion-1nd-step}
\end{figure}
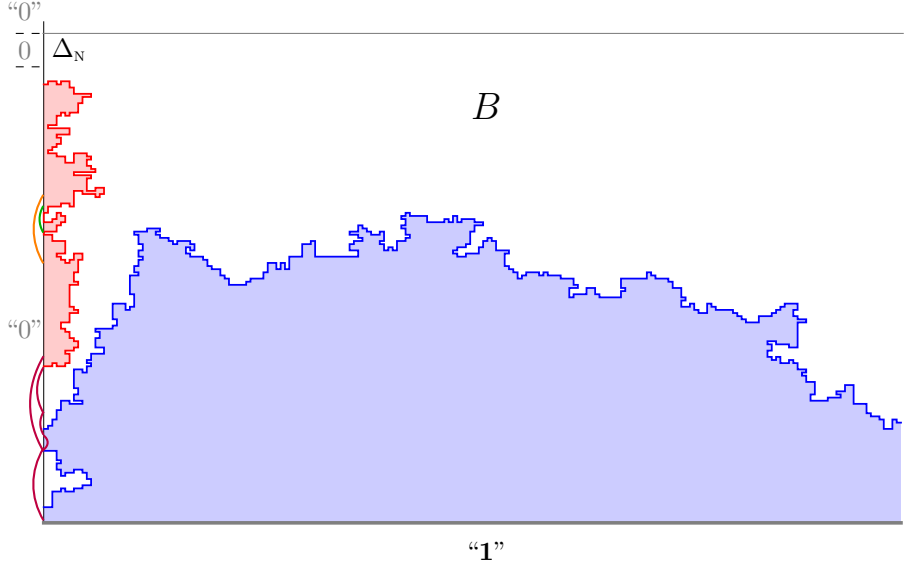

By Proposition~\ref{prop:surface-tension}, we deduce that $\pi_{E_n(Q)}^{0,1}(\Gamma_2^c)\lesssim \exp(-cn^{2\epsilon})$ for some $c>0$. We now bound the probability of $\Gamma_3^c$.

\begin{claim} \label{lem:new-bound} There exists  $c=c(q)>0$ such that for every $\xi'\in\Gamma_1$,
\[\pi_Q^{\xi',1}(\Gamma_3^c\mid \Gamma_2)\lesssim e^{-cn^{3\epsilon}}+e^{-cn^{\frac 12 +\epsilon}}+e^{-cn}\,.
\]
\end{claim}

\begin{proof}
Under $\Gamma_2$, by monotonicity, we can only worsen our bound on the probability of $\Gamma_3^c$ by replacing the boundary conditions on $Q$ by wired on $\partial^- Q -\partial_\south Q$, free on $\partial_\south Q$, and $\xi'$ elsewhere.

Let $\bar Q=\llb 0,n\rrb \times \llb 0, 2m \rrb\supset Q$ with boundary conditions free on $\partial_{\south,\north} \bar Q \cup \Delta_\north$, and wired elsewhere. By the exponential decay of correlations in the free phase, with $\pi_{\bar Q}$-probability $1-e^{-cm}$, the measure this induces on $Q$ dominates the boundary conditions under $\Gamma_2$ on $Q$.

Controlling the probability of $\Gamma_3^c$ can now be expressed in a manner similar to the equilibrium bound, Proposition~\ref{prop:D-estimate}.
As is standard in such problems, (see, e.g., the appendix of~\cite{MaTo10}), we can up to an error of $e^{-cn}$ separate the left and right interfaces (see Proposition~\ref{prop:fk-416} whence the probability that they interact is a large deviation of order $n$), and just consider $\bar Q$ with free boundary conditions now on all of $\partial_\east \bar Q$ also.

Then extend the northern boundary of $\bar Q$ to make $\bar Q$ symmetric about $\Delta_\north$ and call the new domain $Q'$. We can, using monotonicity, let its boundary conditions $(1,0,\Delta_\north)$ be free on $\partial Q'\cap (\{x\geq m^{\frac 12+\epsilon}\}\cup \Delta_\north)$ and wired elsewhere.

At this point, we apply Proposition~\ref{prop:D-estimate} with $\ell=n$ (up to a $\frac{\pi}2$-rotation and a rescaling of $m$ to $n/2$) to obtain the desired bound: if in the new domain, the boundary points are denoted by $w_1,w_2\in \partial_\north Q'\times \partial_\south Q'$ and $z_1,z_2\in \partial_\east Q'$, and $\{C_i\}_{i\in\north,\south}$ are the north or south halves of $Q'$ respectively, Proposition~\ref{prop:D-estimate} implies that there exists $c>0$ such that for large enough $n$,
\[\pi_{Q'}^{1,0,\Delta_\north}\bigg(\bigcap_{i=\north,\south}\{w_i\stackrel {C_i}\longleftrightarrow z_i\}\bigg)\leq e^{-cn^{3\epsilon}}\,.
\]

Putting everything together, we conclude that if $\Gamma'$ is the event that under $\pi_{Q'}^{1,0,\Delta_\north}$, the two interfaces are contained in bottom and top halves of $Q'$ respectively, then there exists $c=c(q)>0$ such that, for large enough $n$,
\[\pi_{\bar Q}^{1,0,\Delta_\north} (\Gamma'^c)\leq 2e^{-cn^{3\epsilon}}+2e^{-cm}+2e^{-cn}\,,
\]
Monotonicity and $n^{\frac 12+\epsilon}\leq m\leq n$ imply the bound on $\bE^{\Delta}[\pi_Q^{\xi',1}(\Gamma_3^c\mid \Gamma_2) \mid \Gamma_1]$.
\end{proof}

By union bounding over the errors that arise from each of the $\Gamma_i$ not occurring, and otherwise conditioning on their occurrence, we obtain \[\bE^{\Delta}\left[\|\pi_Q^{\xi'}-\pi_Q^{\xi',0}\|_{B^c}\right]\lesssim e^{-cn^{2\epsilon}}\,,
\]
as desired.
\end{proof}

The corresponding bound on the second term of~\eqref{eq:first-wired-censor} is the same up to changes of scale corresponding to working with distributions on configurations on $A$, not $Q$. Because of the modification on $\Delta_\south$, as remarked earlier, with probability $1-\exp(-cn^{3\epsilon})$, the boundary conditions on $A$ are in $\mathcal D(A)$. This event is a decreasing event so it only increases the probability of $\Gamma_i$ for $i=1,2,3$. Because $\alpha>1$, the middle rectangle is at least order $n^{\frac 12+\epsilon}$ so the bound on the interface touching $\partial_\south (B^c)$ also still holds. Combined with~\eqref{eq:first-wired-censor} and the bounds on the other terms in~\eqref{eq:first-wired}, we conclude that for some $c_1,c_2>0$ and every large enough $n$,
\[\bE^{\Delta}\left[d_1^{\xi'}(2t)\right]\leq 2\delta+2c_1 e^{-c_2n^{2\epsilon}}\,.
\]

\emph{(ii) Free initial configuration.} Consider the dynamics started from $\omega_0=0$. Let $\tilde P$ denote the transition kernel of the censored dynamics that only updates edges in $B$ until time $t$ at which point all edges in $A$ are reset to $0$ and the dynamics subsequently only updates edges in $A$ until time $2t$. Let $\nu_2^\eta$ denote the distribution obtained between times $t$ and $2t$ given boundary conditions $\eta$ on $A^c$ and initial configuration $0$ on $A$. Let
\[\Delta=\{(x,y)\in \partial_{\east,\west} (A^c): y\geq \lfloor (\alpha -1)m -n^{3\epsilon}\rfloor\}\,,
\]
so that again by Theorem \ref{thm:censoring} and Corollary~\ref{cor:bc-perturbation} it suffices to prove the desired implication under $\bE^\Delta$ for the $\tilde P$ dynamics.

The Markov property and the triangle inequality together imply,
\begin{align}
\bE^{\Delta}\left[\|\tilde P^{2t}(0,\cdot)-\pi_Q^{\xi'}\|_\tv\right]\leq ~& \bE^{\Delta}\left[\|\tilde P^{t}(0,\cdot)-\pi^{\xi',0}_B\|_{A^c}\right]+\bE^{\Delta}\left[\|\pi_B^{\xi',0}-\pi_Q^{\xi'}\|_{A^c}\right] \nonumber \\
 + &\bE^{\Delta}\left[\pi_Q^{\xi'}(\|\nu_2^{\eta}-\pi^{\xi',\eta}_A\|_\tv)\right]\,. \label{eq:first-free}
\end{align}
We can bound the first term by $\delta$ by assumption and the observation that the free initial configuration does not change the boundary conditions on $B$. The second term can be bounded using the same approach as the proof of Lemma~\ref{lem:new-technique}, where now $A^c$ is shorter than the rectangle in Lemma~\ref{lem:new-technique} but still $\gtrsim n^{\frac 12 +\epsilon}$. Via an application of Lemma~\ref{lem:new-technique}, up to an error of $\exp(-cn^{2\epsilon})$ for $c>0$ fixed, the open component of $\partial_\south B$ and its perturbations to the boundary of $A\cup B$ do not reach $\partial_\north A^c$ so the grand coupling couples the two distributions on $A^c$.

In that case, the measure on $\partial A$ is dominated by $\pi_{\mathbb Z^2}^0$ and therefore by~\eqref{eq:exp-decay-1}, up to an error of $e^{-cm}$ the entirety of $B^c$ is disconnected from $A^c$. By monotonicity and the fact that $m\geq n^{\frac 12+\epsilon}$, we obtain that there exists $c>0$ such that
\[\bE^{\Delta}\left[\|\pi_B^{\xi',0}-\pi_Q^{\xi'}\|_{A^c}\right]\lesssim e^{-cn^{\frac 12+\epsilon}}+e^{-cn^{2\epsilon}}\,.
\]

It remains to bound the third term in~\eqref{eq:first-free} following the approach of~\cite{MaTo10}.
\begin{lemma} \label{lem:boxes}
There exist constants $c,c'>0$ such that
\[\pi^{\xi'}(\|\nu_2^{\eta}-\pi^{\xi',\eta}_A\|_\tv)\lesssim e^{-cn^{2\epsilon}}+e^{-c'\ell}\,.
\]
\end{lemma}
\begin{proof}  Using the bound on total variation by the probability of disagreement under a maximal coupling, together with monotonicity and a union bound, write
\[\pi^{\xi'}(\|\nu_2^{\eta}-\pi^{\xi',\eta}_A\|_\tv)\leq \sum_{e\in E(A)} \pi_Q^{\xi'}\left(\nu_2^{\eta}(e\notin\omega)-\pi_Q^{\xi'}(e\notin\omega)\right)\,.
\]
For any $e\in E(A)$, consider $K_\ell=\{e+\llb -\ell,\ell \rrb^2\}\cap A$. Denote by $\nu_{2,\ell}^\eta$ the distribution obtained by the dynamics in $K_\ell$ with boundary conditions given by $(\xi',\eta)$ on $\partial K_\ell\cap \partial A$ and $0$ elsewhere. Then via a very rough mixing time estimate akin to Theorem~\ref{thm:canonical-paths} or Lemma~\ref{lem:gap-shorter-side}, there exists $c>0$ such that
\[\nu_2^{\eta}(e\notin\omega)-\pi_Q^{\xi'}(e\notin\omega)\leq e^{-te^{-c\ell}} + \left(\pi_{K_\ell}^{\xi',\eta}(e\notin\omega)-\pi_Q^{\xi'}(e\notin\omega)\right)\,.
\]
Absorbing a $2n^2$ for the maximum number of edges in $A$, it suffices to prove there exist constants $c,c'>0$ such that for every $e\in E(A)$,
\[\bE^{\Delta}\left[\pi_Q^{\xi'}(\pi_{K_\ell}^{\xi',\eta}(e\notin\omega)-\pi_A^{\xi',\eta}(e\notin\omega))\right]\lesssim e^{-c'\ell}+e^{-cn^{2\epsilon}}\,.
\]
For any fixed $e\in E(A)$ let $\Gamma^c:=\{e\stackrel{K_\ell}\longleftrightarrow \partial K_\ell\cap A^o\}$. By the FKG inequality,
\[\pi^{\xi',\eta}_A(e\notin\omega\mid \Gamma)\geq \pi_{K_\ell}^{\xi',\eta}(e\notin\omega)\,,\]
so it suffices to check that
\[\bE^{\Delta}\left[\pi_{Q}^{\xi'}(\pi_A^{\xi',\eta}(\Gamma^c))\right]=\bE^{\Delta}\left[\pi_{Q}^{\xi'}(\Gamma^c)\right]\lesssim e^{-c\ell}+e^{-cn^{2\epsilon}}\,.
\]

Proving this is very similar to proving the bound on the probability of $\Gamma_2^c$ in the proof of Lemma~\ref{lem:new-technique} as shown in Figure~\ref{fig:gamma-2}. For some $c>0$, up to an error of $e^{-cn}$, we replace $\bE^{\Delta}[\pi_Q^{\xi'}(\Gamma^c)]$ with $\pi^{0,1}_{E_n(Q)}(\Gamma^c)$ where $E_n(Q)$ is the enlarged rectangle defined in~\eqref{eq:enlargement}. Then by Proposition~\ref{prop:surface-tension}, for some other $c>0$ with probability $1-\exp(-cn^{2\epsilon})$, the bottom-most horizontal crossing stays below $\partial_\south A$ at which point it suffices to consider $\pi^0_{E_n(Q)}(\Gamma^c)$ because in that case, $\partial_\south (E_n(Q))$ would be completely disconnected from $K_\ell$. But Eq.~\eqref{eq:exp-decay-1} and monotonicity imply that there exists $c>0$ such that
\[\pi^{0}_{E_n(Q)}(\Gamma^c)\lesssim e^{-c\ell}\,,
\]
at which point a union bound over the two errors concludes the proof.
\end{proof}
Choosing $\ell =\lceil c^{-1}\log t\rceil$ in Lemma~\ref{lem:boxes} and union bounding over all $e\in E(A)$ yields that there exists a new $c>0$ such that for sufficiently large $n$,
\[\bE^{\Delta} \left[\pi^{\xi'}(\|\nu_2^{\eta}-\pi_A^{\xi',\eta}\|_\tv)\right]\lesssim t^{-c}\,.
\]
Combined with the bounds on the first and second terms of~\eqref{eq:first-free} and Theorem \ref{thm:censoring} we see that there exist $c_1,c_2>0$ such that for large enough $n$,
\[ \bE^{\Delta}\left[\|\tilde P^{2t}(0,\cdot)-\pi^{\xi'}_Q\|_\tv\right]\leq 2\delta+c_1e^{-c_2n^{2\epsilon}}+2n^2t^{-c_2}\,,
\]
which combined with part (i) of the proof, allows us to conclude the proof of~\eqref{eq:recursion-1}.
\end{proof}

We now prove the second implication to complete the proof of Proposition~\ref{prop:recursion}.

\begin{proof}  [\textbf{\emph{Proof of Eq.~\eqref{eq:recursion-2}}}] Fix any $\bP\in \mathcal D(\Lambda_{2n,m})$ and observe that the proof is independent of this choice of $\bP$. Consider the quantity, $\bE [\| P^t(\omega_0,\cdot)-\pi_{\Lambda_{2n,m}}^\xi \|_\tv]$ for $\omega_0=0,1$.

\emph{(i) Wired initial configuration.} Divide $Q:=\Lambda_{2n,m}$ into,
\begin{align*}
A= & \Lambda_{n,m}+(\lfloor n/2\rfloor,0)\,, \qquad
B= \Lambda_{n,m}\cup \{\Lambda_{n,m}+(n,0)\}\,, \qquad
C= \{n\}\times \llb 0,m\rrb\,.
\end{align*}
Let $B_\east,B_\west$ be the two connected components of $B$ so that the dynamics on $B$ does not update any of the edges of $C\subset \partial B$. Let $\Delta=\Delta_\south \cup \Delta_\north \cup \Delta_\east \cup \Delta_\west$ where
\begin{align*}
\Delta_\south=\{(x,y)\in \partial_\south A: |x-n|\leq n^{3\epsilon}\}\,, & \qquad \Delta_\north=\{(x,y)\in \partial_\north A: |x-n|\leq n^{3\epsilon}\}\,, \\
 \Delta_\east=\{(x,y)\in \partial_\north A:x\leq  n/2  +n^{3\epsilon}\}\,, & \qquad \Delta_\west=\{(x,y)\in \partial_\north A:x\geq  3n/2 -n^{3\epsilon} \}\,.
\end{align*}

We first observe that, as before, by Corollary~\ref{cor:bc-perturbation} and the size of $|V(\Delta)|$, it suffices up to new choice of constants in~\eqref{eq:recursion-2}, to prove the implication under $\bP^{\Delta}$ as given by Definition~\ref{def:modification}. By monotonicity, Theorem~\ref{thm:censoring}, and Corollary~\ref{cor:bc-perturbation}, up to another change of constants $c_2,c_3$, it suffices to prove,
\[\max_{\omega_0\in \{0,1\}}\bE^{\Delta} \left[\|\tilde P^{2t'}(\omega_0,\cdot)-\pi_{Q}^{\xi'}\|_\tv\right]\leq c\delta'
\]
with $t'=\exp(cn^{3\epsilon}) t$ and $\delta'=8\delta+\exp(-e^{c'n^{3\epsilon}})$
for $c,c'>0$ given by Corollary~\ref{cor:bc-perturbation} and $\tilde P$ a censored dynamics. We begin with the situation in which $\omega_0=1$ and let $\tilde P$ be the transition kernel of the following censored dynamics: for the first time interval $[0,t')$, only accept updates from $A$ then at time $t'$ change all edges interior to $B$ to $1$ and only updates edges interior to $B$ until time $2t'$. As before, let $\nu_1$ denote the distribution after time $t'$ on $A$ and let $\nu_2^{\eta}$ denote the distribution after time $2t'$ of configurations on $B$ given that at time $t'$ all edges in $B$ are reset to $1$ and the configuration on $C$ was $\eta$.

The triangle inequality and Markov property together imply that,
\begin{align}
\bE^{\Delta}\left[\|\tilde P^{2t'}(1,\cdot)-\pi_Q^{\xi'}\|_\tv\right]\leq ~& \bE^{\Delta}\left[\|\nu_1-\pi_A^{\xi',1}\|_C\right] +\bE^{\Delta}\left[\|\pi_A^{\xi',1}-\pi_{Q}^{\xi'}\|_C\right] \nonumber \\
 + &\bE^{\Delta}\left[\|\pi_{Q}^{\xi'}-\pi_Q^{\xi',0}\|_C\right] +\bE^{\Delta}\left[\pi_{Q}^{\xi',0}(\|\nu_2^{\eta}-\pi_B^{\xi',\eta}\|_\tv)\right]\,. \label{eq:second-wired}
\end{align}

Here, boundary conditions of the form $(\xi',\eta)$ denote a boundary condition that agrees with $\xi'$ on the intersection of the boundary of the domain and $\partial Q$, and takes on boundary conditions $\eta$ elsewhere.
As in part (i) of the proof of~\eqref{eq:recursion-1}, we observe that because of the modification on $\Delta_{\east}\cup \Delta_\west$ with $\bP$-probability $1-\exp(-cn^{3\epsilon})$, the boundary conditions on $\partial_\north A$ are dominated by $\pi_{\mathbb Z^2}^0$ in spite of the wired initial configurations on $Q-A$. In what follows, we assume---paying an error of $\exp(-cn^{3\epsilon})$---that this decreasing event holds. Observe also that the wired initial configuration can only make $\partial_\south A$ more wired, and thus those boundary conditions will continue to dominate the marginal of $\pi^1_{\mathbb Z^2}$.
Along with self-duality and Corollary~\ref{cor:bc-perturbation}, this implies that the first term in~\eqref{eq:second-wired} is bounded above by $\delta'$.

The fourth term can be bounded as follows: observe that the distribution over boundary conditions $(\xi',\eta)$ on $B$ under $\bP^{\Delta}(\xi')\pi_{Q}^{\xi',0}(\eta)$ coincides with the $\Delta_\south$ modification of $\bP^{\Delta_{\north,\east,\west}}(\xi')\pi_{Q}^{\xi',0}(\eta)$. Because the boundary conditions on $\partial_\north A$ are dominated by $\pi_{\mathbb Z^2}^0$, an argument like~\eqref{eq:avg-avg} implies $\bP^{\Delta_{\north,\east,\west}}(\xi')\pi_Q^{\xi',0}(\eta)\preceq \pi_{\mathbb Z^2}^0(\eta)$ on $C$.  Thus, by Corollary~\ref{cor:bc-perturbation}, the  third term is bounded above by $2\delta'$, where the factor of $2$ comes from the fact that $B$ consists of two independent copies of $R_n$. We used the fact that the configuration on $E(C)$ is dominated by the partition of $C$ induced by the FK configuration under $\pi_Q^{\xi',0}$.

It remains to bound the second and third terms of~\eqref{eq:second-wired}, which can be treated similarly so we only go through the bound of the former.
\begin{lemma}There exists  $c=c(q)>0$ such that
\[\bE^{\Delta}\left[\|\pi_A^{\xi',1}-\pi_Q^{\xi'}\|_C\right]\lesssim e^{-cn^{3\epsilon}}\,.
\]
\end{lemma}
\begin{proof} We bound the total variation distance on $C$ by bounding the probability that samples from the two distributions agree under the grand coupling.
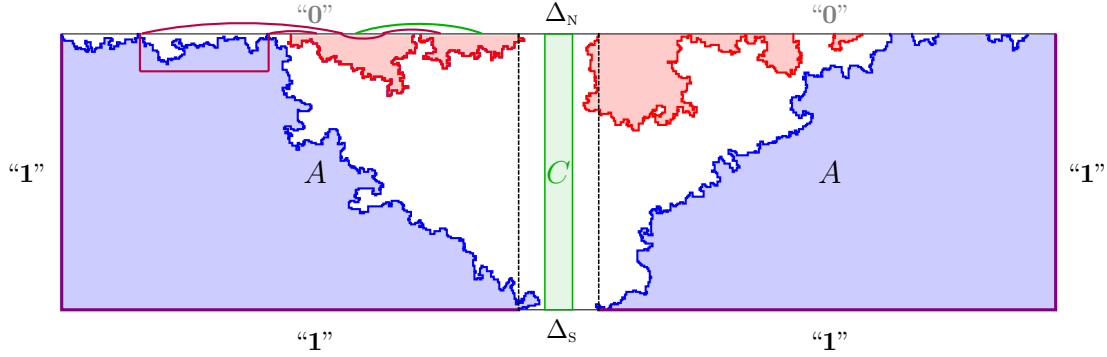
\begin{figure}
  \hspace{-0.15in}

  \begin{tikzpicture}
    \node (plot) at (0,0)
    {\hspace{-0.14in}\vspace{-0.3in}\includegraphics[width=0.95\textwidth]{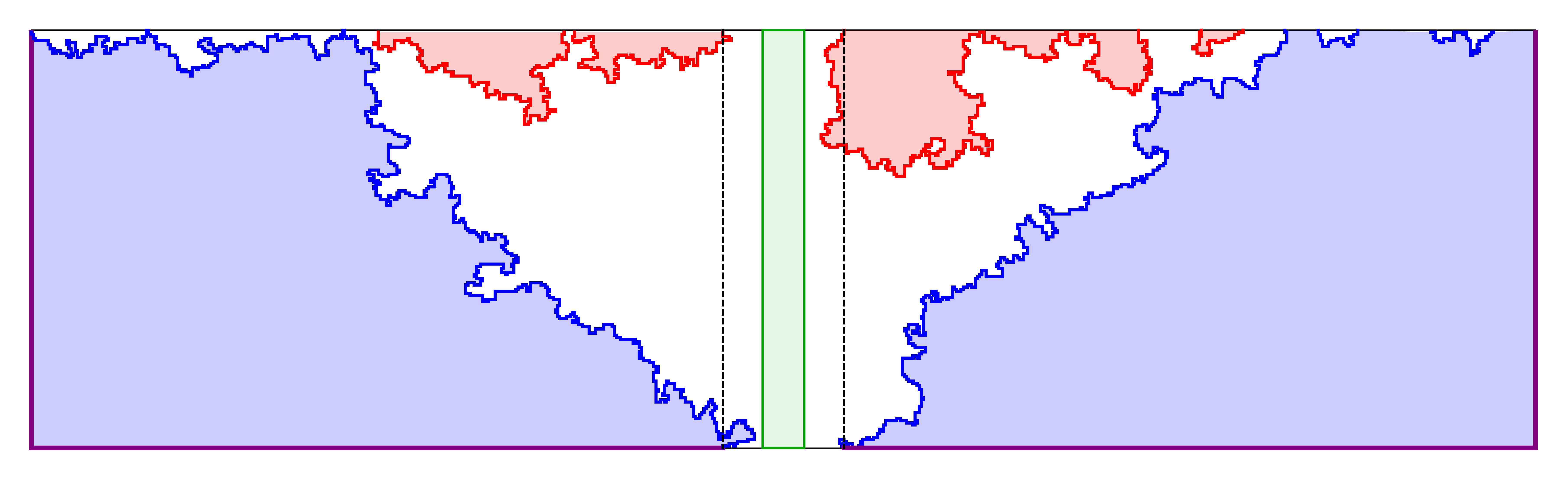}};

    \begin{scope}[shift={(plot.south west)},x={(plot.south
        east)},y={(plot.north west)}, font=\small]


      \draw[color=purple, thick] (0.277,.9121) arc (55:125:.17) ;
            \draw[color=purple, thick] (0.252,0.9121) arc (55:125:.04) ;
      \draw[color=purple, thick] (0.372,0.9121) arc (45: 135:0.04) ;
      \draw[color=DarkGreen, thick] (0.412,0.9121) arc (40: 140:0.08) ;
      \draw[color=purple, thick] (0.275,0.912) arc (200:340:0.023) ;

      \draw[color=purple, thick] (0.08,0.913) -- (.08,0.8) ;
      \draw[color=purple, thick] (0.205,0.913) -- (.205,0.8) ;
      \draw[color=purple, thick] (0.08,0.8) -- (.205,0.8) ;


      \node[color=black] at (0.25,0) {``$\boldsymbol 1$"};
      \node[color=black] at (0.75,0) {``$\boldsymbol 1$"};
      \node[color=black] at (-.03,.5) {``$\boldsymbol 1$"};
      \node[color=black] at (1.0,.5) {``$\boldsymbol 1$"};
      \node[color=black] at (0.49,0.03) {$\Delta_\south$} ;
      \node[color=black] at (0.49,.97) {$\Delta_\north$} ;
      \node[color=DarkGreen][font=\large] at (0.4865,0.5) {$C$} ;
      \node[font=\large] at (0.25,0.5) {$A$} ;
      \node[font=\large] at (0.75,0.5) {$A$} ;
      \node[color=gray] at (0.25,.97) {``$\boldsymbol 0$"} ;
      \node[color=gray] at (0.75,.97) {``$\boldsymbol 0$"} ;


    \end{scope}
  \end{tikzpicture}
  \caption{The modification analogous to Figure 5 for the second step of the recursion (Eq.~\eqref{eq:recursion-2}).}
  \label{fig:recursion-2nd-step}
\end{figure}

Following the proof of Lemma~\ref{lem:new-technique}, we first define the event $\Gamma_1$ as the set of $\xi'$ in which there exists no pair $(x,y)\in (\partial_{\north} B_\east-\Delta_\north)\times (\partial_\north B_\west -\Delta_\north)$ such that $x\stackrel{\xi'}\longleftrightarrow y$ where $\stackrel{\xi'}\longleftrightarrow$ denotes that $x,y$ are in the same boundary component. We split up $\bE^{\Delta}[\|\pi_A^{\xi',1}-\pi_Q^{\xi'}\|_C]$ into an average over those $\xi'\in\Gamma_1$ and those in $\Gamma_1^c$.

As in Lemma~\ref{lem:new-technique}, we obtain the bound $\bE^{\Delta}[\Gamma_1^c]\leq \exp(-cn^{3\epsilon})$ for some $c>0$ using the fact that the boundary conditions on $\partial_\north A$ are obtained from a measure that is dominated by $\pi_{\mathbb Z^2}^0$ and Eq.~\eqref{eq:exp-decay-1} implies an exponential decay of connections. For all such $\xi'$, we bound the distance between the two measures by $1$.

Now consider, for any $e\in E(C)$, $\bE^{\Delta}[\|\pi_A^{\xi',1}-\pi_Q^{\xi'}\|_\tv \mid \Gamma_1]$. Define in analogy with the proof of Lemma~\ref{lem:new-technique}, the decreasing events $\Gamma_2$  and $\Gamma_3$:
\[\Gamma_2:=\bigcap_{i\in\{\east,\west\}} \cC_v^\ast (A\cap B_i)\,,
\]
and $\Gamma_3$ is the event that there does not exist any $x\in \partial_\north A\cap B_\east-\Delta_\north$ such that $x\longleftrightarrow \partial_\west B_\east$ and likewise, there does not exist any $x\in \partial_\north A \cap B_\west-\Delta_\north$ such that $x\longleftrightarrow \partial_\east B_\west$. Under $\Gamma_2 \cap \Gamma_3$, we could expose all the wired components of $\partial A-\Delta_\north-\Delta_\south$ to reveal an outermost dual circuit around $C$. By monotonicity and domain Markov, the two distributions would be coupled under the grand coupling past that dual circuit, so that for all $\xi'\in\Gamma_1$,
\[\|\pi^{\xi',1}_A-\pi_Q^{\xi'}\|_C\leq \pi_A^{\xi',1}(\Gamma_2\cap\Gamma_3)\,.
\]

Let $\bar A=\llb 0,n\rrb \times \llb 2m\rrb$, viewed as two copies of $A$ stacked above one another. Let $(0,1,\Delta_\south)$ boundary conditions on $\bar A$ denote those that are free on $\partial_\north \bar A$ and $\Delta_\south$ and wired on the rest of $\partial \bar A$. By monotonicity and the exponential decay of correlations in the free phase given by~\eqref{eq:exp-decay-1}, we see that there exists $c>0$ such that
\[\bE^{\Delta}\left[\pi_A^{\xi',1}(\Gamma_2^c)\given \Gamma_1\right]\leq \bE^{\Delta_\south}\left[\pi_A^{\xi',1}(\Gamma_2^c)\given \Gamma_1\right]\leq e^{-cm} +\pi_{\bar A}^{0,1,\Delta_\south}(\Gamma_2^c)\,.
\]

Consider $\pi_{\bar A}^{0,1,\Delta_\south}(\Gamma_2^c)$ and let $V=\llb 0,n\rrb \times \llb 0,4m\rrb$, as in Proposition~\ref{prop:D-estimate} with $\ell=4m$. Let $(0,1,\Delta_\south)$ boundary conditions on $V$ denote those that are free above $y=2m$ and on $\Delta_\south$ and wired elsewhere. Then it is clear by monotonicity and the fact that $\Gamma_2^c$ is an increasing event, that
\[\pi_{\bar A}^{0,1,\Delta_\south}(\Gamma_2^c)\leq \pi_V^{0,1,\Delta_\south}(\Gamma_2^c)\,.
\]
Because $m\asymp n^{\frac 12+\epsilon}$, we can apply Proposition~\ref{prop:D-estimate} to see that there exists a new $c(q)>0$ such that $\pi_{\bar A}^{0,1,\Delta_\south}(\Gamma_2^c)\lesssim \exp(-cn^{3\epsilon})$.

We now turn to bounding $\pi_A^{\xi',1}(\Gamma_3^c\mid \Gamma_2)$ using the same approach as in the proof of Claim~\ref{lem:new-bound}. Under $\Gamma_2$, there is a pair of vertical dual crossings in $A$ which allow us to replace, by monotonicity, the boundary conditions $(\xi',1)$ by ones that we denote $(0,1,\Delta_\north)$ which are free on $\partial_{\east,\west,\south} A$ and also free on $\Delta_\north$ and wired elsewhere.

To make the comparison to the setting of Proposition~\ref{prop:D-estimate}, perturb the boundary conditions more by extending the wired segments down along $\partial_{\east,\west} A$ a length $n^{1/2+\epsilon}$. Up to a $\pi$-rotation, Proposition~\ref{prop:D-estimate} with the choice of $\ell=m$ implies that the two interfaces are confined to the left and right halves of $A$ with probability $1-\exp(-cn^{3\epsilon})$ for some new $c>0$, and sufficiently large $n$. Monotonicity implies that for any $\xi'\in \Gamma_1$,
\[\pi_A^{\xi',1} (\Gamma_3^c \mid \Gamma_2) \lesssim e^{-cn^{3\epsilon}}\,,
\]
and together with a union bound, there exists $c'>0$ such that for large enough $n$,
\[
\bE^{\Delta}\left[\|\pi^{\xi',1}_A-\pi_Q^{\xi'}\|_C\right]\leq e^{-c'n^{\frac 12+\epsilon}}+2e^{-cn^{3\epsilon}}\,. \qedhere
\]\end{proof}

Combined with the prior bounds on the first and third terms in the right-hand side of~\eqref{eq:second-wired} and Theorem \ref{thm:censoring}, for some $c>0$,
\[ \bE^{\Delta}\left[\|\tilde P^{2t'}(1,\cdot)-\pi_Q^{\xi'}\|_\tv\right]\lesssim \delta'+e^{-cn^{3\epsilon}}\,.
\]

\emph{(ii) Free initial configuration.} The bound for the free initial configuration,
\[\bE^{\Delta}\left[\|P^{2t'}(0,\cdot)-\pi_Q^{\xi'}\|_\tv\right]\lesssim \delta'+e^{-cn^{3\epsilon}}\,,
\]
is nearly identical to the above bound with the following change: the censored dynamics $\tilde P$ only allows updates in block $A$ until time $t'$ then resets all edge values in $B$ to $0$ then only allows updates in $B$ between time $t'$ and $2t'$. In this case, we can again write, using the same notation as before,
\begin{align}
\bE^{\Delta}\left[\|\tilde P^{2t'}(0,\cdot)-\pi_Q^{\xi'}\|_\tv\right]\leq ~& \bE^{\Delta}\left[\|\nu_1-\pi_A^{\xi',0}\|_C\right]+\bE^{\Delta}\left[\|\pi_Q^{\xi'}-\pi_Q^{\xi',0}\|_C\right] \nonumber \\
 + &\bE^{\Delta}\left[\|\pi_A^{\xi',0}-\pi_Q^{\xi'}\|_C\right]+\bE^{\Delta}\left[\pi_Q^{\xi',0}(\|\nu_2^{\eta}-\pi_B^{\xi',\eta}\|_\tv)\right]\,. \label{eq:second-free}
\end{align}
The free initial configuration precludes the long-range interactions of the FK model modifying the boundary conditions on $A$ and thus they are still in $\mathcal D(A)$. The bound on the first term then follows from the assumption, without appealing to self-duality, and the other three bounds hold as for $\omega_0=1$. Combining the above with part (i) holds, we see that \eqref{eq:recursion-2}, which with~\eqref{eq:recursion-1}, concludes the proof of Proposition~\ref{prop:recursion}.
\end{proof}

\subsection{Proof of Proposition \ref{prop:all-free}}\label{sub:proof-large-q-upper-1}

It now remains to extend Corollary~\ref{cor:3-sides} to boundary conditions that are dominated by the free phase or dominate the wired phase on all four sides, to obtain the desired bounds for free and monochromatic boundary conditions.
Suppose without loss of generality that $\bP\preceq \pi_{\mathbb Z^2}^0$ on $\Lambda$; the case $\bP\succeq \pi_{\mathbb Z^2}^1$ follows from self-duality. First consider the dynamics started from $\omega_0=1$. We wish to prove that there exists $c>0$ such that for $t_\star=\exp(cn^{3\epsilon})$,
\[\bE\left[\|P^{t_\star}(1,\cdot)-\pi^{\xi'}_{\Lambda}\|_{\tv}\right]\lesssim e^{-cn^{2\epsilon}}\,.
\] By an application of Corollary~\ref{cor:bc-perturbation}, up to errors that can be absorbed by adjusting the constant $c$ appropriately, we can modify, as in Definition~\ref{def:modification}, the boundary conditions on $\Delta=\Delta_\north \cup \Delta_\south$, where $\Delta_\north=\Delta_\north^{1}\cup \Delta_\north ^{2}$,
\begin{align*}
\Delta_\north^1= & \{0,n\}\times \llb \tfrac{3n}4-n^{3\epsilon} ,\tfrac{3n}4 \rrb\,, \\
\Delta_\north ^2= & \{0,n\}\times \llb \tfrac n2,\tfrac n2+n^{3\epsilon}\rrb\,,
\end{align*}
and $\Delta_\south$ is the reflection of $\Delta_\north$ across the line $y=n/2$, and consider dynamics on $\Lambda$ with the new measure $\bP^{\Delta}$ on boundary conditions. Let $\Lambda^\pm$ denote the top and bottom halves of $\Lambda$, respectively. Then,
\begin{align*}
\bE^{\Delta}\left[\| P^{t_\star}(1,\cdot) -\pi_\Lambda^{\xi'}\|_\tv\right]\leq  \bE^{\Delta}\left[\|P^{t_\star}(1,\cdot)-\pi^{\xi'}_\Lambda\|_{\Lambda^-}\right]
 +  \bE^{\Delta}\left[\|P^{t_\star}(1,\cdot)-\pi_\Lambda^{\xi'}\|_{\Lambda^+}\right]\,.
\end{align*}
We deal only with the first term since the second can be bounded analogously. Let $\tilde P$ be the dynamics that censors all updates not in $\Lambda_{n,3n/4}$.
Then by Markov property and triangle inequality, the first term can be bounded above by
\[\bE^{\Delta}\left[\|\tilde P^{t_\star}(1,\cdot)-\tilde\pi^{\xi',1}\|_{\Lambda^-}\right]+\bE^{\Delta}\left[\|\pi_\Lambda^{\xi'}-\tilde\pi^{\xi',1}\|_{\Lambda^-}\right]\,,
\]
where $\tilde\pi^{\xi',1}$ denotes the stationary distribution on $\Lambda_{n,3n/4}$ with boundary conditions that are wired on $\partial_\north \Lambda_{n,3n/4}$ and $\xi'$ elsewhere.
First observe that because of $\Delta_\north^1$, with probability $1-\exp(-cn^{3\epsilon})$, the wired initial configuration on $\Lambda-\Lambda_{n,3n/4}$ does not affect the boundary conditions on $\Lambda_{n,3n/4}$ and therefore the boundary conditions on it are, up to a $\pi$-rotation, in $\mathcal D(\Lambda_{n,3n/4})$. At this cost, we assume this decreasing event holds.

The second term can then be bounded as in Lemma~\ref{lem:new-technique} by $\exp(-cn^{2\epsilon})$ for some $c>0$. We use $\Delta^2_\north$ to disconnect the wired boundary condition on $\partial_\north \Lambda_{n,3n/4}$ from $\Lambda^-$; the new choice of $\Delta$ and modifications in the sizes of the boxes do not affect the proofs.

The first term can be bounded via Corollary~\ref{cor:3-sides} for $m=3n/4$ since, with probability $1-\exp(-cn^{3\epsilon})$, the boundary condition on $\partial_{\east,\south,\west} \Lambda_{n,3n/4}$ is dominated by the marginal of $\pi_{\mathbb Z^2}^{0}$ while on $\partial_\north \Lambda_{n,3n/4}$, it is all wired. Combining the two bounds and doing the same for $\Lambda^+$, implies there is a $c>0$ such that for $t_\star=\exp(cn^{3\epsilon})$,
\[\bE^{\Delta} \left[\| P^{t_\star}(1,\cdot)-\pi_\Lambda^{\xi'}\|_\tv\right]\lesssim e^{-cn^{2\epsilon}}\,.
\]

For the dynamics started from the free initial configuration, for every $e\in E(\Lambda)$, $\ell\in\mathbb N$, define $K_\ell=\Lambda\cap \{e+\llb -\ell,\ell\rrb ^2\}$. Let $P_{K_\ell}$ denote the transition kernel of the dynamics restricted to $K_\ell$ with $(\xi,0)$ boundary conditions denoting $\xi$ on $\partial K_\ell \cap \partial \Lambda$ and free elsewhere. We claim that, for some $c>0$,
\begin{align*}\bE\left[\| P^{t_\star}(0,\cdot)-\pi_\Lambda^{\xi}\|_\tv\right]\leq & \sum_{e\in E(\Lambda)} \bE \left[ P^{t_\star}(1,e\notin\omega)-\pi_\Lambda^{\xi}(e\notin\omega)\right] \\
\leq & \sum_{e\in E(\Lambda)} \bE\left[\| P^{t_\star}_{K_\ell}(0,\cdot)-\pi^{\xi,0}_{K_\ell}\|_\tv\right]+e^{-c\ell}+e^{-cn}\,,
\end{align*}
for some $c>0$. The first inequality is an immediate consequence of monotonicity. The second follows from an argument similar to that in the proof of Lemma~\ref{lem:boxes} where now we take a new enlargement, $E_n'(\Lambda)$ that enlarges $\Lambda$ by $n$ in the southern direction also. We can replace, by Eq.~\eqref{eq:exp-decay-1}, $\pi_\Lambda^{\xi}$ by $\pi_{E'_n(\Lambda)}^0$ up to an error of $e^{-cn}$ as before. Again using~\eqref{eq:exp-decay-1} in the free phase, up to an error of $e^{-c\ell}$, we can replace $\pi^{0}_{E'_n(\Lambda)}$ with $\pi_{K_\ell}^{\xi,0}$, noting that the distributions at $e$ match if $e$ is disconnected from $K_\ell$ by a dual circuit.

We can then bound the sum in the right-hand side by Proposition~\ref{prop:base-scale}: uniformly in $\xi$,$e$, $\tmix$ for $K_\ell$ is bounded above by $e^{c\ell}$ for some $c>0$. Choosing $\ell=\lceil c^{-1}\log t_\star\rceil\asymp n^{3\epsilon}$ and union bounding over all the errors yields, for some $c>0$,
\[\bE\left[\| P^{t_\star}(0,\cdot)-\pi_\Lambda^{\xi}\|_\tv\right]\lesssim e^{-cn^{2\epsilon}}\,,
\]
as desired. An application of Markov's inequality, the triangle inequality, and~\eqref{eq:init-config-comparison} to
\[\max_{\omega_0\in\{0,1\}}\bE\left[\| P^{t_\star}(\omega_0,\cdot)-\pi_\Lambda^{\xi}\|_\tv\right]\lesssim e^{-cn^{2\epsilon}}\,,
\]
implies that there exists some $c>0$ such that for $t_\star=e^{cn^{3\epsilon}}$ and $n$ sufficiently large, $\bP(\tmix \geq t_\star)\lesssim \exp({-cn^{2\epsilon}})$
as required. \qed

\subsection*{Acknowledgment} The authors wish to thank C.\ Hongler, F.\ Martinelli, Y.\ Peres and S.\ Shlosman for useful discussions, as well as A.\ Sly, whose joint paper with E.L.\ on the critical Ising model was the starting point of this project.
The research of R.G.\ was supported in part by NSF grant DMS-1507019. The research of E.L.\ was supported in part by NSF grant DMS-1513403.

\bibliographystyle{abbrv}
\bibliography{critical_potts}

\end{document}